\documentclass[reqno, huge]{amsart}
\usepackage{latexsym}
\usepackage{amsmath}
\usepackage{amssymb}
\usepackage{amsthm}
\usepackage[matrix,arrow,curve]{xy}
\usepackage{tabularx,array}
\usepackage{euscript}
\usepackage{mathtools}
\usepackage{bm}
\usepackage{CJK}
\usepackage{amscd}
\usepackage{graphicx}
\usepackage{xcolor}
\usepackage[colorlinks,citecolor=blue]{hyperref}
\usepackage{tikz}
\usepackage{tikz-cd}
\usepackage[a4paper,top=2cm,bottom=2cm,left=2cm,right=2cm]{geometry}
\usepackage{cite}

\numberwithin{equation}{section}
\theoremstyle{plain}

\newtheorem{theorem}{Theorem}[section]
\newtheorem{lemma}[theorem]{Lemma}
\newtheorem{predl}[theorem]{Proposition}
\newtheorem{corollary}[theorem]{Corollary}

\newtheorem{conjecture}[theorem]{Conjecture}
\theoremstyle{definition}
\newtheorem{definition}[theorem]{Definition}
\newtheorem{remark}[theorem]{Remark}

\newcommand{\Ku}{\mathcal{K}u}

\newcommand{\bL}{\bm{\mathrm{L}}}
\newcommand{\cQ}{\mathcal{Q}}
\newcommand{\cE}{\mathcal{E}}
\newcommand{\cA}{\mathcal{A}}
\newcommand{\C}{\mathbb C}
\newcommand{\R}{\mathbb R}

\newcommand{\id}{\mathrm{id}}

\DeclareTextFontCommand{\emph}{\em}
\DeclareMathOperator{\Hom}{\textup{Hom}}

\begin{document}

\title{Bridgeland Moduli spaces for Gushel-Mukai threefolds and Kuznetsov's Fano threefold conjecture}
\thanks{The author is supported by ERC Consolidator grant WallCrossAG, no. 819864. 
}

\address{School of Mathematics, The University of Edinburgh, JCMB Building, Kings Building, Edinburgh, EH9 3FD}
\email{Shizhuo.Zhang@ed.ac.uk}

\author{shizhuo ZHANG}
\maketitle

\begin{abstract}
We study the Hilbert scheme $\mathcal{H}$ of twisted cubics on a special smooth Gushel-Mukai threefolds $X_{10}$. We show that it is a smooth irreducible projective threefold if $X_{10}$ is general among special Gushel-Mukai threefolds, while it is singular if $X_{10}$ is not general. We construct an irreducible component of a moduli space of Bridgeland stable objects in the Kuznetsov component of $X_{10}$ as a divisorial contraction of $\mathcal{H}$. We also identify the minimal model of Fano surface $\mathcal{C}_m(X_{10}')$ of conics on a smooth ordinary Gushel-Mukai threefold with the moduli space of Bridgeland stable objects in the Kuznetsov component of $X_{10}'$. As a result, we show that Kuznetsov components of quartic double solids and that of Gushel-Mukai threefolds are not equivalent.  
\end{abstract}



\section{Introduction}
\subsection{Background}
The study of bounded derived category $D^b(X)$ of a smooth projective variety $X$ has become one of the central topics in algebraic geometry. For many Fano varieties, it has become widely accepted that a certain subcategory $\mathcal{K}u(X)\subset D^b(X)$ called Kuznetsov component encodes much geometric information in a most efficient way. In the case of Fano threefolds of Picard rank one, index two,  one is able to recover $X$ from its Kuznetsov components, known as categorical Torelli theorems, see \cite{BMMS}, \cite{pertusi2020some}, \cite{altavilla2019moduli},\cite{bernardara2013semi},\cite{li2018twisted},\cite{huybrechts2016hochschild} and \cite{bayer2017stability}. The semi-orthogonal decomposition of an index two prime Fano threefold $Y_d$ of degree $d$ is given by 
\[\mathrm{D^b}(Y_d)=\langle\mathcal{K}u(Y_d), \mathcal{O}_{Y_d}, \mathcal{O}_{Y_d}(1)\rangle\]
The Kuznetsov component $\mathcal{K}u(Y_d)$ is then defined as $\mathcal{O}_{Y_d}(H)^{\perp}\cap\mathcal{O}_{Y_d}^{\perp}$. The semi-orthogonal decomposition of the correspondent index one prime Fano threefold $X_{2g-2}$ of genus $g\geq 6$ is given in \cite{kuznetsov2003derived} as
\[\mathrm{D^b}(X_{2g-2})=\langle\mathcal{K}u(X_{2g-2}), \mathcal{E}, \mathcal{O}_{X_{2g-2}}\rangle\]
where $\mathcal{E}$ is a rank $2$ stable vector bundle with ch$(\mathcal{E})=2-H+\frac{g-4}{2}L+\frac{10-g}{12}P$. The Kuznetsov component $\mathcal{K}u(X_{2g-2})$ is then defined as $\mathcal{O}_{X_{2g-2}}^{\perp}\cap \mathcal{E}^{\perp}$. In the paper \cite{kuznetsov2003derived}, Kuznetsov proposes a surprising relationship between Kuznetsov components of prime Fano threefolds of index one and index two. More precisely, he made the following conjecture on their relationship:

\begin{conjecture}
\label{main_conjecture}
Let $\mathcal{MF}^i_d$ be the moduli spaces of prime Fano threefold of index $i$ and degree $d$. Then there is a correspondence $Z_d\subset\mathcal{MF}^2_d\times\mathcal{MF}^1_{4d+2}$ which is dominant over each factor and such that for any point $(Y_d, X_{4d+2})\in Z_d$ there is an equivalence of categories $$\mathcal{K}u(X_{4d+2})\cong\mathcal{K}u(Y_d).$$ 
\end{conjecture}

Kuznetsov shows Conjecture~\ref{main_conjecture} holds for the pair of Fano threefolds $(Y_d, X_{4d+2})$ for $d=3,4,5$ and uses such equivalences to prove isomorphisms between Fano surface of lines $\Sigma(Y_d)$ and Fano surface of conics $\mathcal{C}(X_{4d+2})$ in these cases \cite{KPS2018}. For $d=2$, $Y_2$ is a quartic double solid, constructed as a double cover of $\mathbb{P}^3$ with a quartic K3 surface as its branch locus and correspondent index one prime Fano threefold $X_{10}$ is called Gushel-Mukai threefold. It is either a quadric section of a linear section of $\mathrm{Gr}(2,5)$ or a double cover of $Y_5$ with a degree ten K3 surface as branch locus. In \cite{bernardara2013semi}, the authors show that a correspondence $Z_2\subset\mathcal{MF}^2_2\times\mathcal{MF}^1_{10}$ such that for any point $(Y_2, X_{10})\in Z_2$, $\mathcal{K}u(Y_2)$ is equivalent to $\mathcal{K}u(X_{10})$ by a Fourier-Mukai functor can never be dominant over the second factor. This suggested that the assumption on the dominance of $Z_2$ should be removed. In our article, we show that even if the assumption is removed, the Conjecture~\ref{main_conjecture} can not be true.

\begin{theorem}
\label{theorem_SOGM_Ku_false}
Let $Y$ be a quartic double solid and $X$ be a Gushel-Mukai threefold, then $\mathcal{K}u(Y)\not\simeq\mathcal{K}u(X)$. 
\end{theorem}

The difficulty to study  Conjecture~\ref{main_conjecture} for the case $d=2$ is that we are unaware any nice description of the Kuznetsov categories, where techniques of Bridgeland stability conditions enters into the picture. In our paper, we study Bridgeland moduli spaces $\mathcal{M}_{\sigma'}(\mathcal{K}u(X),s)$ of $\sigma$-stable objects of a appropriately chosen character $s$ in the Kuznetsov component for both ordinary and special Gushel-Mukai threefolds $X$, where $\sigma$ is a Serre-invariant stability conditon(cf.Definition~\ref{definition_S_invariant_stab_condition}), while the corresponding Bridgeland moduli spaces $\mathcal{M}_{\sigma}(\mathcal{K}u(Y),w=H-\frac{1}{2}H^2-\frac{1}{3}H^3)$ and $\mathcal{M}_{\sigma}(\mathcal{K}u(Y),v=1-L)$ for quartic double solid $Y$ have been studied in \cite{pertusi2020some} and \cite{altavilla2019moduli}. Then the conjectural equivalence $\Phi:\mathcal{K}u(Y)\simeq\mathcal{K}u(X)$ would identify the tangent spaces of these moduli spaces at each point and we show that this is not the case, which means the equivalence can't exist. 

\subsection{Bridgeland moduli spaces for Gushel-Mukai threefolds}
Let $X$ be a special Gushel-Mukai threefold with the double covering map $\pi:X\rightarrow Y$ to a degree 5 del Pezzo threefold $Y$. We study the Hilbert scheme $\mathcal{H}$ of twisted cubics and construct the Bridgeland moduli space $\mathcal{M}_{\sigma}(\mathcal{K}u(X),s)$ of $\sigma$-stable object of character $s=1-3L+\frac{1}{2}P$ in $\mathcal{K}u(X)$ from $\mathcal{H}$. The first result is

\begin{theorem}
\label{theorem_main1}
.\leavevmode \begin{enumerate}
\item  If $X$ is general among special Gushel-Mukai threefolds, where the branch locus $\mathcal{B}$ on $Y$ does not contain any line or conic. Then $\mathcal{H}$ is a smooth irreducible projective threefold.
\item  If $X$ is not general, where $\mathcal{B}$ contains line or conic. Then $\mathcal{H}$ is a singular threefold. Its singular locus consists of finitely many isolated points given by $\pi$-preimage of twisted cubics contained in $\mathcal{B}$ and finitely many rational curves such that each of them consists of reducible twisted cubics with a common line component $L$ whose normal bundle $\mathcal{N}_{L|X}\cong\mathcal{O}_L(-2)\oplus\mathcal{O}_L(1)$. 
\end{enumerate}
\end{theorem}

The Hilbert scheme $\mathcal{H}$ has a stratification:  $\mathcal{H}=\mathcal{H}_1\cup\mathcal{H}_2$, and $\mathcal{H}_1$ parametrises twisted cubics whose image under $\pi$ are twisted cubics on $Y$, while $\mathcal{H}_2$ parametrises twisted cubics whose image are conics, and $\mathcal{H}_2$ is a ruled surface over the Hilbert scheme $\Sigma(X)$ of lines.  Let $\mathrm{pr}:D^b(X)\rightarrow\mathcal{K}u(X)$ be the projection functor, it is defined as $\bL_{\mathcal{E}}\bL_{\mathcal{O}_X}$. We project the ideal sheaf $I_C$ of a twisted cubic $C\subset X$ onto $\mathcal{K}u(X)$ to construct an irreducible component of $\mathcal{M}_{\sigma}(\mathcal{K}u(X),s)$.

\begin{theorem}
\label{theorem_main2}Let $X$ be a special Gushel-Mukai threefold. 
\leavevmode \begin{enumerate}
    \item If $X$ is general among the special Gushel-Mukai threefolds, then there exists an irreducible component $\mathcal{Z}\subset \mathcal{M}_{\sigma}(\mathcal{K}u(X),s)$ such that $\mathcal{Z}$ is smooth projective of dimension three. The fixed locus $\mathrm{Fix}_{\tau}(\mathcal{Z})$ is empty. Moreover, $\mathcal{Z}$ is a divisorial contraction of $\mathcal{H}$. The map $p:\mathcal{H}\rightarrow\mathcal\mathcal{Z}$ induced by the projection functor onto $\mathcal{K}u(X)$ contracts a smooth ruled surface $\mathcal{H}_2$ to a smooth irreducible curve $C\cong\Sigma(X)$ of genus $71$. 
    \item If $X$ is not general, then there is an irreducible component $\mathcal{X}\subset \mathcal{M}_{\sigma}(\mathcal{K}u(X),s)$ and $\mathcal{X}$ is singular. The singular locus $\mathrm{Sing}(\mathcal{X})=\mathcal{S}_1\cup \mathcal{S}_2$, where $\mathcal{S}_1$ is a set of finitely many points on $\Sigma(X)$ and $\mathcal{S}_2$ is a set of finitely many points outside $\Sigma(X)$. 
    \end{enumerate}
\end{theorem}

Next, we construct an irreducible component of Bridgeland moduli space $\mathcal{M}_{\sigma}(\mathcal{A}_{X'},x=1-2L)$ of $\sigma$-stable objects of character $x=1-2L$ in the Kuznetsov component of an ordinary Gushel-Mukai threefold $X'$. To simplify arguments and computations, we work with another semiorthogonal decomposition $D^b(X_{10}')=\langle\mathcal{A}_{X'},\mathcal{O}_{X'},\mathcal{E}^{\vee}\rangle$, where $\mathcal{A}_{X'}$ is defined as $\mathcal{O}_{X'}^{\perp}\cap(\mathcal{E}^{\vee})^{\perp}$ is also called the Kuznetsov component of $X'$ and it is known that $\mathcal{K}u(X')\simeq\mathcal{A}_{X'}$. 

\begin{theorem}
\label{theorem_modulispace_OGM}
\leavevmode\begin{enumerate}
    \item If $X'$ is general, then the projection functor $\mathrm{pr}:=\bL_{\mathcal{O}_{X'}}\bL_{\mathcal{E}^{\vee}}$ produces a smooth irreducible component $p(\mathcal{C}(X'))$ of dimension two in the moduli space $\mathcal{M}_{\sigma}(\mathcal{A}_{X'}, x=1-2L)$ of $\sigma$-stable objects in $\mathcal{A}_{X'}$, where $p:\mathcal{C}(X')\rightarrow p(\mathcal{C}(X'))\cong\mathcal{C}_m(X')$ is the blow down morphism from Fano surface $\mathcal{C}(X')$ of conics to its minimal model $\mathcal{C}_m(X')$, contracting the unique exceptional curve to a smooth point.
    \item If $X'$ is not general, then the projection functor $\mathrm{pr}$ produces a singular irreducible component $p(\mathcal{C}(X'))$ of dimension two in the moduli space $\mathcal{M}_{\sigma}(\mathcal{A}_{X'}, x=1-2L)$
\end{enumerate}


\end{theorem}
We call a vector $u$ in numerical Grothendieck group $\mathcal{N}(\mathcal{K}u(V))$ of the Kuznetsov component of a prime Fano threefold $V$ a $r$-class if the Euler character $\chi(u,u)=r$. Up to sign, there are two $(-2)$-class $w, 2v-w$ in $\mathcal{N}(\mathcal{K}u(Y))$ and $s, 3s-2t$ in $\mathcal{N}(\mathcal{K}u(X))$ respectively. Assume that there is an equivalence $\Phi:\mathcal{K}u(Y)\simeq\mathcal{K}u(X)$. Then $\Phi$ would identify $\mathcal{M}_{\sigma}(\mathcal{K}u(Y),w)$ with either $M_{\sigma'}(\mathcal{K}u(X),s)$ or $M_{\sigma'}(\mathcal{K}u(X), 3s-2t)$. If $\phi(w)=3s-2t$, then $\phi(w-2v)=\pm s$. But it is known that $M_{\sigma}(\mathcal{K}u(Y),w)\cong M_{\sigma}(\mathcal{K}u(Y), w-2v)$ by \cite[Proposition 5.4]{pertusi2020some}, so we only need to study the moduli space of stable objects with a fixed class $w$ in $\mathcal{K}u(Y)$ and the corresponding moduli space of stable objects with class $\phi(w)=s$ in $\mathcal{K}u(X)$. On the other hand, we show that the a $\sigma$-stable object of $(-2)$-class or $(-1)$-class in $\mathcal{K}u(Y)$ or $\mathcal{K}u(X)$ would be stable for every Serre-invariant stability conditions. Thus we may fix one Serre-invariant stability condition on the Kuznetsov components of $Y$ and $X$, respectively.

In \cite{altavilla2019moduli}, the authors showed the moduli space $\mathcal{M}_{\sigma}(\mathcal{K}u(Y),w)$
 has two irreducible components $\mathcal{Y}$ and
 $\mathcal{C}$, they intersect at the two dimensional ramification locus $\mathcal{R}\subset Y$.The dimension of tangent spaces of the moduli spaces at each point in $\mathcal{R}$ is $4$ and $3$ else where. As a consequence, the conjectural equivalence $\Phi:\mathcal{K}u(Y)\simeq\mathcal{K}u(X)$ would induce a morphism $\gamma$, between $\mathcal{Y}$(or $\mathcal{C}$) and the irreducible component $\mathcal{Z}$ of $\mathcal{M}_{\sigma'}(\mathcal{K}u(X),s)$ constructed in Theorem~\ref{theorem_main2}. In particular, $\gamma$ would identify their tangent spaces at each point. But this is impossible since $\mathcal{Z}$ is smooth of dimension three everywhere(when $X$ is general) or at most have finitely many points whose tangent spaces have dimension $4$. Thus Theorem~\ref{theorem_SOGM_Ku_false} is proved for quartic double solid $Y$ and special Gushel-Mukai threefold $X$.

Similarly, for the pair of quartic double solid $Y$ and ordinary Gushel-Mukai threefold $X'$, we only need to study the moduli space $\mathcal{M}_{\sigma}(\mathcal{K}u(Y),v=1-L)$ and $\mathcal{M}_{\sigma'}(\mathcal{A}_{X'},x=1-2L)$ with respect to every Serre-invariant stability condition $\sigma$ and $\sigma'$ on $\mathcal{K}u(Y)$ and $\mathcal{A}_{X'}$ respectively. In \cite{pertusi2020some}, the authors show $\mathcal{M}_{\sigma}(\mathcal{K}u(Y),v)$ is isomorphic to Fano surface $\Sigma(Y)$ of lines on quartic double solid. Then the conjectural equivalence $\Phi:\mathcal{K}u(Y)\simeq\mathcal{A}_{X'}$ would identify $\Sigma(Y)$ with the irreducible component $\mathcal{S}$ of $\mathcal{M}_{\sigma'}(\mathcal{A}_{X'},x=1-2L)$ constructed in Theorem~\ref{theorem_modulispace_OGM}, which is the minimal model $\mathcal{C}_m(X')$ of Fano surface of conics. But these two surfaces are not isomorphic by \cite{logachev2012fano}. Thus Theorem~\ref{theorem_SOGM_Ku_false} is proved for quartic double solid $Y$ and ordinary Gushel-Mukai threefold $X'$.

\subsection{Prior and related work}
\subsubsection{Hilbert scheme of low degree rational curves on Fano variety}
The study of Hilbert scheme of low degree rational curves is a classical topic in enumerative geometry. An excellent survey is \cite{ciliberto2019lines}. A comprehensive treatment on variety of lines and conics on Fano threefolds is given in \cite{KPS2018}. Hilbert scheme of lines, conics and twisted cubics on $Y_5$ were studied in \cite{sanna2014rational}. The Hilbert scheme of lines and conics on Gushel-Mukai threefold were described in  \cite{Iliev94line},\cite{Iliev94conic},\cite{debarre2013curves} and \cite{logachev2012fano}. The Hilbert scheme of rational curves on prime Fano threefold of index one is systemtically studied in \cite{lehmann2018rational}. The Hilbert scheme of twisted cubics on $\mathbb{P}^n$ was carefully studied in \cite{RS85Hilb} and \cite{ellingsrud1981variety}. The Hilbert scheme of twisted cubics on a cubic fourfold was systematically studied in \cite{lehn2017twisted}. The authors showed the component parametrising generalized twisted cubics is related to an eight dimensional hyperK$\ddot{a}$hler varietiy $Z$ via a $\mathbb{P}^2$-bundle map followed by a divisorial contraction. In \cite{li2018twisted}, $Z$ is constructed as a Bridgeland stable objects in Kuznetsov component.

\subsubsection{Moduli space of objects in the Kuznetsov component}
The first example of stability condition on a Kuznetsov component of a cubic threefold $Y_3$ is given in \cite{BMMS}, where it is shown that the moduli space of stable objects with class being ideal sheaves of a line $[I_L]$ is isomorphic to the Fano surface of lines $\Sigma(Y_3)$. In \cite{lahoz2015arithmetically}, it is shown that moduli space of stable objects with certain class is isomorphic to Fano variety of lines of blown up of $Y_3$. More recently, the LLVS eightfold constructed in \cite{lehn2017twisted} and Fano variety of lines of cubic fourfold are constructed as moduli space of Bridgeland stable objects in Kuznetsov component on a cubic fourfold in \cite{li2018twisted} and \cite{lahoz2018generalized}. After a direct construction of stability condition in Kuznetsov components for a series Fano threefolds in \cite{bayer2017stability}, the moduli spaces of stable objects in the Kuznetsov components of prime Fano threefolds of index 2 are systematically studied in \cite{altavilla2019moduli}, \cite{pertusi2020some}, \cite{petkovic2020note} ,\cite{bayer2020desingularization} and \cite{qin2021bridgeland}.  In \cite{pertusi2020some}, the author show that $\Sigma(Y_d)$ is constructed as moduli space of Bridgeland stable objects in $\mathcal{K}u(Y_d)$ for any $d\neq 1$ while $\Sigma(Y_1)$ is an irreducible component of moduli space of stable objects in $\mathcal{K}u(Y_1)$\cite{petkovic2020note}. In \cite{PPZ20Enriques}, the authors construct stability conditions on Kuznetsov components of Gushel-Mukai fourfolds and study the moduli space of stable objects in Kuznetsov component. In \cite{LZ2021moduli}, the authors study the moduli spaces of stable objects in $\mathcal{K}u(Y_d)$ and $\mathcal{K}u(X_{4d+2})$ for $d=3,4$ and $5$ and they realize those moduli spaces as a certain class of interesting classical moduli space of sheaves. In \cite{LZ2021rationalcurves}, the authors show that the Hilbert scheme $\Sigma(X)$ of lines of even genus $g\geq 6$ prime Fano threefold $X$ of index one is the \emph{Brill-Noether} locus of the Bridgeland moduli space of stable objects in $\mathcal{K}u(X)$ with respect to  $i^!\mathcal{E}\in\mathcal{K}u(X)$, where $i:\mathcal{K}u(X)\hookrightarrow D^b(X)$ is the inclusion.

\subsubsection{Kuznetsov's type equivalences for Fano threefold}
Equivalences between Kuznetsov components for the pair of Fano threefolds $(Y_d,X_{4d+2}),d=3,4,5$ are established 
in the pioneer work  \cite{kuznetsov2003derived},\cite{kuznetsov2007homological},\cite{kuznetsov2009derived},\cite{kuznetsov2006hyperplane}. Besides our work for $d=2$ case, very recently, the non existence of equivalences between $\mathcal{K}u(Y_2)$ and $\mathcal{K}u(X_{10})$(and $\mathcal{K}u(X'_{10})$) is also obtained in \cite{BP2021Fano} by a completely different method. In \cite{BP2021Fano} and \cite{kuznetsov2021categorical}, the authors show that $\mathcal{K}u(Y_2)$ and $\mathcal{K}u(X_{10})$($\mathcal{K}u(X'_{10})$) are deformation equivalent. The last case of Conjecture~\ref{main_conjecture}($d=1$) is disproved in the upcoming paper \cite{Bayer2021Kuz}. 

\subsubsection{Categorical Torelli Theorem}
For a prime Fano threefold $Y_d$ of index two, categorical Torelli theorem states that the Kuznetsov components $\mathcal{K}u(Y_d)$ uniquely determines $Y_d$ in the moduli space $\mathcal{MF}^2_d$ up to isomorphism while for a prime Fano threefold $X_{4d+2}$ of index one, birational categorical Torelli theorem states that $\mathcal{K}u(X_{4d+2})$ determines $X_{4d+2}$ only up to birational isomorphisms. Categorical Torelli theorems have been proved for $Y_d, d\in\{2,3,4,5\}$ in \cite{altavilla2019moduli}, \cite{BMMS}, \cite{bernardara2013semi} \cite{bayer2020desingularization} and birational categorical Torelli theorems holds automatically for $X_{4d+2}, d\in\{3,4,5\}$ by Conjecture~\ref{main_conjecture}. In  \cite{JLLZh2021}, we prove (birational)categorical Torelli theorem for Gushel-Mukai threefolds. In the upcoming work \cite{JLZ2021}, we prove refined categorical Torelli theorem for a series prime Fano threefolds of index one.

\subsection{Organization of the paper}
In section~\ref{section2}, we classify twisted cubics on special Gushel-Mukai threefold and describe the action of involution on twisted cubics. In Section~\ref{section3}, we prove Theorem~\ref{theorem_main1}(1)(Theorem~\ref{theorem_projectivity_hilbertscheme}). In section~\ref{section4}, we review the definition of (weak)stability condition on a triangulated category. Then we prove a series useful criterion on stability of objects with small ext-groups in the Kuznetsov components of an Enriques category. We show that the Serre-invariant stability conditions on $\mathcal{K}u(Y_2)$ and $\mathcal{K}u(X_{10})$ are unique up to the $\widetilde{GL_+(2,\mathbb{R})}$-action.  In section~\ref{section5}, we prove stability of the projection of ideal sheaves of twisted cubics onto Kuznetsov components.
In section~\ref{section_bridgelandmoduli}, we study the Bridgeland moduli space of stable objects with class $[I_C]$ in $\mathcal{K}u(X_{10})$ on general special Gushel-Mukai threefolds and prove Theorem~\ref{theorem_main2}(Theorem~\ref{Theorem_projectivity_modulispace}). In section~\ref{section_nongeneral_GM}, we study Hilbert scheme of twisted cubics on non general special Gushel-Mukai threefolds and corresponding Bridgeland moduli spaces and prove Theorem~\ref{theorem_main1}(2)(Theorem~\ref{theorem_singular_hilbertscheme}) and Theorem~\ref{prop_singular_modulispace}. In section~\ref{section_OGM}, we study the Bridgeland moduli space of stable objects of a $(-1)$-class in the Kuznetsov component $\mathcal{A}_{X_{10}'}$ on smooth ordinary Gushel-Mukai threefolds and  prove Theorem~\ref{theorem_modulispace_OGM}. In section~ \ref{section7}, we prove the main theorem~ \ref{theorem_SOGM_Ku_false} for the the case of special Gushel-Mukai threefolds. In section~\ref{section_maintheorem_for_OGM}. we prove the main Theorem~\ref{theorem_SOGM_Ku_false} for the case of ordinary Gushel-Mukai threefolds.

\subsection{Acknowledgement}
I would like to thank many people for their help.  First of all, I am very grateful to  Arend Bayer who introduces me to the subject of Bridgeland stability condition and for suggesting this problem to me. I am grateful that Alexander Kuznetsov, Atanas Iliev, Dmitry Logachev, Alexander Perry, Xiaolei Zhao answered many of my questions on Gushel-Mukai threefolds. I thank Junyan Xu, Renjie Lyu, Yilong Zhang, Yu Zhao, Chen Jiang, Fei Xie, Zhengyu Hu, Xuqiang Qin, Zhiwei Zheng and Qingyuan Jiang for answering questions on geometry of Fano threefolds. I thank Franco Rota, Laura Pertusi, Song Yang, Bingyu Xia, Naoki Koseki, Augustinas Jacovskis, Huachen Chen for answering me questions on Bridgeland stability conditions and wall crossing. I would like to thank Tingyu Sun for his encouragement and support. Special thanks to Yong Hu for teaching me birational geometry of projective threefolds. The author is supported by ERC Consolidator grant WallCrossAG, no. 819864.

\section{Gushel-Mukai threefolds and geometry of twisted cubics}
\label{section2}

In this section we briefly recall the definition of Gushel-Mukai threefolds and their properties. Then we study the geometry of twisted cubics on it.

\subsection{Gushel-Mukai threefolds}
A smooth Gushel-Mukai threefold $X$ is a Fano threefold with $\mathrm{Pic}(X)=\mathbb{Z}H$ and index $1$, genus $g=6$ and degree $H^3=10$. It is either a quadric section of a linear section of codimension $2$ of the Grassmannian $\mathrm{Gr}(2,5)$, in which case it is called ordinary,  or a double cover of a linear section $Y_5\subset\mathrm{Gr}(2,5)$ of codimension 3 ramified in a quadric hypersurface $\mathcal{B}$, in which case it is called special. A special Gushel-Mukai threefold has a natural involution $\tau, X\rightarrow X$ induced by double cover $\pi, X\rightarrow Y_5$. There exists a unique stable vector bundle $\mathcal{E}$ of rank $2$ with $c_1(\mathcal{E})=-H$ and $c_2(\mathcal{E})=4L$, where $L$ is the class of a line on $X$. It is the pull back of the tautological sub-bundle on $\mathrm{Gr}(2,5)$. In addition, $\mathcal{E}$ is exceptional and $H^{\bullet}(X,\mathcal{E})=0$ \cite{mukai1989biregular}. If the branch locus $\mathcal{B}\in |\mathcal{O}_Y(2)|$ is  general, then it does not contain a line or conic.

\subsection{Lines and conics on $Y_5$ and $X_{10}$}
Let $Y_5$ be the smooth Fano threefold with Picard rank one, index two and degree 5. Denote by $\mathcal{U}$ the rank $2$ tautological bundle on $Y_5$. We collect some facts on geometry of lines and conics $Y_5$.
\begin{predl}\cite{sanna2014rational}
\label{proposition_lines_Y}
\leavevmode \begin{enumerate}
    \item The Hilbert scheme $\Sigma(Y_5)$ of lines   is isomorphic to $\mathbb{P}^2$.
    \item The Hilbert scheme $\mathcal{C}(Y_5)$ of conics is isomorphic to $\mathbb{P}(H^0(\mathcal{U}^{\vee}))=\mathbb{P}^4$. This means that every conic in $Y_5$ is a zero locus of the section of $\mathcal{U}^{\vee}$.
    \end{enumerate}
\end{predl}

\begin{predl}
\label{proposition_tangentlocus}
\leavevmode \begin{enumerate}
    \item Given a line $L\subset Y_5$, the locus of all lines intersecting $L$ is parametrised by $\mathbb{P}^1$. 
    \item Let $\mathcal{B}$ be a quadric hypersurface in $Y_5$. The locus of all lines tangent to $\mathcal{B}$ is a degree $10$ curve $\Gamma\subset\mathbb{P}^2$. 
    \item The Hilbert scheme of lines $\Sigma(X)$ on smooth special Gushel-Mukai threefold $X$ is a double cover of $\Gamma$. If there exists a line $l\in\mathcal{B}$, then its preimage $\pi^{-1}(l)_{red}$ is a singular point in $\Sigma(X)$. Otherwise, $\Sigma(X)$ is a smooth irreducible curve of genus $71$. 
\end{enumerate}
\end{predl}
\begin{proof}
The intersection locus of two lines $\{(L_1,L_2)\subset\Sigma(Y_5)\times\Sigma(Y_5)|L_1\cap L_2\neq\emptyset\}$ is a flag variety $\mathrm{Fl}(1,2;3)\subset\mathbb{P}^2\times\mathbb{P}^2$\cite[Proposition 2.26]{sanna2014rational}. Let $p,q$ be the two projection maps from $\mathrm{Fl}(1,2;3)\rightarrow\Sigma(Y_5)$. Then the locus of all lines intersecting the fixed line $L$ is just $q(p^{-1}(L))\subset\Sigma(Y_5)$. Then this locus parametrises all the one dimensional vector spaces in a fixed two dimensional vector space, which is nothing but $\mathrm{Gr}(1,2)=\mathbb{P}^1$. 

$(2)$ and $(3)$ follows from \cite[Section 3.1,3.2]{Iliev94line}. If $\mathcal{B}$ contains a line $l$, then there is a line $L=(\pi^{-1}l)_{red}$ with non-reduced conic structure, which gives a singular point in $\Sigma(X)$ since the normal bundle $\mathcal{N}_{L|X}\cong\mathcal{O}_L(1)\oplus\mathcal{O}_L(-2)$. The connectedness of $\Sigma(X)$ follows from the same argument as \cite[Lemma 2.13]{debarre2020gushel}. 
\end{proof}

\subsection{Geometry of twisted cubics on special Gushel-Mukai threefold}
A twisted cubic with respect to the very ample divisor $H$ of $X$ is a subscheme $C\subset X\subset\mathbb{P}^7$ with Hilbert polynomial $p_C(t)=3t+1$. Denote by $\mathcal{H}=\mathrm{Hilb}^{p(t)=3t+1}(X)$ the Hilbert scheme of twisted cubics on $X$.
$\mathcal{H}$ is a subscheme of $\mathrm{Hilb}^{3t+1}(\mathbb{P}^7)$. The latter has two irreducible components, one is parametrising the aCM twisted cubics and the other is parametrising plane cubics, with an embedded point. As $X$ is defined as an intersection of quadrics in projective space and it does not contain a plane, it can not contain any plane cubic, therefore it only has aCM twisted cubics. An aCM twisted cubic $C$ can only be in the list of those in \cite[Section 1]{lehn2017twisted}. They are classified as follows.


\begin{predl}
\label{prop_classificationcubic}
Let $X$ be a special Gushel-Mukai threefold and $C\subset X$ a twisted cubic. Then $C$ is pure one dimensional and is either \begin{enumerate}
    \item a smooth irreducible twisted cubic, or
    \item the union of a line and a conic(smooth, reducible and double lines), or
    \item the union of a line and another line $M$ with Hilbert polynomial $p_M(t)=2t+2$.  
    \end{enumerate}
If in addition the branch locus on $Y$ does not contain lines or conics, then $C$ does not have double lines as its component. 
\end{predl}

\begin{proof}
$C$ is an aCM twisted cubic with $p_C(t)=3t+1$, which implies that all irreducible components of $C$ have dimensional $1$ and the sum of degrees with multiplicities of all components is $3$. \begin{enumerate}
    \item If $C$ is integral, it must be a smooth irreducible twisted cubic.
    \item If $C$ has two irreducible components $L$ and $M$ of degree $1$ and $2$ respectively, then $L$ is a line and  $p_C(t)=(t+1)+p_M(t)-\delta=3t+1$, where $\delta$ is intersection multiplicities of $L$ and $M$. Then $p_M(t)=2t+\delta$, if $\delta=1$, then $M$ is a conic (smooth, double lines). If $\delta=2$, then $M$ is a line with Hilbert polynomial $p_M(t)=2t+2$.  
    \item If $C$ has three irreducible components, then all of them are lines and $C$ is union of three lines(chain of three lines or three lines intersecting at one point).
    \item If $C$ is irreducible but non reduced, then it is either the first infinitesimal neighborhood of a line or a union of a line $L$ and a double line $M$. But if $C$ is the first infinitesimal neighborhood of a line $L$, we have a short exact sequence:
$$0\rightarrow\mathcal{N}_{L|X}^{\vee}\rightarrow\mathcal{O}_C\rightarrow\mathcal{O}_L\rightarrow 0$$  By \cite{iskovskikh1999fano} the conormal bundle 
$\mathcal{N}_{L|X}^{\vee}$ of a line in $X$ can only be $\mathcal{O}_L\oplus\mathcal{O}_L(1)$ or $\mathcal{O}_L(-1)\oplus\mathcal{O}_L(2)$. Thus $p_C(t)=\chi(\mathcal{O}_C(tH))=\chi(\mathcal{N}_{L|X}^{\vee}(tH))+\chi(\mathcal{O}_L(tH))=3t+4$, which is impossible.
\end{enumerate}
If $X$ is general, then $\Sigma(X)$ is smooth by Proposition~\ref{proposition_tangentlocus}. Every line $L$ does not admit a non reduced conic structure. Therefore $C$ does not have double lines as its components. 
\end{proof}




Let $\mathcal{E}$ be the rank 2 stable vector bundle on $X$, it is the pull back of a rank 2 vector bundle $\mathcal{U}$ on $Y$ under the covering map $\pi:X\rightarrow Y$. To understand the action of the involution $\tau$ on a twisted cubic, we classify the zero locus of sections of the dual vector bundle $\mathcal{E}^{\vee}$ on $X$: 


\begin{predl}
\label{prop_classfication_zerolocus}
Let $X$ be a special Gushel-Mukai threerold, then the zero locus of a section of $\mathcal{E}^{\vee}$ only contains twisted cubic of the form $(2)$ in Proposition~\ref{prop_classificationcubic}.
\end{predl}


\begin{proof}
As $c_2(\mathcal{E}^{\vee})=4L$, the zero locus of a section $s$ of $\mathcal{E}^{\vee}$ is a degree $4$ curve. Note that $\pi_*\mathcal{O}_X\cong\mathcal{O}_Y\oplus\mathcal{O}_Y(-H)$. Then  $H^0(X,\mathcal{E}^{\vee})=H^0(X,\pi^*\mathcal{U}^{\vee})\cong H^0(Y,\pi_*\pi^*\mathcal{U}^{\vee})\cong H^0(Y,\mathcal{U}^{\vee}\otimes\pi_*\mathcal{O}_X)\cong H^0(Y,\mathcal{U}^{\vee}\oplus\mathcal{U}^{\vee}(-H))=H^0(Y,\mathcal{U}^{\vee})$. Thus the zero locus of section $\mathcal{E}^{\vee}$ is nothing but $\pi$-preimage of  $Y\cap\mathrm{Gr}(2,4)$, which is preimage of a linear section of $\mathrm{Gr}(2,4)$ of codimension $3$.  But $\mathrm{Gr}(2,4)\subset\mathbb{P}^5$ is a quadric, it follows that its linear section of codimension $3$ is either a plane or a conic, but $Y$ does not contain any plane, hence $Y\cap\mathrm{Gr}(2,4)$ is a conic. Therefore the zero locus of a section of $\mathcal{E}^{\vee}$ is preimage of a conic on $Y_5$. It is clear that a smooth twisted cubic on $X$ is not contained in the locus. On the other hand, if a twisted cubic $C$ of the form $(3)$ in Proposition~\ref{prop_classificationcubic} is contained in the locus of $s\in H^0(\mathcal{E}^{\vee})$, then $\pi(C)$ must be a reducible conic on $Y$, which is also impossible. 

\end{proof}

\begin{predl}
\label{prop_cubic_inKu}
$\mathrm{Hom}(\mathcal{E},I_C)=k$ if and only if the scheme theoretical image of twisted cubic $C$ under covering map $\pi$ is a conic on $Y$. 
\end{predl}

\begin{proof}
Applying $\mathrm{Hom}(\mathcal{E},-)$ to the standard exact sequence for $C$, the part of the long exact sequence is given by
$$0\rightarrow\mathrm{Hom}(\mathcal{E},I_C)\rightarrow\mathrm{Hom}(\mathcal{E},\mathcal{O}_X)\xrightarrow{p} \mathrm{Hom}(\mathcal{E},\mathcal{O}_C)$$
Then $\mathrm{Hom}(\mathcal{E},I_C)\cong\mathrm{ker}(p)$, which is nothing but the sections of $\mathcal{E}^{\vee}$ whose zero locus contain $C$. If $\mathrm{Hom}(\mathcal{E},I_C)=k$, then $C$ is contained in the zero locus of section of $\mathcal{E}^{\vee}$, which is the preimage of a conic on $Y$ by Proposition~\ref{prop_classfication_zerolocus}. Conversely, every conic on $Y$ is zero locus of a section $s_Y$ of $\mathcal{U}^{\vee}$ by Proposition~\ref{proposition_lines_Y}. If $\pi(C)$ is a conic on $Y$, then  $C\subset\pi^{-1}Z(s_Y)=Z(\pi^*(s_Y))=Z(s_X)$, where $\pi^*(s_Y)=s_X\in H^0(\mathcal{E}^{\vee})$, which implies $\mathrm{hom}(\mathcal{E},I_C)\geq 1$. If $\mathrm{hom}(\mathcal{E},I_C)\geq 2$, then $\pi(C)$ would be contained in the intersection of two distinct conics on $Y$, which is evidently absurd. 
\end{proof}





\begin{predl}
\label{prop_classfication_cubic_degree}
Let $C$ be a twisted cubic on $X$. Then $\pi(C)$ is either a conic or a twisted cubic. In particular, if $\pi(C)$ is a twisted cubic $D$, then $D$ is either triple tangent to $\mathcal{B}$ or $D\subset\mathcal{B}$. 
\end{predl}

\begin{proof}
By Proposition~\ref{prop_cubic_inKu}, $\mathrm{hom}(\mathcal{E},I_C)$ is either $0$ or $1$. If $\mathrm{hom}(\mathcal{E},I_C)=1$, then $\pi(C)$ is a conic on $Y$ by Proposition~\ref{prop_cubic_inKu}. If $\mathrm{hom}(\mathcal{E},I_C)=0$, then $\pi(C)$ is not a conic. \begin{enumerate}
    \item $C$ is a smooth twisted cubic, then $D=\pi(C)$ is also a smooth twisted cubic and $D$ is either in branch locus or triple tangent to it. Then $C$ is either fixed by $\tau$ or $\tau(C)$ is another smooth twisted cubic. 
    \item $C$ is the union of a line $L$ and a conic $M$. If $C$ is fixed by $\tau$, then $\pi(L)$ is a line $l\subset\mathcal{B}$ and $\pi(M)$ is a conic $m\subset\mathcal{B}$. If $C$ is not fixed by $\tau$, a case by case analysis shows $\pi(C)$ is a twisted cubic triple tangent to $\mathcal{B}$.
    \item $C$ is the union of a line $L$ and another line $N$ with $p_N(t)=2t+2$, similar argument applies. 
\end{enumerate}
\end{proof}

\begin{corollary}
\label{corollary_twistedcubic_notfix}
Let $X$ be a smooth special Gushel-Mukai threefold such that the branch locus is general. Then if $C$ is a twisted cubic on $X$, it is not fixed by the involution $\tau$. 
\end{corollary}

\begin{proof}
If $\tau(C)=C$, by Proposition~\ref{prop_classfication_cubic_degree}, there must exist a line,a conic or a smooth twisted cubic in $\mathcal{B}$. But $\mathcal{B}$ is general, it does not contain any lines,conics or twisted cubics. This is a contradiction. 
\end{proof}

\begin{theorem}
\label{theorem_criterion_inKu}
Let $X$ be a smooth special Gushel-Mukai threefold and $C$ be a twisted cubic. Then $I_C\in\mathcal{K}u(X)$ if and only $\pi(C)$ is a twisted cubic. In particular, if $C$ is smooth, then $I_C\in\mathcal{K}u(X)$. 
\end{theorem}

The proof of Theorem~\ref{theorem_criterion_inKu} breaks into several lemmas:

\begin{lemma}
\label{lemma_restrictedto_cubic_vanishing}
Let $C$ be a twisted cubic curve on a special Gushel-Mukai threefold $X$, $H^0(X,\mathcal{E}|_C)=0$.
\end{lemma}

\begin{proof}
Put $W=H^0(X,\mathcal{E}^{\vee})^{\vee}\cong k^5$. The pull back of the tautological sequence on Grassmannian to the twisted cubic $C$ is
$$0\rightarrow\mathcal{E}|_C\rightarrow W\otimes\mathcal{O}_C\xrightarrow{p}(W/\mathcal{E})|_C\rightarrow 0$$ This means that $H^0(\mathcal{E}|_C)=\mathrm{Ker}(H^0(W\otimes\mathcal{O}_C)\xrightarrow{p}H^0((W/\mathcal{E})|_C))$. Assume that $H^0(\mathcal{E}|_C)\neq 0$. This implies that $C$ is contained in the zero locus of $w\in W$, considered as a global section of the quotient bundle $W/\mathcal{E}:=\cQ\cong\pi^*\mathcal{V}$, where $\mathcal{V}$ is the tautological quotient bundle on $Y_5$. Thus $C$ would be contained in the $\pi$-preimage of the zero locus of section of $\mathcal{V}$. But by \cite[Lemma 2.18]{sanna2014rational}, the zero locus of section of $\mathcal{V}$ is either a line or a point, so that the twisted cubic $C$ would be contained in either two points or a conic on $X$, which is impossible. This implies $H^0(\cE|_C)=0$. 
\end{proof}

\begin{lemma}
\label{lemma_globalsection_stablebundle}
$\mathrm{Hom}(\mathcal{O}_X,\mathcal{E}^{\vee})=k^5$ and $\mathrm{Ext}^i(\mathcal{O}_X,\mathcal{E}^{\vee})=0$ for $i\geq 1$.
\end{lemma}

\begin{proof}
Let $S\in |H|$ be a generic member, which is a smooth K3 surface. Applying global section functor to $0\rightarrow\mathcal{E}(-H)\rightarrow\mathcal{E}\rightarrow\mathcal{E}_S\rightarrow 0$, we have the long exact sequence:
$$0\rightarrow H^0(\mathcal{E}(-H))\rightarrow H^0(\mathcal{E})\rightarrow H^0(\mathcal{E}_S)\rightarrow \mathrm{H}^1(\mathcal{E}(-H))\rightarrow \mathrm{H}^1(\mathcal{E})\rightarrow \mathrm{H}^1(\mathcal{E}_S)\rightarrow \mathrm{H}^2(\mathcal{E}(-H))\rightarrow \mathrm{H}^2(\mathcal{E})$$

Since $\langle\mathcal{E},\mathcal{O}_X\rangle$ is an exceptional pair, $H^i(\mathcal{E})=0$ for all $i$ and $\mathrm{H}^1(\mathcal{E}_S)=0$ by \cite{bayer2017stability}[Theorem 6.2].  Thus $\mathrm{H}^2(\mathcal{E}(-H))=0$ and $\mathrm{Ext}^1(\mathcal{O}_X,\mathcal{E}^{\vee})=\mathrm{Ext}^2(\mathcal{E}^{\vee},\mathcal{O}_X(-H))\cong\mathrm{Ext}^2(\mathcal{O}_X,\mathcal{E}(-H))=\mathrm{H}^2(\mathcal{E}(-H))=0$.  Similarly, $0=H^0(\mathcal{E}(-H))=\mathrm{Ext}^3(\mathcal{O}_X,\mathcal{E}^{\vee})=0$. The global section $H^0(\mathcal{E}^{\vee})=k^5$ \cite[Remark 4.11]{kuznetsov2003derived},  and $\chi(\mathcal{E}^{\vee})=5+\mathrm{H}^2(\mathcal{E}^{\vee})-0=\chi(2+H+L-\frac{1}{3}P)=2+\frac{17}{6}+\frac{1}{2}-\frac{1}{3}=2+3=5$. Then $\mathrm{H}^2(\mathcal{E}^{\vee})=0$
\end{proof}

\begin{lemma}
\label{lemma_globalsection_vectorbundle_cubic}
Let $C$ be a twisted cubic on $X$, then $\mathrm{Ext}^i(\mathcal{E},\mathcal{O}_C)=0$ for $i\geq 1$ and $\mathrm{Hom}(\mathcal{E},\mathcal{O}_C)=k^5$.
\end{lemma}

\begin{proof}
The Euler characteristic $\chi(\mathcal{E},\mathcal{O}_C)$ only depends on the character of $j_*\mathcal{O}_C$, so it is enough to assume $C$ is smooth. Note that $\chi(\mathcal{E},\mathcal{O}_C)=h^0(\mathcal{E}^{\vee}|_C)-\mathrm{h}^1(\mathcal{E}^{\vee}|_C)+\mathrm{h}^2(\mathcal{E}^{\vee}|_C)-\mathrm{h}^3(\mathcal{E}^{\vee}|_C)$. Since $\mathcal{E}^{\vee}$ is globally generated, by Lemma~\ref{lemma_restrictedto_cubic_vanishing}, we have $\mathcal{E}|_C\cong\mathcal{O}_C(-1)\oplus\mathcal{O}_C(-2)$. Then $\chi(\mathcal{E},\mathcal{O}_C)=5$. It is clear that $\mathrm{Ext}^i(\mathcal{E},\mathcal{O}_C)=0$ for $i\geq 2$ since $C$ is pure one dimensional. It remains to show that $\mathrm{h}^1(\mathcal{E}^{\vee}|_C)=0$. Since $C$ is a rational curve, its canonical bundle $\omega_C$ is negative on all its components. Then $\mathrm{h}^1(\mathcal{E}^{\vee}|_C)\cong \mathrm{h}^0(\mathcal{E}|_C\otimes\omega_C)=0$ since $\mathrm{h}^0(\mathcal{E}|_C)=0$ by Lemma~\ref{lemma_restrictedto_cubic_vanishing}. It follows that $\mathrm{Hom}(\mathcal{E},\mathcal{O}_C)=k^5$.
\end{proof}

\begin{lemma}
\label{lemma_cohomology_idealsheaf}
Let $C$ be a twisted cubic on $X$, then $\mathrm{Hom}^i(\mathcal{O}_X,I_C)=0$ for all $i$.
\end{lemma}
\begin{proof}
Consider the standard exact sequence
$$0\rightarrow I_C\rightarrow\mathcal{O}_X\rightarrow\mathcal{O}_C\rightarrow 0$$  
Then we have long exact sequence of cohomology:
\begin{align*}
0\rightarrow\mathrm{Hom}(\mathcal{O}_X,I_C)\rightarrow\mathrm{Hom}(\mathcal{O}_X,\mathcal{O}_X)\rightarrow\mathrm{Hom}(\mathcal{O}_X,\mathcal{O}_C)\rightarrow\mathrm{Ext}^1(\mathcal{O}_X,I_C)\\
\rightarrow\mathrm{Ext}^1(\mathcal{O}_X,\mathcal{O}_X)\rightarrow\mathrm{Ext}^1(\mathcal{O}_X,\mathcal{O}_C)
\rightarrow\mathrm{Ext}^2(\mathcal{O}_X,I_C)\rightarrow\mathrm{Ext}^2(\mathcal{O}_X,\mathcal{O}_X)
\end{align*}

By \cite[Corollary 1.38]{sanna2014rational} $\mathrm{H}^1(\mathcal{O}_C)=0$, then $\mathrm{Ext}^2(\mathcal{O}_X,I_C)=0$ and $\mathrm{Hom}(\mathcal{O}_X,\mathcal{O}_C)=H^0(\mathcal{O}_C)=k$. Thus $\mathrm{Hom}(\mathcal{O}_X,I_C)=0$ and $\mathrm{Ext}^1(\mathcal{O}_X,I_C)=0$. It is clear that $\mathrm{Ext}^3(\mathcal{O}_X,I_C)=0$. 
\end{proof}

Next, we prove Theorem~\ref{theorem_criterion_inKu}.

\begin{proof}[Proof of Theorem~\ref{theorem_criterion_inKu}]
The long exact sequence of standard exact sequence for $C$ is 
$$0\rightarrow\mathrm{Ext}^1(\mathcal{E},\mathcal{O}_C)\rightarrow\mathrm{Ext}^2(\mathcal{E},I_C)\rightarrow\mathrm{Ext}^2(\mathcal{E},\mathcal{O}_X)\rightarrow\mathrm{Ext}^2(\mathcal{E},\mathcal{O}_C)\rightarrow\mathrm{Ext}^3(\mathcal{E},I_C)\rightarrow\mathrm{Ext}^3(\mathcal{E},\mathcal{O}_X)=0$$
By Lemma~\ref{lemma_globalsection_vectorbundle_cubic}, $\mathrm{Ext}^2(\mathcal{E},I_C)=\mathrm{Ext}^3(\mathcal{E},I_C)=0$. Note that $\chi(\mathcal{E},I_C)=\chi(\mathcal{E}^{\vee}\otimes I_C)=\chi_0(2+H-5L-\frac{7}{3}P)=0$, then $\mathrm{Hom}(\mathcal{E},I_C)=\mathrm{Ext}^1(\mathcal{E},I_C)$. 

If $C$ is not in the zero locus of sections of $\mathcal{E}^{\vee}$, then $\mathrm{ext}^1(\mathcal{E},I_C)=\mathrm{Hom}(\mathcal{E},I_C)=0$. Then $\mathrm{Hom}^k(\mathcal{E},I_C)=0$ for all $k$, which implies that $I_C\in\mathcal{K}u(X)$.
Then the statement follows from Proposition~\ref{prop_cubic_inKu} and Proposition~\ref{prop_classfication_cubic_degree}. 
The other direction is trivial. In particular, if $C$ is smooth, then $C$ is not contained in the zero locus of section of $\mathcal{E}^{\vee}$ by Proposition~\ref{prop_classfication_zerolocus}. Then $I_C\in\mathcal{K}u(X)$. 
\end{proof}

\begin{remark}
\label{remark_stratification}
\leavevmode \begin{enumerate}
    \item By Theorem~\ref{theorem_criterion_inKu}, one can stratify $\mathcal{H}$ into $\mathcal{H}=\mathcal{H}_1\cup\mathcal{H}_2$, where $\mathcal{H}_1$ is the double cover of (the closure of) twisted cubics on $Y$ which are triple tangent to $\mathcal{B}$ and $\mathcal{H}_2$ is preimage of reducible conics on $Y$ and $\mathcal{H}_1=\{C|I_C\in\mathcal{K}u(X)\}$ and $\mathcal{H}_2=\{C|I_C\not\in\mathcal{K}u(X)\}$.  We will see in Lemma~\ref{lemma_connctedness_Hilbertscheme} that $\mathcal{H}_2$ is connected but not a connected component.
    \item In Section~\ref{section_bridgelandmoduli}, we will see that $\mathcal{H}_2$ is a ruled surface in $\mathcal{H}$. Its open complement $\mathcal{H}_1$ is an open denset subset since $\mathcal{H}$ is irreducible. Then $\mathcal{H}_2\subset\overline{\mathcal{H}_1}$.
\end{enumerate}
\end{remark}

\section{Hilbert scheme of twisted cubics}
\label{section3}
Thoughout this section, we assume $X$ is always a special Gushel-Mukai threefold with branch locus general. The main result of the section is following theorem:

\begin{theorem}
\label{Theorem_twistedcubics}
Let $X$ be a smooth special Gushel-Mukai threefold with general branch locus on $Y$. 
Then the Hilbert scheme $\mathcal{H}$ of twisted cubics is a smooth irreducible projective threefold.
\end{theorem}

To prove this theorem, we identify $\mathcal{H}$ with a moduli space of Gieseker stable sheaves with character $1-3L+\frac{1}{2}P$. 

\begin{lemma}
 \label{GiesekerHilb}
 Let $F$ be a stable sheaf on $X$ with character $1-3L+\frac{1}{2}P$, then $F$ is an ideal sheaf of a twisted cubic $C\subset X$. Conversely, if $C\subset X$ is a twisted cubic, then the ideal sheaf $I_C$ is stable and $\mathrm{ch}(I_C))=1-3L+\frac{1}{2}P$.
\end{lemma}

\begin{proof}
Since $F$ is a stable sheaf of rank one, it is a pure sheaf of $dimX$, hence a torsion free sheaf of rank one. We claim that $F$ must be of the form $F\cong I_C$ for a twisted cubic $C\subset X$. Note that the map: $F\xhookrightarrow{j}F^{\vee\vee}$ is injective and $F^{\vee\vee}$ is a reflexive sheaf of rank one over $X$, then $F^{\vee\vee}\cong\mathcal{O}_X(D)$ for some Cartier divisor $D$. Then $F\otimes\mathcal{O}_X(-D)\xhookrightarrow{}\mathcal{O}_X$ is an injection and $\mathcal{F}\otimes\mathcal{O}_X(-D)$ must be an ideal sheaf $I_C$ of a subscheme $C\subset X$. Then $F\cong I_C\otimes\mathcal{O}_X(D)$. Note that $\mathrm{ch}(F)=\mathrm{ch}(F^{\vee\vee})$. Then $\mathrm{ch}_1(F^{\vee\vee})=\mathrm{ch}_1(\mathcal{O}(D))=\mathrm{ch}_1(I_C\otimes\mathcal{O}(D))=\mathrm{ch}_1(I_C)+D$,  hence $\mathrm{ch}_1(I_C)=0$ and  $\mathrm{codim(C)}\geq 2$. On the other hand, $\mathrm{ch}_2(I_C)=-3L$, this means that $\mathrm{dim}(C)=1$. Thus $F$ is ideal sheaf of a curve $C$. By assumption on character of $F$, we have $\mathrm{ch}(j_*\mathcal{O}_C)=3L-\frac{1}{2}P$, where $C\xhookrightarrow{j} X$. Then its Hilbert polynomial $p_C(t)=\chi(\mathcal{O}_C(tH))=\chi(3L+3tP-\frac{1}{2}P)=\frac{3}{2}+3t-\frac{1}{2}=3t+1$, so $F$ is an ideal sheaf of twisted cubic $I_C$. Conversely, it is clear that the ideal sheaf $I_C$ of a twisted cubic is stable since it is a torison free sheaf of rank one. It is easy to see $\mathrm{ch}_0(i_*\mathcal{O}_C)=\mathrm{ch}_1(j_*\mathcal{O}_C)=0$. Since $C$ is a rational cubic, $\mathrm{ch}_2(j_*\mathcal{O}_C)=3L$ and $\chi(j_*\mathcal{O}_C)=1$. By Hirzebruch-Riemann-Roch, one has $\mathrm{ch}_3(j_*\mathcal{O}_C)=-\frac{1}{2}P$. 
\end{proof}

\begin{predl}
\label{Prop_computingcone}
If $X$ is a special Gushel-Mukai threefold and $C\in\mathcal{H}_2$ a twisted cubic. Define $G=\mathrm{cone}(\mathcal{E}\xrightarrow{s} I_C)$ in $D^b(X)$, where $s$ is the non zero map $\mathcal{E}\rightarrow I_C$. Then there is a triangle $\mathcal{O}_X(-H)[1]\rightarrow G\rightarrow\mathcal{O}_{L}(-2)$, where $L\subset X$ is a line.
\end{predl}

\begin{proof}
Taking cohomology of the triangle $\mathcal{E}\rightarrow I_C\rightarrow G$ with respect to $\mathrm{Coh}(X)$, we get a long exact sequence
$$0\rightarrow \mathcal{H}^{-1}(G)\rightarrow\mathcal{E}\xrightarrow{s} I_{C}\rightarrow \mathcal{H}^0(G)\rightarrow 0$$ and other cohomologies of $G$ are zero. Note that the image of $s$ must be an ideal sheaf of the vanishing locus of section $s$ of $\mathcal{E}^{\vee}$, containing $C$. Denote it by $I_D$, where $D=Z(s)$ is an elliptic quartic. By Proposition~\ref{prop_classfication_zerolocus}, it is $\pi$-preimage of a conic on $Y$.  Then we have two short exact sequence:
$$0\rightarrow \mathcal{H}^{-1}(G)\rightarrow\mathcal{E}\xrightarrow{s} I_D\rightarrow 0$$
$$0\rightarrow I_D\rightarrow I_C\rightarrow \mathcal{H}^0(G)\rightarrow 0.$$
Then $\mathcal{H}^{-1}(G)$ is a torsion free sheaf of rank one with character $\mathrm{ch}(\mathcal{E})-\mathrm{ch}(I_D)=(2-H+L+\frac{1}{3}P)-(1-4L+2P)=1-H+5L-\frac{5}{3}P$. It is easy to see $\mathcal{H}^{-1}(G)\cong\mathcal{O}_X(-H)$. On the other hand, the second exact sequence is the decomposition sequence for $D$, then $\mathcal{H}^0(G)\cong\mathcal{O}_L(-2)$, where $L$ is the residual curve of $C$ in $D$, which is a line on $X$. Then the desired result follows. 
\end{proof}

\begin{corollary}\label{prop_deriveddual_line}
The object $G$ defined in Proposition~\ref{Prop_computingcone} is the \emph{twisted derived dual} of an ideal sheaf of a line $L$ in the sense of \cite[Lemma 3.14]{kuznetsov2012instanton}. i.e. $G\cong\mathrm{R}\mathcal{H}om(I_L,\mathcal{O}_X(-H))[1]$.
\end{corollary}

\begin{proof}
Denote the \emph{twisted derived dual} of $I_L$ by $\mathbb{D}(I_C)$. Apply the functor $\mathrm{R}\mathcal{H}om(-,\mathcal{O}_X(-H))$ to the standard exact sequence associate to $j:L\hookrightarrow X$,  $0\rightarrow I_L\rightarrow\mathcal{O}_X\rightarrow\mathcal{O}_L\rightarrow 0$, we get an exact triangle:
$$\mathrm{R}\mathcal{H}om(\mathcal{O}_L,\mathcal{O}_X(-H))\rightarrow \mathrm{R}\mathcal{H}om(\mathcal{O}_X,\mathcal{O}_X(-H))\rightarrow \mathrm{R}\mathcal{H}om(I_L,\mathcal{O}_X(-H)).$$

It is clear that the middle term is $\mathcal{O}_X(-H)$. The third term is exactly $\mathbb{D}(I_L)[-1]$. Now we compute the first term. 
\begin{align*}
    \mathrm{R}\mathcal{H}om(\mathcal{O}_L,\mathcal{O}_X(-H))& \cong \mathrm{R}\mathcal{H}om(\mathcal{O}_L,j^{!}(\mathcal{O}_X(-H))) 
    \\
&\cong \mathrm{R}\mathcal{H}om(\mathcal{O}_L,\mathcal{O}_X(-H)\otimes\omega_{L|X}[-2])
\\
&\cong \mathrm{R}\mathcal{H}om(\mathcal{O}_L,\mathcal{O}_X(-H)\otimes\mathcal{O}_L(-1)[-2])
\\
&\cong\mathcal{O}_L(-2)[-2]
\end{align*}
The first isomorphism follows from Grothendieck duality. The second isomorphism holds since  $j^!=-\otimes\omega_{L|X}[-2]$. The third isomorphism holds because relative canonical bundle $\omega_{L|X}\cong\omega_L\otimes\omega_X^{-1}|_L\cong\mathcal{O}_L(-2)\otimes\mathcal{O}_X(H)|_L\cong\mathcal{O}_L(-1)$. 
Then we have the triangle:
$\mathcal{O}_L(-2)[-2]\rightarrow\mathcal{O}_X(-1)\rightarrow\mathbb{D}(I_L)[-1]$. Shift by $[1]$ and rotate, we get the triangle  
$$\mathcal{O}_X(-1)[1]\rightarrow\mathbb{D}(I_L)\rightarrow\mathcal{O}_L(-2)$$
Note that $\mathrm{Ext}^1(\mathcal{O}_L(-2),\mathcal{O}_X(-1)[1])\cong\mathrm{H}^1(L,\mathcal{O}_L(-2))=k$. Therefore $G\cong\mathbb{D}(I_L)$. 
\end{proof}

\begin{predl}\label{lemma_involutive_D}
Let $M$ be an object in subcategory $\mathcal{O}_X^{\perp}\subset D^b(X)$, then $\mathbb{D}(M)\in\mathcal{O}_X^{\perp}$ and $\mathbb{D}(\mathbb{D}(M))\cong M$. i.e. $\mathbb{D}$ is an involutive auto-equivalence of $\mathcal{O}_X^{\perp}$. 
\end{predl}

\begin{proof}
To show $\mathbb{D}(M)\in\mathcal{O}_X^{\perp}$, it is enough to show that $\mathrm{RHom}(\mathcal{O}_X,\mathbb{D}(M))=0$. Then \begin{align*}
    \mathrm{RHom}(\mathcal{O}_X,\mathbb{D}(M))&\cong \mathrm{RHom}(\mathcal{O}_X,\mathrm{R}\mathcal{H}om(M,\mathcal{O}_X(-H))[1])\\
    &\cong \mathrm{RHom}(M,\mathcal{O}_X(-H)[1])\\
    &\cong\mathrm{RHom}(\mathcal{O}_X(-H)[1], M\otimes\mathcal{O}_X(-H))\\
    &\cong \mathrm{RHom}(\mathcal{O}_X,M)=0
\end{align*}
The first isomorphism follows from the definition of $\mathbb{D}$. The second isomorphism follows from adjunction. The third isomorphism follows from Serre duality. Therefore $\mathbb{D}(M)\in\mathcal{O}_X^{\perp}$. Also note that \begin{align*}
\mathbb{D}(\mathbb{D}(M))&\cong\mathrm{R}\mathcal{H}om(\mathbb{D}(M),\mathcal{O}_X(-H))[1]\\
&\cong\mathrm{R}\mathcal{H}om(\mathrm{R}\mathcal{H}om(M,\mathcal{O}_X(-H))[1],\mathcal{O}_X(-H))[1]\\&\cong\mathrm{R}\mathcal{H}om(\mathrm{R}\mathcal{H}om(M,\mathcal{O}_X)\otimes\mathcal{O}_X(-H)[1],\mathcal{O}_X(-H))[1]\\
&\cong \mathrm{R}\mathcal{H}om(M^{\vee}[1],\mathcal{O}_X)[1]\\
&\cong (M^{\vee})^{\vee}\\
&\cong M
\end{align*}
The first and second isomorphism follows from the definition of $\mathbb{D}$. The third isomorphism and the last isomorphisms follow from the property of derived dual $M\mapsto M^{\vee}:=\mathrm{R}\mathcal{H}om(M, \mathcal{O}_X)$. This implies $\mathbb{D}\circ\mathbb{D}\cong\mathrm{Id}$. Then $\mathbb{D}$ is an involutive auto-equivalence of $\mathcal{O}_X^{\perp}$. 
\end{proof}

\subsection{Dimension of Hilbert scheme}
In this section, we find the dimension of $\mathcal{H}$ by computing $\mathrm{ext}^1(I_C,I_C)$, which is the dimension of tangent space at the point $[I_C]$. We recall a very useful spectral sequence below.

\begin{lemma}\cite[Lemma 2.43]{pirozhkov2020admissible}
\label{lemma_SS}
Let $X$ be a smooth algebraic variety. Suppose that there are two distinguished triangle in $D^b(X)$:
$$A_1\rightarrow B_1\rightarrow C_1\rightarrow A_1[1], A_2\rightarrow B_2\rightarrow C_2\rightarrow A_2[1]$$  There exists a $E_1$-spectral sequence which degenerates at $E_3$ and converges to $\mathrm{Ext}^{\bullet}(C_1,C_2)$:$$E^{p,q}_1=\begin{cases}
\mathrm{Ext}^q(B_1,A_2), p=-1\\
\mathrm{Ext}^q(A_1,A_2)\oplus\mathrm{Ext}^q(B_1,B_2), p=0\\
\mathrm{Ext}^q(A_1,B_2), p=1\\
0, p\leq -2, p\geq 2. \end{cases}$$
The differential $d_1$ is given by compositions with morphisms $A_1\rightarrow B_1$ and $A_2\rightarrow B_2$. 
\end{lemma}

\begin{lemma}
\label{lemma_computing_extG}
Let $G$ be the cone of $s:\mathcal{E}\rightarrow I_C$ defined in Proposition~\ref{Prop_computingcone}. Then $\mathrm{Ext}^{\bullet}(G,G)=k\oplus k[-1]$ if $\mathcal{N}_{L|X}=\mathcal{O}_L\oplus\mathcal{O}_L(-1)$ and $\mathrm{Ext}^{\bullet}(G,G)=k\oplus k^2[-1]\oplus k[-2]$ if $\mathcal{N}_{L|X}=\mathcal{O}_L(1)\oplus\mathcal{O}_L(-2)$. 
\end{lemma}

\begin{proof}
Note that  $G\in\mathcal{O}_X^{\perp}$ since $\mathrm{RHom}(\mathcal{O}_X,\mathcal{E})=\mathrm{RHom}(\mathcal{O}_X,I_C)=0$. Then $\mathrm{Ext}^{\bullet}(G,G)\cong\mathrm{Ext}^{\bullet}(D(G),D(G))\cong\mathrm{Ext}^{\bullet}(I_L, I_L)$ by Proposition~\ref{prop_deriveddual_line} and Lemma~\ref{lemma_involutive_D}. If $\mathcal{N}_{L|X}=\mathcal{O}_L\oplus\mathcal{O}_L(-1)$, then $[L]$ is a smooth point in Hilbert scheme $\Sigma(X)$ of lines. The tangent space at $[L]$ is isomorphic to  $\mathrm{Ext}^1(I_L,I_L)=k$ and $\mathrm{Ext}^2(I_L,I_L)=0$. If $\mathcal{N}_{L|X}=\mathcal{O}_L(1)\oplus\mathcal{O}_L(-2)$, then $[L]$ is a singular point in $\Sigma(X)$ and $\mathrm{Ext}^1(I_L,I_L)=k^2, \mathrm{Ext}^2(I_L,I_L)=k$. It is clear that $\mathrm{Hom}(I_L,I_L)=k$ and $\mathrm{Ext}^3(I_L,I_L)=0$ by slope-stability of $I_L$. 
\end{proof}





 \begin{predl}
 \label{Prop_extgroup}
 Let $C$ be a twisted cubic on a special Gushel-Mukai threefold $X$ with a general branch locus, then $$\mathrm{Hom}^i(I_C,I_C)=\begin{cases}
 k, i=0\\
 k^3, i=1\\
 0, i=2\\
 0, i=3\end{cases}$$.
 \end{predl}

\begin{proof}
\leavevmode \begin{enumerate}
    \item If $I_C\in\mathcal{K}u(X)$, then $\mathrm{Ext}^2(I_C,I_C)\cong\mathrm{Hom}^2_{\mathcal{K}u(X)}(I_C,I_C)\cong\mathrm{Hom}(I_C,\tau(I_C))$ since $\tau\circ[2]$ is the Serre functor($\mathcal{K}u(X)$ \cite[Prop 2.6]{kuznetsov2018derived}) and $\tau$ is induced by geometric involution associated with double cover. Thus $\tau(I_C)\cong I_{\tau C}$. Since $X$ is general, by Corollary~\ref{corollary_twistedcubic_notfix} $\tau(C)\neq C$. Then $\mathrm{Hom}(I_C, \tau(I_C))=0$ by slope stability. It is clear that $\mathrm{Hom}(I_C,I_C)=k$ and $\mathrm{Ext}^3(I_C,I_C)\cong\mathrm{Hom}(I_C,I_C\otimes\mathcal{O}_X(-H))=0$ by slope stability. The Euler characteristic $\chi(I_C,I_C)=\mathrm{hom}(I_C,I_C)-\mathrm{ext}^1(I_C,I_C)+\mathrm{ext}^2(I_C,I_C)-\mathrm{ext}^3(I_C,I_C)$. Then $\mathrm{Ext}^1(I_C,I_C)=k^3$. 
    \item If $I_C\not\in\mathcal{K}u(X)$. Then we have the exact triangle $G[-1]\rightarrow\mathcal{E}\rightarrow I_C$ given in Proposition~\ref{Prop_computingcone}. We compute $\mathrm{Ext}^2(I_C,I_C)$ via the spectral sequence in Lemma~\ref{lemma_SS}. Then the first page is given by $$E^{p,q}_1=\begin{cases}
\mathrm{Ext}^3(\mathcal{E},G[-1])=\mathrm{Ext}^2(\mathcal{E},G), p=-1\\
\mathrm{Ext}^2(G,G)\oplus\mathrm{Ext}^2(\mathcal{E},\mathcal{E}), p=0\\
\mathrm{Ext}^2(G,\mathcal{E}), p=1\\
0, p\leq -2, p\geq 2. \end{cases}$$
Apply $\mathrm{Hom}(\mathcal{E},-)$ to the triangle $\mathcal{O}_X(-H)[1]\rightarrow G\rightarrow\mathcal{O}_{L}(-2)$, we get $\mathrm{Ext}^2(\mathcal{E},G)=0$. Apply $\mathrm{Hom}(-,\mathcal{E})$ to the same triangle, we get a long exact sequence: $$0\rightarrow\mathrm{Ext}^1(G,\mathcal{E})\rightarrow\mathrm{Ext}^1(\mathcal{O}_X(-H)[1],\mathcal{E})\rightarrow\mathrm{Ext}^2(\mathcal{O}_L(-2),\mathcal{E})\rightarrow\mathrm{Ext}^2(G,\mathcal{E})\rightarrow 0$$

Note that $\mathrm{Ext}^1(\mathcal{O}_X(-H)[1],\mathcal{E})=\mathrm{Hom}(\mathcal{O}_X(-H),\mathcal{E})=\mathrm{Hom}(\mathcal{O}_X,\mathcal{E}^{\vee})=H^0(\mathcal{E}^{\vee})=k^5$ and $\mathrm{Ext}^2(\mathcal{O}_L(-2),\mathcal{E})$ $\cong H^0(\mathcal{E}^{\vee}|_L)=k^3$. Then $\mathrm{ext}^1(G,\mathcal{E})-\mathrm{ext}^2(G,\mathcal{E})=2$, so $ext^1(G,\mathcal{E})\geq 2$. We claim that $\mathrm{ext}^1(G,\mathcal{E})=2$. It is clear that $\mathrm{Ext}^1(G,\mathcal{E})$ is the set of sections $s$ of $\mathcal{E}^{\vee}$ such that $L$ is contained in the zero locus of $s$. If $\mathrm{ext}^1(G,\mathcal{E})>2$, then $L$ would be contained in $\pi$-preimage of a point, which is impossible. Thus $\mathrm{ext}^1(G,\mathcal{E})=2$. Then $\mathrm{Ext}^2(G,\mathcal{E})=0$. By Lemma~\ref{lemma_computing_extG},  $\mathrm{Ext}^2(G,G)=0$ and  $\mathrm{Ext}^2(\mathcal{E},\mathcal{E})=0$ since $\mathcal{E}$ is an exceptional bundle. Then $E^{p,q}_1=0$ for all $p,q$ with $p+q=2$. Thus $\mathrm{Ext}^2(I_C,I_C)=0$.  But $\chi(I_C,I_C)=-2$ and $\mathrm{Hom}(I_C,I_C)=k, \mathrm{Ext}^3(I_C,I_C)=0$ by slope stability. 
Therefore $\mathrm{Ext}^1(I_C,I_C)=k^3$.
\end{enumerate}

\end{proof}

\begin{predl}\label{lemma_connctedness_Hilbertscheme}
Let $X$ be a special Gushel-Mukai threefold with general branch locus $\mathcal{B}$, then Hilbert scheme $\mathcal{H}$ of twisted cubics is connected.
\end{predl}

\begin{proof}
Let $\pi:X\rightarrow Y$ be the double cover branched over $\mathcal{B}$. The Hilbert scheme $\mathcal{H}$  has a stratification $\mathcal{H}=\mathcal{H}_1\bigcup\mathcal{H}_2$. 
Every twisted cubic in $\mathcal{H}_2$ is contained in the $\pi$-preimage of a reducible conic on $Y$, then it is a union of a line and a conic and $\mathcal{H}_1$ is the double cover of (closure of) family of twisted cubics on $Y$ which are triple tangent to $\mathcal{B}$. It is easy to see $\mathcal{H}_1$ is connected. Also note that $\mathcal{H}_2$ is connected since the Fano scheme $\Sigma(X)$ of lines and Fano scheme of conics $\mathcal{C}(X)$ are connected by \cite{Iliev94line} and \cite{Iliev94conic}. We claim that $\mathcal{H}_2$ is not a connected component. As we know that $\Sigma(X)$ and $\mathcal{C}(X)$ are of dimension 1 and 2 respectively and a line intersects a conic is a codimension one condition. Thus $\mathcal{H}_2$ is of dimension two. On the otherhand, for every twisted cubic $C\in\mathcal{H}_2$, $\mathrm{ext}^1(I_C,I_C)=3$ and $\mathrm{ext}^2(I_C,I_C)=0$ by Proposition~\ref{Prop_extgroup}. So if $\mathcal{H}_2$ were a connected component, then it is an irreducible component and its dimension is $3$, which lead a contradiction. This implies that at least one twisted cubic in $\mathcal{H}_2$ (hence everyone) is deformed from the one in $\mathcal{H}_1$. Therefore $\mathcal{H}$ is connected. 
\end{proof}
    
\begin{theorem}
\label{theorem_projectivity_hilbertscheme}
Let $X$ be a special Gushel-Mukai threefold with a general branch locus. Then the Hilbert scheme $\mathcal{H}$ of twisted cubics is a smooth irreducible projective threefold. 
\end{theorem}

\begin{proof}
Let $\mathcal{C}$ be the universal subscheme for $\mathcal{H}$. Note that $\mathcal{H}$ can be identified with moduli space $\mathcal{M}$ of ideal sheaves of twisted cubics by \cite[Lemma B.5.6]{KPS2018}. By Lemma~\ref{GiesekerHilb}, it is a moduli space of stable sheaves with character  $1-3L+\frac{1}{2}P$. The universal subscheme $\mathcal{C}$ is sent to universal ideal sheaf $\mathcal{I}$ on moduli space, thus $\mathcal{M}$ is a fine moduli space represented by a projective scheme. The tangent space at each point $[I_C]$ is given by $\mathrm{Ext}^1(I_C,I_C)$. By Proposition~\ref{Prop_extgroup}  $\mathrm{ext}^1(I_C,I_C)=3$ and  $\mathrm{ext}^2(I_C,I_C)=0$ for each $C$. Thus $\mathcal{H}$ is smooth. Moreover $\mathcal{H}$ is connected by Lemma~\ref{lemma_connctedness_Hilbertscheme}, therefore $\mathcal{H}$ is smooth projective of dimension three. 
\end{proof}

\begin{remark}\label{Lemma_H2_isa_ruledsurface}
In fact the strata $\mathcal{H}_2$ is a smooth ruled surface. Indeed, suppose $[C]\in\mathcal{H}_2$, then $\pi(C)$ is a reducible conic $l\cup l'\subset Y$, where $l,l'$ are two lines and $l$ is tangent to the branch locus $\mathcal{B}$. Fix line $l$ and let $l_t'$ varies such that $l'\cap l\neq\emptyset$. By Proposition~\ref{proposition_tangentlocus}, $l_t$ form a $\mathbb{P}^1$. After pulling back under $\pi:X\rightarrow Y$,  by Proposition~\ref{proposition_lines_Y}, we get a family of twisted cubics $C_t=L+M_t$ parametrised by $\mathbb{P}^1$, where $L+\tau(L)=\pi^{-1}(l)$ and $M_t=\pi^{-1}(l'_t)$. Note that the curve $\Gamma$ of all lines tangent to $\mathcal{B}$ is a degree 10 planar curve by Proposition~\ref{proposition_lines_Y}. Thus the construction above provides a $\mathbb{P}^1$ family of twisted cubics over $\Sigma(X)$(as an unramified double cover of $\Gamma$), which is a smooth ruled surface over a smooth curve of genus 71. 
\end{remark}



\section{Bridgeland stability conditions on the Kuznetsov component}
In this section we recall the definition of (weak) stability conditions on triangulated category and use \cite[Proposition 5.1, Proposition 6.9]{bayer2017stability} to induce stability conditions on the Kuznetsov components $\mathcal{K}u(X)$ on Gushel-Mukai threefolds. Section $4.1$ can be found in \cite[Section 2]{pertusi2020some} and \cite[Appendix]{altavilla2019moduli}.

\subsection{Weak stability condition}
\label{section4}
A (weak) stability condition on a triangulated category $\mathcal{T}$ is given by a heart of a bounded t-structure and a (weak) stability function:
\begin{definition}
A heart of a bounded $t$-structure on $\mathcal{T}$ is a full subcategory $\mathcal{A}\subset \mathcal{T}$ such that
\begin{enumerate}
\item[(i)]  for $E, F\in \mathcal{A}$ and $k<0$ we have $\Hom(E, F[k])=0$, and
\item[(ii)]  for every object $E\in \mathcal{T}$ there is a sequence of morphisms
$$
\xymatrix@C=0.5cm{
  0 =E_0 \ar[r]^{\phi_{1}}& E_{1} \ar[r]^{} & \cdots  \ar[r]^{\phi_{m}\;\;\;\;\;} & E_{m} =E }
$$
such that $\mathrm{Cone}(\phi_{i})$ is of the form $A_{i}[k_i]$ for some sequence $k_1>k_2>\cdots>k_m$
of integers and objects $A_i\in \mathcal{A}$.
\end{enumerate}
\end{definition}

\begin{definition}
\label{def_weakstab}
Let $\mathcal{A}$ be an abelian category. A  weak stability function on $\mathcal{A}$ is a group homomorphism
$$
\begin{array}{cccl}
Z:& K(\mathcal{A})&\longrightarrow& \C \\
&E& \longmapsto& \Re Z(E) +i \Im Z(E),
\end{array}
$$
where $K(\mathcal{A})$ denotes the Grothendieck group of $\mathcal{A}$, such that for every non-zero object $E\in \mathcal{A}$, we have
$$
\Im Z(E)\geq 0, \; \textrm{and}\; \Im Z(E)= 0 \Rightarrow \Re Z(E)\leq 0.
$$
We say that $Z$ is a stability function on $\mathcal{A}$ if in addition for $E \neq 0$, $\Im Z(E)= 0$ implies $\Re Z(E)< 0$.
\end{definition}

\begin{definition}
A weak stability condition on $\mathcal{T}$ with respect to a finite rank lattice $\Lambda$ is a pair $\sigma=(\mathcal{A},Z)$ where $\mathcal{A}$ is the heart of a bounded $t$-structure and $Z:\Lambda\rightarrow\mathbb{C}$ is a group homomorphism such that the following conditions holds:
\begin{enumerate}
    \item The composition $K(\mathcal{A})=K(\mathcal{T})\xrightarrow{v}\Lambda\xrightarrow{Z}\mathbb{C}$ is a weak stability function on $\mathcal{A}$: denote $Z(E):=Z(v(E))$. The slope associated to $Z$ for any object $E\in\mathcal{A}$ is defined by $$
\mu_{\sigma}(E)
:=
\begin{cases}
-\frac{\Re Z(E)}{\Im Z(E)} & \textrm{if}\, \Im Z(E)>0; \\
+\infty &  \textrm{otherwise};
\end{cases}
$$
An object $0\neq E\in\mathcal{A}$ is called $\sigma$-semistable (resp. $\sigma$-stable) if for every non-zero proper subobject $F\subset E$, we have $\mu_{\sigma}(F)\leq\mu_{\sigma}(E)$ (resp. $\mu_{\sigma}(F)<\mu_{\sigma}(E)$).
\item  (HN-filtrations) Any object of $\mathcal{A}$ has a Harder-Narasimhan filtration in $\sigma$-semistable ones.

\item (Support property) There is a quadratic form $Q$ on $\Lambda\otimes \mathbb{R}$ such that $Q|_{\ker Z}$ is negative definite, and $Q(E)\geq 0$ for all $\sigma$-semistable objects $E\in \mathcal{A}$
\end{enumerate}
\end{definition}

\begin{definition}
A weak stability condition $\sigma=(\mathcal{A},Z)$ on $\mathcal{T}$ with respect to the lattice $\Lambda$ is called a Bridgeland stability condition if $Z$ is a stability condition.
\end{definition}

\begin{definition}
\label{def_slicing}
The phase of a $\sigma$-semistable object $E \in \mathcal{A}$ is
$$\phi(E):=\frac{1}{\pi}\text{arg}(Z(E)) \in (0,1]$$
and for $F=E[n]$, we set
$$\phi(E[n]):=\phi(E)+n.$$
A slicing $\mathcal{P}$ of $\mathcal{T}$ is a collection of full additive subcategories $\mathcal{P}(\phi) \subset \mathcal{T}$ for $\phi \in \mathbb{R}$, such that:
\begin{enumerate}
\item  for $\phi \in (0,1]$, the subcategory $\mathcal{P}(\phi)$ is given by the zero object and all $\sigma$-semistable objects with phase $\phi$;
\item  for $\phi+n$ with $\phi \in (0,1]$ and $n \in \mathbb{Z}$, we set $\mathcal{P}(\phi+n):=\mathcal{P}(\phi)[n]$.
\end{enumerate}
\end{definition}

\noindent We will both use the notation $\sigma=(\mathcal{A},Z)$ and $\sigma=(\mathcal{P},Z)$ for a (weak) stability condition with heart $\mathcal{A}=\mathcal{P}((0,1])$, where $\mathcal{P}$ is a slicing.

We denote by $\mathrm{Stab}(\mathcal{T})$ the set of stability conditions on $\mathcal{T}$. 
Recall that the universal covering space $\widetilde{\mathrm{GL}}^+_2(\mathbb{R})$ of $\mathrm{GL}^+_2(\mathbb{R})$ has a right action on $\mathrm{Stab}(\mathcal{T})$, defined as follows. For $\tilde{g}=(g,M) \in \widetilde{\mathrm{GL}}^+_2(\mathbb{R})$, where $g: \mathbb{R} \to \mathbb{R}$ is an increasing function such that $g(\phi+1)=g(\phi)+1$ and $M \in \mathrm{GL}^+_2(\mathbb{R})$, and $\sigma=(\mathcal{P},Z) \in \mathrm{Stab}(\mathcal{T})$, we have that $\sigma \cdot \tilde{g}=(\mathcal{P}',Z')$ is a stability condition with $Z'=M^{-1} \circ Z$ and $\mathcal{P}'(\phi)=\mathcal{P}(g(\phi))$. The group of autoequivalences $\mathrm{Aut}(\mathcal{T})$ of $\mathcal{T}$ acts on the left of $\mathrm{Stab}(\mathcal{T})$ by $\Phi \cdot \sigma=(\Phi(\mathcal{P}), Z \circ \Phi_*^{-1})$, where $\Phi \in \mathrm{Aut}(\mathcal{T})$ and $\Phi_*$ is the automorphism of $K(\mathcal{T})$ induced by $\Phi$.

To construct Bridgeland stability conditions, we usually start with a weak stabilty condition $\sigma=(\mathcal{A},Z)$ on $\mathcal{T}$ and then by "tilting" $\mathcal{A}$ to get a new heart and define a new weak stability funtion. Let $\mu \in \R$; we define the following subcategories of $\mathcal{A}$:
\begin{align*}
\mathcal{T}^{\mu}_{\sigma}&:=  E \in \mathcal{A}: \text{all HN factors } F \text{ of }E \text{ have slope } \mu_{\sigma}(F)>\mu \rbrace\\
                  & = \langle E \in \mathcal{A}: E \text{ is } \sigma\text{-semistable with }\mu_{\sigma}(E) >\mu \rangle
\end{align*}
and
\begin{align*}
\mathcal{F}^{\mu}_{\sigma} &:=  E \in \mathcal{A}: \text{all HN factors } F \text{ of }E \text{ have slope } \mu_{\sigma}(F)\leq\mu \rbrace\\
                   &= \langle E \in \mathcal{A}: E \text{ is } \sigma\text{-semistable with }\mu_{\sigma}(E) \leq \mu \rangle.
\end{align*}
Where $\langle ,\rangle$ means the extension closure.

\begin{predl}[\cite{happel1996tilting}]
\label{prop_HRS}
The category
$$\mathcal{A}_{\sigma}^{\mu}:= \langle \mathcal{T}^{\mu}_{\sigma},\mathcal{F}^{\mu}_{\sigma}[1] \rangle$$
is the heart of a bounded t-structure on $\mathcal{T}$.
\end{predl}

\noindent We say that the heart $\mathcal{A}_{\sigma}^{\mu}$ is obtained by tilting $\mathcal{A}$ with respect to the weak stability condition $\sigma$ at the slope $\mu$.

Thus, if we start with the weak stability condition given by slope stability on $\mathcal{A}=\mathrm{Coh}(X)$ we construct new hearts $\mathrm{Coh}^{\beta}(X)$ for $\mu=\beta$. To define the new weak stability conditions on these hearts, we need to have new stability functions:
Recall the twisted Chern character $\mathrm{ch}^{\beta}(E):=e^{-\beta}\mathrm{ch}(E)$. The first three terms are:
$$\mathrm{ch}_0^{\beta}(E):=\mathrm{ch}_0(E), \quad  \mathrm{ch}_1^{\beta}(E):=\mathrm{ch}_1(E)-\beta H \mathrm{ch}_0(E)$$
and
$$\mathrm{ch}_2^{\beta}(E):= \mathrm{ch}_2(E) -\beta H \mathrm{ch}_1(E) +\frac{\beta^2 H^2}{2} \mathrm{ch}_0(E).$$

\begin{predl}[\cite{bayer2017stability}, Proposition 2.12]
\label{first-tilting-wsc}
For any $(\alpha, \beta)\in \mathbb{R}_{>0}\times \mathbb{R}$,
the pair
$$\sigma_{\alpha, \beta}=(\mathrm{Coh}^{\beta}(X), Z_{\alpha, \beta})$$
with
$$
Z_{\alpha, \beta}(E)
:=\frac{1}{2}\alpha^2 H^{n}\mathrm{ch}_{0}^{\beta}(E)-H^{n-2}\mathrm{ch}_{2}^{\beta}(E)
+i H^{n-1}\mathrm{ch}_{1}^{\beta}(E)
$$
defines a weak stability condition on ${D}^b(X)$ with respect to $\Lambda_{H}^{2}$.
The quadratic form $Q$ can be given by the discriminant $\Delta_{H}(E):=(H^{n-1}\mathrm{ch}_{1}(E))^{2}-2H^{n}\mathrm{ch}_{0}(E)\cdot H^{n-2}\mathrm{ch}_{2}(E)$.

\end{predl}

In \cite{bayer2017stability}, the authors take the second tilt of $\mathrm{Coh}^{\beta}(X)$: fix a slope $\mu\in\mathbb{R}$ and define $\mathrm{Coh}^{\mu}_{\alpha,\beta}(X):=\mathrm{Coh}^{\beta}(X)^{\mu,\sigma_{\alpha,\beta}}$ with a stability function
$$Z^{\mu}_{\alpha,\beta}:=\frac{1}{u}Z_{\alpha,\beta}$$
where $u\in\mathbb{C}$ is the upper half plane with  $\mu=-\frac{\Re u}{\Im u}$.

\begin{predl}\cite{bayer2017stability}
The pair $\sigma^{\mu}_{\alpha,\beta}=(\mathrm{Coh}^{\beta}(X)^{\mu,\sigma_{\alpha,\beta}}, Z^{\mu}_{\alpha,\beta})$ is a weak stability condition on $D^b(X)$.
\end{predl}

Eventually, \cite[Proposition 5.1, Theorem 6.9]{bayer2017stability} enables us to define a stability condition on the Kuznetsov components of Fano threefolds of index one, in particular, special Gushel-Mukai threefolds.

\begin{theorem}
\label{Theorem_existence_stability}
Let $X$ be a Fano threefold of index $1$ and even genus $g\geq 6$ with $\mathrm{Pic}(X)=\mathbb{Z}$ and let $\sigma^0_{\alpha,\beta}$ be the weak stability condition constructed above for $\mu=0$, $\beta=-1+\epsilon$ and $0<\alpha<\epsilon$ with $\epsilon$ sufficiently small. Let $\mathcal{A}=\mathcal{K}u(X)\cap\mathrm{Coh}^0_{\alpha,\beta}(X)$ and $Z=Z^0_{\alpha,\beta}$. Then the pair $(\mathcal{A},Z)$ is a Bridgeland stability condition on $\mathcal{K}u(X)$.
\end{theorem}

.
\begin{definition}
\label{def_wall}
Let $v\in\Lambda^2_H\cong\mathbb{Z}^3$.\begin{enumerate}
    \item A numerical wall for $v$ is the set of pairs $(\alpha,\beta)\in\mathbb{R}_{>0}\times\mathbb{R}$ such that there is a vector $w\in\Lambda^2_H$ such that $\mu_{\alpha,\beta}(v)=\mu_{\alpha,\beta}(w)$.
    \item A wall for $F\in\mathrm{Coh}^{\beta}(X)$ is a numerical wall for $v:=ch_{\leq 2}(F)$ such that for every $(\alpha,\beta)$ on the wall there is an exact sequence of semistable objects $0\rightarrow E\rightarrow F\rightarrow G\rightarrow 0$ in $\mathrm{Coh}^{\beta}(X)$ such that $\mu_{\alpha,\beta}(F)=\mu_{\alpha,\beta}(E)=\mu_{\alpha,\beta}(G)$ which gives the numerical wall.
\end{enumerate}
\end{definition}

\subsection{Serre invariant stability conditions}
By Theorem~\ref{Theorem_existence_stability}, the stability condition $\sigma(\alpha,\beta)=(\mathcal{A},Z)$ is a stability condition on Kuznetsov component $\mathcal{K}u(X)$ of a special Gushel-Mukai threefold $X$. 

\begin{definition}\label{definition_S_invariant_stab_condition}
A stability condition $\sigma$ defined on $\mathcal{K}u(X)$ is called Serre-invariant if $S_{\mathcal{K}u(X)}\cdot\sigma=\sigma\cdot\tilde{g}$ for $\tilde{g}\in\widetilde{GL}^+_2(\mathbb{R})$. In particular, if $S_{\mathcal{K}u(X)}\cong\tau\circ[2]$ for an involutive auto-equivalence, then $\sigma$ is called $\tau$-invariant and $\tau(\sigma)=\sigma$. Fix $\alpha$ and $\beta$, define $\mathcal{K}:=\sigma(\alpha,\beta)\cdot\widetilde{GL}^+_2(\mathbb{R})\subset\mathrm{Stab}(\mathcal{K}u(X))$, called the orbit of $\sigma(\alpha,\beta)$. 
\end{definition}

\begin{predl}
\label{prop_heart_GM}
Let $X$ be a smooth special Gushel-Mukai threefold, let $\tau$ be the involution of $\mathcal{K}u(X)$, then \begin{enumerate}
    \item $\tau$ fixes the closure of $\mathcal{K}$ in $\mathrm{Stab}(\mathcal{K}u(X))$.
    \item For every $\sigma=(\mathcal{A},Z)$ in the closure of $\mathcal{K}$, the heart $\mathcal{A}$ has homological dimension $2$.
    \item For any nonzero object $A$ in $\mathcal{A}$, $\mathrm{hom}^1(A,A)\geq 2$.
\end{enumerate}
\end{predl}

 \begin{proof}\leavevmode \begin{enumerate}
    \item The argument for $(1)$ is essentially the same as \cite[Lemma 5.11]{pertusi2020some}. Note that $\tau$ fixes hyperplane section $H$ and it acts trivially on the characters of objects in $D^b(X)$ and their slopes. Also, $\tau$ preserves the coherent sheaves and $\tau$ is an exact functor thus preserves the inclusion of sheaves. Then it preserves the stability of an object. By construction of the tilted heart and the second tilted heart, $\tau$ preserves both hearts. Hence $\tau\mathcal{A}=\mathcal{A}$ since $\tau$ also preserves $\mathcal{K}u(X)$. Then $\tau\sigma(\alpha,\beta)=\sigma(\alpha,\beta)$. $\tau$ preserves $\mathcal{K}$ since the action of $\widetilde{\text{GL}}_2^+(\R)$ commutes with autoequivalences. Also note that the action of $\tau$ on the stability manifold is continuous, the desired result follows.
    \item As $\mathcal{A}$ is the heart of the stability condition $\sigma(\alpha,\beta)$, then $\mathrm{Ext}^i(E,F)=0$ for $i<0$. If $i\geq 3$, then $\mathrm{Ext}^i(E,F)=\mathrm{Hom}(E,F[i])=\mathrm{Hom}(F[i],S_{\mathcal{K}u(X)}(E))=\mathrm{Ext}^{2-i}(F,\tau E)=0$ since  $\tau(E)\in\mathcal{A}$ and $2-i<0$. 
    \item Note that $[A]\in\mathcal{N}(\mathcal{K}u(X))$, then $[A]=as+bt$ for $a,b\in\mathbb{Z}$. Then $\chi(A,A)=-(2a^2+6ab+5b^2)=-(a+b)^2-(a+2b)^2$. Since $A$ is a nonzero object in the heart $\mathcal{A}$, $a,b$ can not be zero simultaneously. Then $\chi(A,A)\leq -1$. It follows that $\mathrm{hom}^1(A,A)\geq 2$. 

\end{enumerate}
\end{proof}

By the same argument in \cite[Lemma 5.14]{pertusi2020some}, we have:
\begin{lemma}[Weak Mukai Lemma for GM threefolds]
\label{lemma_Mukailemma}
Let $A\rightarrow E\rightarrow B$ be a triangle in $\mathcal{K}u(X)$ with $\mathrm{Hom}(A,B)=\mathrm{Hom}(A,\tau(B))=0$. Then $$\mathrm{hom}^1(A,A)+\mathrm{hom}^1(B,B)\leq \mathrm{hom}^1(E,E)$$
\end{lemma}

\subsection{Stability of objects with small $\mathrm{Ext}^1$}
In this section, we prove an object $F$ in the Kuznetsov component $\mathcal{K}u(X)$ of a special Gushel-Mukai threefold $X$ is $\sigma$-stable for every $\tau$-invariant stability condition $\sigma$ if $\mathrm{ext}^1(F,F)$ is small. 

\begin{lemma}
\label{lemma_ext3_inheart}
Let $F\in\mathcal{K}u(X)$ with $\mathrm{hom}^1(F,F)=3$ or $2$, then $F\in\mathcal{A}$ up to shifts. 
\end{lemma}

\begin{proof}
The statement follows from \cite[Lemma 3.13]{altavilla2019moduli}. Indeed, consider the spectral sequence for objects in $\mathcal{K}u(X)$ whose second page is 
$$E^{p,q}_2=\bigoplus_i\mathrm{Hom}^p(\mathcal{H}^i(F),\mathcal{H}^{i+q}(F))\Longrightarrow\mathrm{Hom}^{p+q}(F,F),$$ where the cohomology is taken with respect to the heart $\mathcal{A}$. By Proposition~\ref{prop_heart_GM}, the homological dimension of $\mathcal{A}$ is $2$, it follows that $E^{1,q}_{\infty}=E^{1,q}_2$, so if we take $q=0$, we have $$3=\mathrm{hom}^1(F,F)=\mathrm{dim}(\bigoplus_{p+q=1}E^{p,q}_{\infty})\geq \mathrm{dim}(E^{1,0}_{\infty})=\bigoplus_{i}\mathrm{hom}^1(\mathcal{H}^i(F),\mathcal{H}^i(F))\geq 2r, $$ 
where $r>0$ is the number of non-zero cohomology objects of $F$. This is because if $\mathcal{H}^i(F)\neq 0$, then by Proposition~\ref{prop_heart_GM}, $\mathrm{hom}^1(\mathcal{H}^i(F),\mathcal{H}^i(F))\geq 2$. Then we have $r=1$, then $F[k]\in\mathcal{A}$ for some $k$. The same argument applies if $\mathrm{ext}^1(F,F)=2$. 
\end{proof}

\begin{corollary}
\label{lemma_ext3_stable}
Let $F\in\mathcal{A}$ with $\mathrm{ext}^1(F,F)=3$ and $\chi(F,F)=-2$, then $F$ is $\sigma$-stable. 
\end{corollary}

\begin{proof}
Assume that we already show that $F$ is $\sigma$-semistable. Now we prove $F$ is $\sigma$-stable. To prove this, we apply weak Mukai Lemma~\ref{lemma_Mukailemma}. The proof is a complete analogue of \cite[Theorem 1.2]{pertusi2020some}: Assume that $F$ is strictly $\sigma$-semistable of phase $\phi$. If $F$ has only one stable factor $C'$ up to isomorphism, then we have $\chi(F,F)=n^2\chi(C',C')=-2$ for $n\geq 2$, then we have a contradiction. Assume that $F$ has more than one stable factor up to isomorphism. Then we can write a sequence in $\mathcal{P}(\phi)$ of the form $$0\rightarrow A\rightarrow F\rightarrow B\rightarrow 0$$ where $A,B$ are $\sigma$-semistable and $B$ has only one stable factor $C$ up to isomorphism such that  $\mathrm{Hom}(A,B)=0$. If $\mathrm{Hom}(A,\tau(B))=0$, then by Lemma~\ref{lemma_Mukailemma},  $$\mathrm{hom}^1(A,A)+\mathrm{hom}^1(B,B)\leq\mathrm{hom}^1(F,F)=3.$$  But $A,B$ are both non-zero objects in $\mathcal{A}$, then $\mathrm{hom}^1(A,A)\geq 2$ and $\mathrm{hom}^1(B,B)\geq 2$, then this is a contradiction, which means that $F$ is $\sigma$-stable. Assume that $\mathrm{Hom}(A,\tau(B))\neq 0$. Then we have $\mathrm{Hom}(A,\tau(C))\neq 0$. Note that $\tau(C)$ is a quotient of $A$. Thus we have a sequence $D\rightarrow A\rightarrow E$, where $D$ and $F$ are $\sigma$-semistable, the only stable factor of $E$ is $\tau(C)$, and $C,\tau(C)$ are not stable factors of $D$. Then we have the following commutative diagram:
$$
\xymatrix{
D \ar[d]^{\id} \ar[r] & A \ar[d] \ar[r] & E \ar[d]\\
D \ar[d] \ar[r] & F \ar[d] \ar[r] & G \ar[d] \\
0 \ar[r] & B \ar[r]^{\id} & B,
}
$$

where $G$ is $\sigma$-semistable and whose stable factors are isomorphic to $C$ or $\tau(C)$. Thus we have $\mathrm{Hom}(D,G)=\mathrm{Hom}(D,\tau(G))=0$ since the stability condition $\sigma$ is $\tau$-invariant. Then by Mukai Lemma~\ref{lemma_Mukailemma} again, we have 
$$\mathrm{hom}^1(D,D)+\mathrm{hom}^1(G,G)\leq\mathrm{hom}^1(F,F)=3$$ but $\mathrm{hom}^1(D,D)\geq 2$ and $\mathrm{hom}^1(G,G)\geq 2$, another contradiction arrived.  Thus $\mathrm{Hom}(A.\tau(B))=0$, which completes the proof of $\sigma$-stability of $F$. Now, we prove $F$ is $\sigma$-semistable. Assume that $F$ is $\sigma$-unstable, then there is a destablizing triangle $A\rightarrow F\rightarrow B$ such that $A$ and $B$ are $\sigma$-semistable and $\phi(A)>\phi(B)$. Then $\mathrm{Hom}(A,B)=0$ and $\mathrm{Hom}(A,\tau(B))=0$ since $\tau(B)\in\mathcal{A}$ and $\tau(B)$ is still $\sigma$-semistable with the same phase as $B$. Then by Mukai Lemma~\ref{lemma_Mukailemma}, $\mathrm{ext}^1(A,A)+\mathrm{ext}^1(B,B)\leq\mathrm{ext}^1(E,E)=3$, then we have another contradiction. Then $E$ is $\sigma$-semistable. 
\end{proof}

\begin{lemma}
\label{lemma_inheart_ext14}
Let $F\in\mathcal{K}u(X)$ with $\mathrm{hom}(F,F)=1,\mathrm{ext}^1(F,F)=4,\mathrm{ext}^2(F,F)=1, \mathrm{ext}^3(F,F)=0$ and $\chi(F,F)=-2$, then $F\in\mathcal{A}$ up to shifts. 
\end{lemma}

\begin{proof}
The argument is entirely similar to \cite[Lemma 3.14]{altavilla2019moduli} and Corollary~\ref{lemma_ext3_stable}. We have $$4=\mathrm{hom}^1(F,F)\geq\bigoplus_{i}\mathrm{hom}^1(\mathcal{H}^i(F),\mathcal{H}^i(F))\geq 2r$$ thus $r=1$ or $r=2$. If $r=1$, then $\mathcal{F}\in\mathcal{A}$ up to shifts. Assume by contradiction that $r=2$, and let $M$ and $N$ be the two non-zero cohomology objects of $F$ in $\mathcal{A}$. Then we have $\mathrm{hom}^1(M,M)\oplus\mathrm{hom}^1(N,N)=4$. But $M,N$ are non-zero objects of $\mathcal{A}$, by Proposition $4.17$, $\mathrm{hom}^1(M,M)=\mathrm{hom}^1(N,N)=2$. Note that $\chi(E,E)\leq -1$ for all objects $E\in\mathcal{K}u(X)$. But $\chi(M,M)=\mathrm{hom}(M,M)-\mathrm{hom}^1(M,M)+\mathrm{hom}^2(M,M)\leq -1$. Thus $\mathrm{hom}(M,M)=1, \mathrm{hom}^2(M,M)=0$. Similarly, $\mathrm{hom}(N,N)=1$ and $\mathrm{hom}^2(N,N)=0$. Also note that $\chi(M,M)=\chi(N,N)=-1$. Since $[M]\in\mathcal{N}(\mathcal{K}u(X))$, assume it is of the form $as+bt$ for some $a,b\in\mathbb{Z}$. Then $\chi(M,M)=-(2a^2+6ab+5b^2)=-1$. Thus $[M]=\pm (t-2s)$ or $\pm (s-t)$. Similarly, $[N]=\pm (t-2s)$ or $\pm (s-t)$. In order for the classes of $M$ and $N$ (with appropriate shifts) add up to give the class of $F$ ($[F]$ is a $-2$-class in $\mathcal{N}(\mathcal{K}u(X))$, by the similar computation, $[F]=\pm s$ or $\pm (3s-2t)$. We can assume that $[F]=\pm s$ and the case of $[F]=\pm (3s-2t)$ are proved in the exactly the same way), and for the spectral sequence give the correct numbers of $\mathrm{hom}^*(F,F)$, the only possiblity is the following: 
$$M=\mathcal{H}^j(F), [M]=\pm (2s-t), N=\mathcal{H}^{j\pm 1}(F), [F]=\pm (s-t)$$ In all these cases, one note that $\chi([M],[N])=0$. Then we gather all these information together and obtain exactly the same table of dimensions of ext-groups in the $E_2$-page of the spectral sequence above to compute $\mathrm{ext}^*(F,F)$ as in \cite[Lemma 3.14]{altavilla2019moduli}. Then by the exact same argument there, we find that there is a contradiction, this means that the cohomology of $F$ can only have one non-zero object. This means that $F\in\mathcal{A}$ up to shifts. 
\end{proof}

\begin{predl}
\label{prop_stability_ext14}
Let $F\in\mathcal{A}$ with 
$$\mathrm{hom}^k(F,F)=\begin{cases} 
1, k=0\\
4, k=1\\
1, k=2\\
0, k=3\end{cases}$$  Then $F$ is $\sigma$-stable. 
\end{predl}

\begin{proof}
Similar to the proof of Corollary~\ref{lemma_ext3_stable}, assume we already show that $F$ is $\sigma$-semistable, then we prove it is $\sigma$-stable. Assume that $F$ is strictly $\sigma$-semistable of phase $\phi$. Then $F$ can not have only one stable factor up to isomorphism by the same argument in Corollary~\ref{lemma_ext3_stable}. Then we assume that $F$ has more than one stable factor up to isomorphism. Then we can write a sequence in $\mathcal{P}(\phi)$ of the form $$0\rightarrow A\rightarrow F\rightarrow B\rightarrow 0$$ where $A,B$ are $\sigma$-semistable and $B$ has only one stable factor $C$ up to isomorphism such that  $\mathrm{Hom}(A,B)=0$. Then if $\mathrm{Hom}(A,\tau(B))=0$, by Mukai Lemma~\ref{lemma_Mukailemma}, we have $$\mathrm{hom}^1(A,A)+\mathrm{hom}^1(B,B)\leq\mathrm{hom}^1(F,F)=4. $$ Then by Proposition~\ref{prop_heart_GM}, we have $\mathrm{hom}^1(A,A)=\mathrm{hom}^1(B,B)=2$. Note that $\chi(A,A)=\mathrm{hom}(A,A)-\mathrm{hom}^1(A,A)+\mathrm{hom}^2(A,A)\geq -1$ but $\chi(A,A)\leq -1$ for any non-trvial object in $\mathcal{A}\subset\mathcal{K}u(X)$. Then we have $\chi(A,A)=-1$ and hence $\mathrm{hom}(A,A)=1$ and $\mathrm{hom}^2(A,A)=0$. Simlarly, we have $\mathrm{hom}(B,B)=1$ and $\mathrm{hom}^2(B,B)=0$ and $\chi(B,B)=-1$. Then by the same computation as in Corollary~\ref{lemma_inheart_ext14}, we get $[A]$ and $[B]$ are $\pm (t-2s)$ or $\pm (s-t)$. Now we assume that $[F]=\pm s$ (The case of $[F]=\pm (3s-2t)$ can be argued in entirely the same way). Then up to sign, $[A]=2s-t$ and $[B]=t-s$. In all these cases, $\chi(A,B)=\chi(B,A)=0$. Let $\mathrm{hom}(B,\tau(A))=c$ and $\mathrm{hom}(B,A)=b$.
Since $\chi(A,B)=\mathrm{hom}(A,B)-\mathrm{hom}^1(A,B)+\mathrm{hom}^2(A,B)=0$ and $\mathrm{hom}(A,B)=0$ by assumption, thus $\mathrm{hom}^1(A,B)=\mathrm{hom}^2(A,B)=\mathrm{hom}(B,\tau(A))=c$. Similarly, $\chi(B,A)=\mathrm{hom}(B,A)-\mathrm{hom}^1(B,A)+\mathrm{hom}^2(B,A)=0$ and $\mathrm{hom}^2(B,A)=\mathrm{hom}(A,\tau(B))=0$ by assumption. Thus we have $\mathrm{hom}^1(B,A)=\mathrm{hom}(B,A)=b$. 
We summarize the numerics of ext-groups in the table below. 
\begin{table}[h]
\begin{center}
\caption{Ext groups of $(A,B), (A,A), (B,B)$}
\label{table_yes}
\begin{tabular}{|p{2cm}|p{2cm}|p{2cm}|p{2cm}|p{2cm}|}
 \hline
   & $\mathrm{hom}^k(A,B)$ & $\mathrm{hom}^k(B,A)$ & $\mathrm{hom}^k(A,A)$ & $\mathrm{hom}^k(B,B)$ \\
 \hline
 $k\leq -1$ & $0$ & $0$ & $0$ & $0$  \\
 \hline
 $k=0$ & $0$ & $b$ & $1$ & $1$ \\
 \hline
 $k=1$ & $c$ & $b$ & $2$ & $2$  \\
 \hline
 $k=2$ & $c$ & $0$ & $0$ & $0$  \\
 \hline
 $k\geq 3$ & $0$ & $0$ & $0$ & $0$  \\
 \hline
 
\end{tabular}
\end{center}
\end{table}

Note that we have an exact triangle $B[-1]\rightarrow A\rightarrow F$ and there is a spectral sequence which degenerates at $E_3$ converging to $\mathrm{Ext}^{\bullet}(F,F)$ with $E_1$-page being 
$$E^{p,q}_1=\begin{cases}
\mathrm{Ext}^q(A,B[-1])=\mathrm{Ext}^{q-1}(A,B), p=-1\\
\mathrm{Ext}^q(B,B)\oplus\mathrm{Ext}^q(A,A), p=0\\
\mathrm{Ext}^q(B[-1],A)=\mathrm{Ext}^{q+1}(B,A), p=1\\
0, p\leq -2, p\geq 2\end{cases}
$$

Now, we compute $E_1$-page of the spectral sequence:
\begin{align*}
    E_1^{-1,4}&=\mathrm{Ext}^3(A,B)=0, E_1^{0,4}=\mathrm{Ext}^4(A,A)\oplus\mathrm{Ext}^4(B,B)=0,\\
    E_1^{1,4}&=0, E_1^{-1,3}=\mathrm{Ext}^2(A,B)=k^c, \\
E_1^{0,3}&=\mathrm{Ext}^3(A,A)\oplus\mathrm{Ext}^3(B,B)=0,
E_1^{1,3}=0\\
E_1^{-1,2}&=\mathrm{Ext}^1(A,B)=k^c, E_1^{0,2}=\mathrm{Ext}^2(A,A)\oplus\mathrm{Ext}^2(B,B)=0,\\
E_1^{1,2}&=\mathrm{Ext}^3(B,A)=0,
E_1^{-1,1}=\mathrm{Hom}(A,B)=0,\\ E_1^{0,1}&=\mathrm{Ext}^1(A,A)\oplus\mathrm{Ext}^1(B,B)=k^4, E_1^{1,1}=\mathrm{Ext}^2(B,A)=0, E_1^{-1,0}=0.
\end{align*}

\begin{align*}
    E_1^{0,0}&=\mathrm{Hom}(B,B)\oplus\mathrm{Hom}(A,A)=k^2, E_1^{1,0}=\mathrm{Ext}^1(B,A)=k^b, E_1^{-1.-1}=0,\\ E_1^{0,-1}&=0, E_1^{1,-1}=\mathrm{Hom}(B,A)=k^b.
\end{align*}

Then the $E_1$-page looks like:
\begin{center}
\begin{tabular}{cc|cc}
 0 & 0 & 0 & 0\\
 0 & $c$ & 0 & 0\\
 0 & $c$ & 0 & 0 \\
 0 & 0 & 4 & 0 \\
 0 & 0 & 2 & $b$ \\
\hline
0 & 0 & 0 & $b$\\ 
\end{tabular}
\end{center}

with the horizontal differential. Everywhere else are $0$ since $\mathcal{A}$ has homological dimension $2$. The next page looks like
\begin{center}
\begin{tabular}{cc|cc}
 0 & 0 & 0 & 0\\
 0 & $c$ & 0 & 0\\
 0 & $c$ & 0 & 0 \\
 0 & 0 & 4 & 0 \\
 0 & 0 & $x$ & $y$ \\
\hline
0 & 0 & 0 & $b$\\ 
\end{tabular}
\end{center}

On the first page, we have $E_1^{0,0}=k^2\xrightarrow{\pi} E_1^{1,0}=k^b$. Thus $E_2^{0,0}=\mathrm{Ker}(\pi)$. Note that there is a sequence of morphisms $0=E_2^{-2,1}\rightarrow E_2^{0,0}\rightarrow E_2^{2,-1}=0$. Since the spectral sequence degenerates at $E_3$, then $E_{\infty}^{0,0}=\mathrm{Ker}(\pi)$. Also note that $E_1^{-1,1}=0$, then $E_{\infty}^{-1,1}=0$. Similarly, it is easy to see that $E_{\infty}^{1,-1}=E_2^{1,-1}=k^b$. So we have $$dim(\mathrm{Ker}(\pi))+b=1$$ i.e. $x+b=1$. 

Now, $E_1^{1,0}=k^b$, $E_2^{1,0}=\frac{k^b}{im(\pi)}$. From the second table, it is easy to see that $E_{\infty}^{1,0}=\frac{k^b}{im(\pi)}$. $E_2^{0,1}=E_1^{0,1}=k^4$. Then $E_{\infty}^{0,1}=k^4$. Similarly, From the second table, we conclude that $E_{\infty}^{-1,2}=k^c$. All of them contribute to $\mathrm{hom}^1(F,F)=4$. So we have
$$dim(\frac{k^b}{im(\pi)})+4+c=4$$ i.e $y+4+c=4$. 

Moreover, $E_1^{1,1}=0, E_1^{0,2}=0$, thus $E_{\infty}^{1,1}=E_{\infty}^{0,2}=0$. But $E_1^{-1,3}=k^c$. From the second table, we know that $E_{\infty}^{-1,3}=k^c$. Then we have $c=\mathrm{hom}^2(F,F)=1$. 

Eventually, we have the system of equations:$$\begin{cases}
x+b=1\\
y+c=0\\
c=1, x,y,b,c\geq 0
\end{cases}$$

It is clear that we have a contradiction. This means that there does not exist any destablizing triangle $A\rightarrow F\rightarrow B$, so that $F$ is $\sigma$-stable. 

Now, assume that $\mathrm{Hom}(A,\tau(B))\neq 0$. Run the argument of Corollay~\ref{lemma_ext3_stable}, we get another an exact triangle $D\rightarrow F\rightarrow G$ with $D,G, \sigma$-semistable such that $\mathrm{Hom}(D,G)=\mathrm{Hom}(D,\tau(G))=0$. Then it reduces to the case above. This means that $\mathrm{Hom}(A,\tau(B))$ must be $0$. Then the desired result follows. Now, we show that $F$ is $\sigma$-semistable. By the exactly the same argument in Corollary~\ref{lemma_ext3_stable} and the first part of the proof of Proposition~\ref{prop_stability_ext14}, the result follows. \end{proof}

\begin{corollary}
\label{corollary_uniqueness_elementary}
Let $F\in\mathcal{A}$ with $\chi(F,F)=-2$. Let $\sigma_1,\sigma_2$ be $\tau$-invariant stability conditions on $\mathcal{K}u(X)$ and $F$ is $\sigma_1$-stable. Then $F$ is $\sigma_2$-stable.
\end{corollary}

\begin{proof}
Since $F$ is $\sigma_1$-stable and $\chi(F,F)=-2$. Then we have $\mathrm{ext}^1(F,F)=3+\mathrm{hom}(F,\tau(F))$. Since $\sigma_1$ is $\tau$-invariant, $\tau(F)$ is also $\sigma_1$-stable object with the same phase as $F$. Thus $\mathrm{ext}^1(F,F)=3$ or $4$. If $\mathrm{ext}^1(F,F)=3$, by Corollary~\ref{lemma_ext3_stable}, $F$ is $\sigma_2$-stable. If $\mathrm{ext}^1(F,F)=4$, by Proposition~\ref{prop_stability_ext14}, $F$ is $\sigma_2$-stable. 
\end{proof}

We close this section by showing that the $\widetilde{GL_+(2,\mathbb{R})}$-orbit $\mathcal{K}$ of Serre-invariant stability condition $\sigma:=\sigma(\alpha,\beta)$ on the Kuznetsov components of a series prime Fano threefolds is unqiue, which is a generalization of Corollary~\ref{corollary_uniqueness_elementary}. 

\subsection{Uniqueness of Serre invariant stability conditions}
Let $Y_d$ be smooth index two degree $d\geq 2$ prime Fano threefold and $X_{4d+2}$ index one degree $4d+2$ prime Fano threefold. In this section, we show that all Serre-invariant stability conditions on $\mathcal{K}u(Y_d)$ and $\mathcal{K}u(X_{4d+2})$(or $\mathcal{A}_{X_{4d+2}}$) are in the same $\widetilde{\mathrm{GL}}^+(2,\mathbb{R})$-orbit for each $d\geq 2$.

\begin{lemma} \label{homo dim 2}
Let $\sigma'$ be a Serre-invariant stability condition on $\Ku(Y_d)$ and $d\geq 2$, then the heart of $\sigma'$ has homological dimension at most 2.

\end{lemma}

\begin{proof}
When $d=2$, this is from the same argument in Proposition \ref{prop_heart_GM}. When $d=3$, this is from \cite[Lemma 5.10]{pertusi2020some}. When $d=4$ and $5$, since $\Ku(Y_4)\cong D^b(C_2)$ and $\Ku(Y_5)\cong D^b(\mathrm{Rep}(K(3)))$ where $C_2$ is a genus 2 smooth curve and $\mathrm{Rep}(K(3))$ is the category of representations of 3-Kronecker quiver (\cite{KPS2018}), then in these two cases the heart has homological dimension 1.
\end{proof}

\begin{lemma} \label{phase jump 2}
Let $\sigma'$ be a Serre-invariant stability condition on $\Ku(Y_d), d\geq 2$. 
If $E,F$ are two $\sigma'$-semistable objects with phases $\phi'(E)<\phi'(F)$, then $\mathrm{Hom}(E, F[2])=0$.
\end{lemma}

\begin{proof}
When $d=4$ and $5$, this is from that the heart of $\sigma'$ has homological dimension 1.  When $d=2$ and $3$ this is from \cite[Sec. 5, Sec. 6]{pertusi2020some}. 
\end{proof}

\begin{lemma} \emph{(Weak Mukai Lemma)} \label{mukai lemma}
Let $\sigma'$ be a Serre-invariant stability condition on $\Ku(Y_d), d\geq 2$. Let
$$F\to E\to G$$
be an exact triangle in $\Ku(Y_d)$ such that $\mathrm{Hom}(F,G)=\mathrm{Hom}(G,F[2])=0$. Then we have
\[\mathrm{ext}^1(F,F)+\mathrm{ext}^1(G,G)\leq \mathrm{ext}^1(E,E)\]
\end{lemma}

\begin{lemma} \label{decrease ext1}
Let $\sigma'$ be a Serre-invariant stability condition on $\Ku(Y_d)$ and $d\geq 2$. Assume that there's a triangle of $E\in \Ku(Y_d)$
\[F\to E\to G\]
such that phases of all $\sigma'$-semistable factors of $F$ are greater than  $\sigma'$-semistable factors of $G$, then we have $\mathrm{ext}^1(F,F)<\mathrm{ext}^1(E,E)$ and $\mathrm{ext}^1(G,G)<\mathrm{ext}^1(E,E)$
\end{lemma}

\begin{proof}
Since $\phi'(F)> \phi'(G)$, by Lemma \ref{phase jump 2} we have $\mathrm{Hom}(F,G)=0$ and $\mathrm{Hom}(G,F[2])=\mathrm{Hom}(F[2], S_{\Ku(Y_d)}(E))=0$. Thus the result follows from Lemma \ref{mukai lemma}.
\end{proof}


Let $\sigma=\sigma(\alpha, -\frac{1}{2})$ and $Y:=Y_d, d\geq 2$. As shown in \cite[Section 4]{pertusi2020some}, the moduli space $\mathcal{M}_{\sigma}(\Ku(Y), -v)$ and $\mathcal{M}_{\sigma}(\Ku(Y), w-v)$ are non-empty. Let $A, B\in \mathcal{A}(\alpha, -\frac{1}{2})$ such that $[A]=-v, [B]=w-v$ are $\sigma$-stable objects. We denote the phase with respect to $\sigma=\sigma(\alpha, -\frac{1}{2})$ by $\phi(-):=\phi(\alpha, \beta)(-)$.

Now let $\sigma_1$ be any Serre-invariant stability condition on $\Ku(Y)$. By \cite[Remark 5.14]{pertusi2020some}, there's a $T=(t_{ij})_{1\leq i,j\leq 2}\in \mathrm{GL}_2^+(\mathrm{R})$ such that $Z_1=T\cdot Z(\alpha, -\frac{1}{2})$. Since $A$ is stable with respect to every Serre-invariant stailibty condition by \cite[Lemma 5.16]{pertusi2020some}, we can assume $A[m]\in \mathcal{A}_1$. Let $\sigma_2=\sigma\cdot \Tilde{g}$ for $\widetilde{g}:=(g, T)\in \widetilde{\mathrm{GL}}^+_2(\mathbb{R})$ such that $A[m]\in \mathcal{A}_2$ and $Z_2=Z_1$. Then we have $\phi_1(A)=\phi_2(A)$ and $\mathcal{A}_2=\mathcal{P}(\alpha, -\frac{1}{2})((g(0), g(0)+1])$.

\begin{lemma} \label{phase same}\cite[Lemma 4.21]{JLLZh2021}
Notations as above, then $A$ and $B$ are $\sigma_1$-stable with phase $\phi_1(A)=\phi_2(A)$ and $\phi_1(B)=\phi_2(B)$.
\end{lemma}

\begin{proof}
The stability of $A$ and $B$ is from \cite[Lemma 5.13]{pertusi2020some}. By definition of $\sigma_2$, we know $\phi_1(A)=\phi_2(A)$ and $\phi_2(B)< \phi_2(A)< \phi_2(B)+1$. And from \cite[Remark 4.8]{pertusi2020some} we know $\phi_1(B)<\phi_1(A)=\phi_2(A)< \phi_1(B)+1$. Thus $\phi_1(B)=\phi_2(B)$.




\end{proof}


\begin{theorem}\label{all_in_one_orbit}\cite[Proposition 4.22]{JLLZh2021}

All Serre-invariant stability conditions on $\mathcal{K}u(X)$ are in the same $\widetilde{\mathrm{GL}}^+_2(\mathbb{R})$-orbit. Here $X:=X_{4d+2}$ or $Y_{d}$ for all $d\geq 2$ and $Y_2$.
\end{theorem}

\begin{proof}
Notations as above, we are going to show $\sigma_1=\sigma_2$. Since $\Ku(X_{12}), \Ku(X_{16})$ and $\Ku(X_{18})\cong \Ku(Y_4)$ are equivalent to the bounded derived categories of some smooth curves of positive genus, the results for these three cases are follow from \cite[Theorem 2.7]{macri2007stability}. And the results for $X_{14}$ and $X_{22}$ are from the results for $\Ku(Y_3)$ and $\Ku(Y_{5})$ and equivalences $\Ku(Y_d)\cong \Ku(X_{4d+2}), d\geq 3$ in \cite{KPS2018}. Thus we only need to prove this for $Y_d, d\geq 2$ and $X_{10}$


We first prove this for $Y_d, d\geq 2$.

Let $E\in \mathcal{A}(\alpha, -\frac{1}{2})$ be a $\sigma$-semistable object with $[E]=av+bw$. First we are going to show that if $E$ is $\sigma_1$-semistable, then $\phi_2(E)= \phi_1(E)$.  Note that we have following relations:
\begin{enumerate}

    \item $\chi(E,A)=a+(d-1)b, \chi(A,E)=a+b$; $\mu^0_{\alpha, -\frac{1}{2}}(E)>\mu^0_{\alpha, -\frac{1}{2}}(A) \iff b<0$
    
    \item $\chi(E,B)=-b, \chi(B,E)=-[(d-2)a+(d-1)b]$; $\mu^0_{\alpha, -\frac{1}{2}}(E)>\mu^0_{\alpha, -\frac{1}{2}}(B) \iff a+b<0$
    
    
\end{enumerate}
From the definition of stability $\sigma=\sigma(\alpha, -\frac{1}{2})$ we have $a\leq 0$. When $a=0$, by definition of stability we have $b<0$. Thus in the case $b>0$ we always have $a<0$ and $\chi(D,E)<0$. 
Note that by definition of $\sigma_2$ we have $\phi_2(B)<\phi_2(A)$ and both of them lie in the interval $(g(0), g(0)+1]$.

\begin{itemize}

    
    \item Assume $b>0$ and $a+b>0$. Then $\mu^0_{\alpha, -\frac{1}{2}}(E)<\mu^0_{\alpha, -\frac{1}{2}}(B)<\mu^0_{\alpha, -\frac{1}{2}}(A)$ and hence $\phi_2(E)<\phi_2(B)< \phi_2(A)$. We also have $\chi(E,A)>0$. Thus by Lemma \ref{phase jump 2} we know $\mathrm{Hom}(E,A[2])=0$. Thus $\chi(E,A)=\mathrm{hom}(E,A)-\mathrm{hom}(E,A[1])>0$ implies $\mathrm{hom}(E,A)>0$, and therefore $\phi_1(E)<\phi_{1}(A)$. Also from $\chi(B, E)<0$ and Lemma \ref{homo dim 2} we have $\phi_1(B)-1< \phi_1(E)$. Then we have $\phi_1(B)-1< \phi_1(E)< \phi_1(A)$. But by Lemma \ref{phase same} we know $\phi_1(B)=\phi_2(B)$, $\phi_1(A)=\phi_2(A)$. And from definition of $\sigma_2$ we have $|\phi_2(B)-\phi_2(A)|<1$ and $|\phi_2(A)-\phi_2(E)|<1$. Thus $\phi_2(E)-\phi_1(E)=0$ or $1$. But if $\phi_2(E)=\phi_1(E)+1$, then $\phi_2(B)-1=\phi_1(B)-1<\phi_2(E)<\phi_1(B)=\phi_2(B)$. This implies $1=\phi_1(B)-\phi_1(B)+1>\phi_2(E)-\phi_1(B)+1=\phi_1(E)-\phi_1(B)+2$, which is impossible since $\phi_1(B)-1< \phi_1(E)$. Thus we have $\phi_1(E)=\phi_2(E)$.

    
    \item Assume $b>0$ and $a+b<0$. Then $\mu^0_{\alpha, -\frac{1}{2}}(B)< \mu^0_{\alpha, -\frac{1}{2}}(E)< \mu^0_{\alpha, -\frac{1}{2}}(A)$ and hence $\phi_2(B)< \phi_2(E)< \phi_2(A)$. Since $\chi(A,E)<0$ and $\chi(E,B)<0$, from Lemma \ref{homo dim 2} we know $\mathrm{hom}(A, E[1])>0$ and $\mathrm{hom}(E, B[1])>0$, hence  $\phi_1(A)-1< \phi_1(E)< \phi_1(B)+1$. This means $|\phi_1(E)-\phi_2(E)|=0$ or $1$. But $|\phi_1(E)-\phi_2(E)|=1$ is impossible since $\phi_1(B)=\phi_2(B)< \phi_2(E)<\phi_2(A)=\phi_1(A)$. Therefore we have $\phi_1(E)=\phi_2(E)$.

    
    
    
    \item Assume $b<0$. Then $\mu^0_{\alpha, -\frac{1}{2}}(B)< \mu^0_{\alpha, -\frac{1}{2}}(A)< \mu^0_{\alpha, -\frac{1}{2}}(E)$ and hence $\phi_2(B)<\phi_2(A)< \phi_2(E)$. Since $\chi(E,A)<0$, from Lemma \ref{homo dim 2} we have $\mathrm{hom}(E, A[1])>0$ and $\phi_1(E)< \phi_1(A)+1$. By Lemma \ref{phase jump 2}, $\mu^0_{\alpha, -\frac{1}{2}}(B)<\mu^0_{\alpha, -\frac{1}{2}}(E)$ and  $\chi(B,E)>0$ we know that $\mathrm{hom}(B,E)>0$. Thus $\phi_1(B)< \phi_1(E)< \phi_1(A)+1$. Hence  $\phi_1(E)-\phi_2(E)=0$ or $1$. But since $\mu^0_{\alpha, -\frac{1}{2}}(A)<\mu^0_{\alpha, -\frac{1}{2}}(E)$, we have $\phi_2(A)=\phi_1(A)<\phi_2(E)$. Thus $\phi_1(A)<\phi_2(E)<\phi_1(A)+1$. Then $\phi_1(E)-\phi_2(E)=1$ is impossible since $\phi_1(E)<\phi_1(A)+1$. Therefore we have $\phi_1(E)=\phi_2(E)$.
    
    
    \item When $b=0$, we have $[E]=-a\cdot [A]$. Hence $\chi(E,A)=\chi(A,E)<0$ and we have $\phi_1(A)-1\leq  \phi_1(E)\leq \phi_1(A)+1$. But $\mu_1(E)=\mu_1(A)$, we know $\phi_1(E)-\phi_1(A)$ is an integer. Thus $\phi_1(E)=\phi_1(A)\pm 1$. But from definition of stability function, we have $\mathrm{Im}(Z_1(E[\pm 1]))=-\mathrm{Im}(Z_1(A))$. Thus $\phi_1(E)=\phi_1(A)=\phi_2(E)$.
    
    \item When $a+b=0$, we have $[E]=-a\cdot [B]$. Hence $\chi(E,B)=\chi(B,E)<0$ and we have $\phi_1(B)-1\leq  \phi_1(E)\leq \phi_1(B)+1$. But $\mu_1(E)=\mu_1(B)$, we know $\phi_1(E)-\phi_1(B)$ is an integer. Thus $\phi_1(E)=\phi_1(B)\pm 1$. But from definition of stability function, we have $\mathrm{Im}(Z_1(E[\pm 1]))=-\mathrm{Im}(Z_1(B))$. Thus $\phi_1(E)=\phi_1(B)=\phi_2(E)$.
    
\end{itemize}


Next we prove that $E\in \mathcal{A}_2$ is $\sigma$-semistable if and only if $E\in \mathcal{A}_1$ is $\sigma_1$-semistable. We prove by induction. 

If $\mathrm{ext}^1(E,E)<2$, this is from \cite[Sec. 5]{pertusi2020some}. Now assume this is true for all $E\in \mathcal{A}_2$ $\sigma_2$-semistable such that $\mathrm{ext}^1(E,E)<N$. 

When $E\in \mathcal{A}_2$ is  $\sigma_2$-semistable and has $\mathrm{ext}^1(E,E)=N$, assume otherwise $E$ is not $\sigma_1$-semistable. Let $A_0$ be the first HN-factor of $E$ with respect to $\sigma_1$ and $A_n$ be the last one. Then $\phi_1(A_0)>\phi_1(A_n)$. By Lemma \ref{decrease ext1}, $\mathrm{ext}^1(A_0,A_0)< N$ and $\mathrm{ext}^1(A_n,A_n)< N$. Thus $A_0$ and $A_n$ are $\sigma_2$-semistable by induction hypothesis and $\phi_2(A_0)>\phi_2(A_n)$ by results above. Since $\mathrm{Hom}(A_0, E)$ and $\mathrm{Hom}(E,A_n)$ are both non-zero, we know $\phi_2(A_0)\leq \phi_2(E)$ and $\phi_2(E)\leq \phi_2(A_n)$, which implies $\phi_2(A_0)\leq \phi_2(A_n)$ and makes a contradiction. Thus $E$ is $\sigma_1$-semistable. When $E\in \mathcal{A}_1$ is $\sigma_1$-semistable,  the same argument shows that $E\in \mathcal{A}_2$ is also $\sigma_2$-semistable.

Now since every objects in the heart is the extensions of semistable objects, we have $\mathcal{A}_1=\mathcal{A}_2$. And since $Z_1=Z_2$, we know actually $\sigma_1=\sigma_2=\sigma\cdot \widetilde{g}$. Hence $\sigma_1$ is in the orbit of $\sigma=\sigma(\alpha, -\frac{1}{2})$.  

Now we prove our statement for Serre-invariant stability conditions on Kuznetsov component $\mathcal{A}_X$ for GM threefold $X:=X_{10}$. The argument is similar to previous cases. It is not hard to find two $\sigma$-stable objects $A,B\in \mathcal{A}_X$ such that $\mathrm{hom}(B, A)>0$ and $\mathrm{hom}(A, B[1])>0$. This means $\phi_{\sigma'}(A)-1<\phi_{\sigma'}(B)< \phi_{\sigma'}(A)$ for every Serre-invariant stability condition $\sigma'$ on $\mathcal{A}_{X_{10}}$. Then $[A]=-x$,  $[B]=y-2x$ and $\phi_{\sigma'}(B)< \phi_{\sigma'}(A)< \phi_{\sigma'}(B)+1$. For more details, see \cite[Theorem 4.25]{JLLZh2021}. 


Let $\sigma:=\sigma(\frac{1}{10}, -\frac{1}{10})$. Assume $\sigma_1$ is any Serre-invariant stability condition on $\mathcal{A}_{X_{10}}$. Then as in \cite[Remark 5.14]{pertusi2020some}, there's a $T=(t_{ij})_{1\leq i,j\leq 2}\in \mathrm{GL}_2^+(\mathrm{R})$ such that $Z_1=T\cdot Z(\frac{1}{10}, -\frac{1}{10})$. Since $A$ is stable with respect to every Serre-invariant stailibty condition by \cite[Lemma 5.16]{pertusi2020some}, we can assume $A[m]\in \mathcal{A}_1$. Let $\sigma_2=\sigma\cdot \widetilde{g}$ for $\widetilde{g}:=(g, T)\in \widetilde{\mathrm{GL}}^+_2(\mathbb{R})$ such that $A[m]\in \mathcal{A}_2$ and $Z_2=Z_1$. Then we have $\phi_1(A)=\phi_2(A)$ and $\mathcal{A}_2=\mathcal{P}(\frac{1}{10}, -\frac{1}{10})((g(0), g(0)+1])$. Now the same argument in Lemma \ref{phase same} shows that $\phi_1(B)=\phi_2(B)$. Let $E\in \mathcal{A}(\frac{1}{10}, -\frac{1}{10})$ be a $\sigma$-semistable object with $[E]=ax+by$, then by definition of stability we have $a+b\leq 0$. We have following relations:

\begin{enumerate}
    \item $\chi(E,A)= a+2b, \chi(A,E)=a+2b$; $\mu^0_{\frac{1}{10},- \frac{1}{10}}(E)>\mu^0_{\frac{1}{10},- \frac{1}{10}}(A) \iff b<0$
    
    \item $\chi(E,B)=-b, \chi(B,E)=-b$; $\mu^0_{\frac{1}{10},- \frac{1}{10}}(E)>\mu^0_{\frac{1}{10},- \frac{1}{10}}(B) \iff a+2b<0$
\end{enumerate}
Let $a':=a+b$, then $a'\leq 0$ and these relations are of the same form as index 2 case. Since above lemmas is also true for $\mathcal{A}_{X_{10}}$, the same argument works in this case.

\end{proof}

\begin{remark}
The proof of Theorem~\ref{all_in_one_orbit} in the first version of the article is incomplete and Zhiyu Liu pointed out a gap, i.e. the induction hypothesis was too strong and the $\sigma_2$-(semi)stability of the object $E$ and the assertion $\phi_1(E)=\phi_2(E)$ could not be proved simultaneously. In \cite[Proposition 4.22]{JLLZh2021}, we give a uniform proof for uniqueness of Serre-invariant stability conditions on $\Ku(Y_d)$ and $\Ku(X_{4d+2})$ for all $d\geq 2$. I thank him for agreeing to include the results in this paper.  
\end{remark}

\section{Stability of ideal sheaf of twisted cubic}
\label{section5}
In this section, we prove stability of the ideal sheaf of a twisted cubic and its projection onto the Kuznetsov components. $X$ is assumed to be a special Gushel-Mukai threefold with general branch locus. 

\begin{theorem}
\label{theorem_stabilityofidealsheaf}
Let $C$ be a twisted cubic on $X$ such that $I_C\in\mathcal{K}u(X)$. Then $I_C$ is $\sigma$-stable for every $\tau$-invariant stability condition on $\mathcal{K}u(X)$. 
\end{theorem}

\begin{proof}
By Proposition~\ref{Prop_extgroup}, $\mathrm{ext}^1(I_C,I_C)=3$ and $\chi(I_C,I_C)=-2$. Then $I_C\in\mathcal{A}$ up to shifts by Lemma~\ref{lemma_ext3_inheart}. The shifts can be chosen uniformly by argument in Theorem~\ref{Theorem_bijective_closedpoints}. By Corollory~\ref{lemma_ext3_stable}, $I_C$ is $\sigma$-stable.    
\end{proof}

For twisted cubics $C$ such that $I_C\not\in\mathcal{K}u(X)$, we project it onto the Kuznetsov component and show that $\mathrm{pr}(I_C)$ is $\sigma$-stable. 

\begin{predl}
\label{prop_projectionideal}
Let $C$ be a twisted cubic whose ideal sheaf $I_C\not\in\mathcal{K}u(X)$. Let $E:=\mathrm{pr}(I_C)\in\mathcal{K}u(X)$ be the projection of the ideal sheaf of twisted cubic onto the Kuznetsov component. Then we have a distinguished triangle  $$\mathcal{O}_X(-H)[1]\rightarrow E\rightarrow\mathcal{E}_L$$ where $\mathcal{E}_L$ is the extension of $\mathcal{O}_L(-2)$ and $\mathcal{E}$, 
$$0\rightarrow\mathcal{O}_L(-2)\rightarrow\mathcal{E}_L\rightarrow\mathcal{E}\rightarrow 0$$
\end{predl}

\begin{proof}
The projection functor $\mathrm{pr}:D^b(X)\rightarrow\mathcal{K}u(X)$ is given by $\bL_{\mathcal{E}}\bL_{\mathcal{O}_X}$. Then $\mathrm{pr}(I_C)=\bL_{\mathcal{E}}\bL_{\mathcal{O}_X}(I_C)$. By Lemma~\ref{lemma_cohomology_idealsheaf},  $\mathrm{Hom}^i(\mathcal{O}_X,I_C)=0$ for all $i$. Thus $\mathrm{pr}(I_C)=\bL_{\mathcal{E}}(I_C)$. It is defined by an exact triangle:
$$\mathrm{RHom}(\mathcal{E},I_C)\otimes\mathcal{E}\rightarrow I_C\rightarrow \bL_{\mathcal{E}}(I_C)$$

By Proposition~\ref{prop_cubic_inKu}, we get a long exact sequence of cohomology with respect to $\mathrm{Coh}(X)$:

$$0\rightarrow \mathcal{H}^{-1}(\bL_{\mathcal{E}}I_C)\rightarrow\mathcal{E}\xrightarrow{\pi} I_C\rightarrow \mathcal{H}^0(\bL_{\mathcal{E}} I_C)\rightarrow\mathcal{E}\rightarrow 0$$ 

Split it to two short exact sequences,
$$0\rightarrow\mathcal{H}^{-1}(\bL_{\mathcal{E}} I_C)\rightarrow\mathcal{E}\xrightarrow{\pi}I_D\rightarrow 0$$ and $$0\rightarrow\frac{I_C}{I_D}\rightarrow \mathcal{H}^0(\bL_{\mathcal{E}} I_C)\rightarrow\mathcal{E}\rightarrow 0.$$ where $D$ is the zero locus of a section of $\mathcal{E}^{\vee}$ containing $C$.  Then by the same argument in Proposition~\ref{Prop_computingcone}, we have $\mathcal{H}^{-1}(\bL_{\mathcal{E}} I_C)\cong\mathcal{O}_X(-H)$ and $0\rightarrow\mathcal{O}_L(-2)\rightarrow\mathcal{H}^0(\bL_{\mathcal{E}}(I_C))\rightarrow\mathcal{E}\rightarrow 0$. The desired result follows. 
\end{proof}

Next proposition relates $\mathrm{pr}(I_C)$ to \emph{twisted derived dual} of a line $G$ we constructed in Proposition~\ref{Prop_computingcone}. Note that $\mathrm{Ext}^1(\mathcal{E},G)=k$. 

\begin{predl}
\label{prop_triangle_projection}
Let $\mathcal{F}$ be the object in $D^b(X)$ formed by the non trivial extension $G\rightarrow \mathcal{F}\rightarrow\mathcal{E}$, then $\mathcal{F}\cong\bL_{\mathcal{E}}(I_C)=\mathrm{pr}(I_C)$.
\end{predl}

\begin{proof}
Apply projection functor $\mathrm{pr}$ to the triangle $\mathcal{E}\rightarrow I_C\rightarrow G$ we get $\mathrm{pr}(I_C)\cong\mathrm{pr}(G)$ since $\mathrm{pr}(\mathcal{E})=\bL_{\mathcal{E}}(\mathcal{E})=0$. Note that $\mathrm{Hom}^{\bullet}(\mathcal{E},G)=k[-1]$. Then we have the exact triangle 
$$\mathcal{E}[-1]\rightarrow G\rightarrow\mathrm{pr}(G).$$ This implies that $\mathrm{pr}(G)$ is the non trivial extension of $\mathcal{E}$ by $G$, the desired result follows. 
\end{proof}

\begin{corollary}\label{corollary_computing_ext_projection_object}
$$\mathrm{hom}^k(\mathrm{pr}(I_C),\mathrm{pr}(I_C))=\begin{cases} 
1, k=0\\
3, k=1\\
0, k=2\\
0, k=3\end{cases}$$ 
In particular, $\mathrm{pr}(I_C)$ is $\sigma$-stable for every $\tau$-invariant stabilty condition $\sigma$ on $\mathcal{K}u(X)$.
\end{corollary}
\begin{proof}
By Proposition~\ref{prop_triangle_projection}, we have the exact triangle $\mathcal{E}[-1]\rightarrow G\rightarrow\mathrm{pr}(I_C)$. Then we compute $\mathrm{hom}^k(\mathrm{pr}(I_C),\mathrm{pr}(I_C))$ by spectral sequence in Lemma~\ref{lemma_SS}. First, we show $\mathrm{hom}^2(\mathrm{pr}(I_C),\mathrm{pr}(I_C))=0$. Then the first page is given by 
$$E^{p,q}_1=\begin{cases}
\mathrm{Ext}^3(G,\mathcal{E}[-1])=\mathrm{Ext}^2(G,\mathcal{E}), p=-1\\
\mathrm{Ext}^2(G,G)\oplus\mathrm{Ext}^2(\mathcal{E},\mathcal{E}), p=0\\
\mathrm{Ext}^2(\mathcal{E},G), p=1\\
0, p\leq -2, p\geq 2. \end{cases}$$

By the proof of Proposition~\ref{Prop_extgroup}, $\mathrm{Ext}^2(G,\mathcal{E})=0$. 
Note that $G,\mathcal{E}\in\mathcal{O}_X^{\perp}$. Then by Proposition~\ref{lemma_involutive_D}, \begin{align*}
    \mathrm{Ext}^2(\mathcal{E},G)\cong\mathrm{Ext}^2(\mathbb{D}(G),\mathbb{D}(\mathcal{E}))
    \cong\mathrm{Ext}^2(I_L,\mathcal{E}[1])\cong\mathrm{Ext}^3(I_L,\mathcal{E})
    \cong\mathrm{Hom}(\mathcal{E}^{\vee},I_L)
\end{align*}
Apply $\mathrm{Hom}(\mathcal{E}^{\vee},-)$ to the standard exact sequence $0\rightarrow I_L\rightarrow\mathcal{O}_X\rightarrow\mathcal{O}_L\rightarrow 0$ we get $\mathrm{Hom}(\mathcal{E}^{\vee},I_L)=0$. 
It is clear that $\mathrm{Ext}^2(G,G)\oplus\mathrm{Ext}^2(\mathcal{E},\mathcal{E})=0$ since $\Sigma(X)$ is smooth and $\mathcal{E}$ is exceptional bundle. Thus $E_1^{p,q}=0$ for all $p+q=2$. Therefore $\mathrm{ext}^2(\mathrm{pr}(I_C),\mathrm{pr}(I_C))=0$.  Next we compute $\mathrm{ext}^3(\mathrm{pr}(I_C),\mathrm{pr}(I_C))$. The first page of the spectral sequence is given by 
$$E^{p,q}_1=\begin{cases}
\mathrm{Ext}^3(G,\mathcal{E}), p=-1\\
\mathrm{Ext}^3(G,G)\oplus\mathrm{Ext}^3(\mathcal{E},\mathcal{E}), p=0\\
\mathrm{Ext}^3(\mathcal{E},G), p=1\\
0, p\leq -2, p\geq 2. \end{cases}$$

Apply $\mathrm{Hom}(-,\mathcal{E})$ to the triangle $\mathcal{O}_X(-H)[1]\rightarrow G\rightarrow\mathcal{O}_L(-2)$, the part of the long exact sequence is given by 
$$\ldots\rightarrow\mathrm{Ext}^3(\mathcal{O}_L(-2),\mathcal{E})\rightarrow\mathrm{Ext}^3(G.\mathcal{E})\rightarrow\mathrm{Ext}^3(\mathcal{O}_X(-H)[1],\mathcal{E})\rightarrow\ldots$$
Note that $\mathrm{Ext}^3(\mathcal{O}_L(-2),\mathcal{E})\cong\mathrm{Hom}(\mathcal{E}^{\vee},\mathcal{O}_L(-2))\cong\mathrm{H}^0(L,\mathcal{E}|_L\otimes\mathcal{O}_L(-2))=0$ since $\mathcal{E}$ is globally generated. It is easy to see $\mathrm{Ext}^3(\mathcal{O}_X(-H)[1],\mathcal{E})\cong\mathrm{Ext}^2(\mathcal{O}_X,\mathcal{E}^{\vee})=0$. Then $\mathrm{Ext}^3(G,\mathcal{E})=0$. The Ext-group $\mathrm{Ext}^3(\mathcal{E},G)\cong\mathrm{Ext}^3(\mathbb{D}(G),\mathbb{D}(\mathcal{E}))\cong\mathrm{Ext}^3(I_L,\mathcal{E}[1])=0$ by dimension reason. Finally, $\mathrm{Ext}^3(G,G)\oplus\mathrm{Ext}^3(\mathcal{E},\mathcal{E})=0$. Then $E^{p,q}_1=0$ for all $p+q=3$. Therefore $\mathrm{ext}^3(\mathrm{pr}(I_C),\mathrm{pr}(I_C))=0$. 

Next, we compute $\mathrm{Hom}(\mathrm{pr}(I_C),\mathrm{pr}(I_C))$. The first page of the spectral sequence is given by
$$E^{p,q}_1=\begin{cases}
\mathrm{Ext}^1(G,\mathcal{E}[-1])=\mathrm{Hom}(G,\mathcal{E}), p=-1\\
\mathrm{Hom}(G,G)\oplus\mathrm{Hom}(\mathcal{E},\mathcal{E}), p=0\\
\mathrm{Hom}(\mathcal{E},G), p=1\\
0, p\leq -2, p\geq 2. \end{cases}$$
Note that $\mathrm{Hom}(G,\mathcal{E})\cong\mathrm{Hom}(\mathbb{D}(\mathcal{E}),\mathbb{D}(G))\cong \mathrm{Hom}(\mathcal{E}[1],I_L)=0$. Apply $\mathrm{Hom}(\mathcal{E},-)$ to the triangle $\mathcal{O}_X(-H)[1]\rightarrow G\rightarrow\mathcal{O}_L(-2)$, we get $\mathrm{Hom}(\mathcal{E},G)=0$. It is clear that $\mathrm{Hom}(G,G)\oplus\mathrm{Hom}(\mathcal{E},\mathcal{E})=k^2$.  Then $E^{0,0}_1=k^2$ and we have a sequence of morphism: $$\mathrm{Hom}(G,\mathcal{E}[-1])=E^{-1,0}_1\rightarrow\mathrm{Hom}(G,G)\oplus\mathrm{Hom}(\mathcal{E},\mathcal{E})=E^{0,0}_1\rightarrow\mathrm{Hom}(\mathcal{E}[-1],G)=E^{1,0}_1$$
Note that $\mathrm{Hom}(G,\mathcal{E}[-1])\cong\mathrm{Hom}(\mathcal{E}[2],I_L)=0$ and $\mathrm{Hom}(\mathcal{E}[-1],G)\cong\mathrm{Ext}^1(\mathcal{E},G)\cong\mathrm{Ext}^1(\mathcal{E}^{\vee}, I_L)=k$. 
The map $E_1^{0,0}\xrightarrow{p} E_1^{1,0}$ is non-trivial since this map is just (pre)postcompose of $\mathcal{E}\rightarrow G[1]$ with identity map. Then $p$ is surjective and $E_2^{0,0}\cong\mathrm{Ker}(p)\cong k$. Further note that $E^{2,-1}_2=E^{-2,1}_2=0$, then $E^{0,0}_3=k$. Then we have $\mathrm{Hom}(\mathrm{pr}(I_C),\mathrm{pr}(I_C))=k$. On the other hand, $\chi(\mathrm{pr}(I_C),\mathrm{pr}(I_C))=-2$, so $\mathrm{ext}^1(\mathrm{pr}(I_C),\mathrm{pr}(I_C))=3$. Then by Corollary~\ref{lemma_ext3_stable}, $\mathrm{pr}(I_C)$ is $\sigma$-stable. 
\end{proof}

\section{Bridgeland Moduli Space on special Gushel-Mukai threefold}
\label{section_bridgelandmoduli}
In this section, we construct a smooth irreducible component of Bridgeland moduli space of stable objects in the Kuznetsov components of special Gushel-Mukai threefold $X$ with general branch locus. 

Fix a $\tau$-invariant stability condition $\sigma$ on $\mathcal{K}u(X)$. Denote by $\mathcal{M}_{\sigma}(s)$ or $\mathcal{M}$ the Bridgeland moduli space of stable objects in Kuznetsov component of $X$ with class $s=[I_C]=1-3L+\frac{1}{2}P$. Let $\mathrm{pr}$ be the projection functor $\bL_{\mathcal{E}}\bL_{\mathcal{O}_X}:D^b(X)\rightarrow\mathcal{K}u(X)$, it follows from \cite[Lemma 3.25]{bayer2021stability} and \cite[Theorem 7.1]{kuznetsov2011base} that the projection functor is a Fourier-Mukai transform, i.e. $\mathrm{pr}\cong\Phi_K$ for some integral kernel $K\in D^b(X\times X)$. Let $\mathcal{I}$ be the universal ideal sheaf on $\mathcal{H}\times X$. Define: 
$$\Phi_K\times\mathrm{id}_{\mathcal{H}}=\Phi_{K\boxtimes\mathcal{O}_{{\Delta}_{\mathcal{H}}}}:D^b(X\times\mathcal{H})\rightarrow \mathcal{K}u(X\times \mathcal{H})$$ Thus $\Phi_{K\boxtimes\mathcal{O}_{\Delta_H}}(\mathcal{I})$ is a family of objects on $\mathcal{K}u(X)$ parametrised by $\mathcal{H}$, which defines a morphism $p:\mathcal{H}\rightarrow\mathcal{M}$.

\begin{lemma}
\label{lemma_compatibility_stableobject}
Let $\mathcal{I}$ be the universal ideal sheaf for Hilbert scheme $\mathcal{H}$ of twisted cubics. Let $s\in\mathcal{H}$ be a point. Denote by $i_s:\{s\}\times X\rightarrow\mathcal{H}\times X$. Then $$\Phi_K(i_s^*(\mathcal{I}))\cong i_s^*(\Phi_{K\boxtimes\mathcal{O}_{{\Delta}_{\mathcal{H}}}}(\mathcal{I}))$$
\end{lemma}

\begin{proof}
It is standard: $$\Phi_K(i_s^*(\mathcal{I}))\cong\mathrm{pr}(I_C)$$ and $$i_s^*(\Phi_{K\boxtimes\mathcal{O}_{{\Delta}_{\mathcal{H}}}}(\mathcal{I}))\cong\Phi_{i_s^*(K\boxtimes\mathcal{O}_{{\Delta}_{\mathcal{H}}})}(I_C)\cong\Phi_{K}(I_C)=\mathrm{pr}(I_C)$$
\end{proof}

This means that the image of point $p(s)\in\mathcal{M}$ is given by the $\sigma$-stable object $\mathrm{pr}(I_C)\in\mathcal{K}u(X)$(Theorem~\ref{theorem_stabilityofidealsheaf}, Corollary~\ref{corollary_computing_ext_projection_object}). Note that the morphism $p$ is injective on $\mathcal{H}_1$ since $\mathrm{pr}(I_C)=I_C\not\cong I_{C'}=\mathrm{pr}(I_{C'})$ for any two points $s=[I_C], s'=[I_{C'}]\in\mathcal{H}_1$. But $p$ is not injective on $\mathcal{H}_2$. Indeed, by Remark~\ref{Lemma_H2_isa_ruledsurface} $\mathcal{H}_2$ is a smooth ruled surface over $\Sigma(X)$. Let $s\in\mathcal{H}_2$, it is of the form $s=[I_{C_t}]$ for some $t\in\mathbb{P}^1$, where $C_t=L+M_t$ is a reducible twisted cubic with the line $L$ as the fixed component and the conic $M_t$ for some $t\in\mathbb{P}^1$. That is to say, $s$ is on the fiber $F_{[L]}\cong\mathbb{P}^1$ of the structure map $\pi:\mathcal{H}_2\rightarrow\Sigma(X)$ over a point $[L]\in\Sigma(X)$. Then for any $s\in F_{[L]}$ for some $[L]\in\Sigma(X)$, the image $p(s)$ is a fixed object $E:=\mathrm{pr}(I_{L+M_t})$, given by the exact triangle $\mathcal{O}_X(-H)[1]\rightarrow E\rightarrow\mathcal{E}_{\tau L}$ obtained in Proposition~\ref{prop_projectionideal}. By Corollary~\ref{corollary_computing_ext_projection_object}, $p(s)$ is a smooth point. 

\begin{predl}
\label{prop_irreduciblecomponent_modulispace}
The projection functor $\mathrm{pr}$ produces an irreducible smooth component $\mathcal{Z}=p(\mathcal{H})$ of dimension $3$ in moduli space $\mathcal{M}_{\sigma}(\mathcal{K}u(X), [I_C])$ of stable objects in $\mathcal{K}u(X)$ and the map $p:\mathcal{H}\rightarrow\mathcal{Z}$ is a birational proper morphism, contracting a smooth ruled surface $\mathcal{H}_2$ to a curve.
\end{predl}
\begin{proof}
The tangent space at a point $s=[I_C]\in\mathcal{H}$ is given by $\mathrm{Ext}^1(I_C,I_C)$, see \cite{Lieb05moduli}, \cite{Inaba02moduli}. The tangent space at  point $p(s)\in\mathcal{M}_{\sigma}(\mathcal{K}u(X), [I_C])$ is given by $\mathrm{Ext}^1(\mathrm{pr}(I_C),\mathrm{pr}(I_C))$ since $\mathrm{pr}(I_C)$ is a $\sigma$-stable object in $\mathcal{K}u(X)$. The tangent map $d_p:\mathrm{Ext}^1(I_C,I_C)\rightarrow\mathrm{Ext}^1(\mathrm{pr}(I_C),\mathrm{pr}(I_C))$ is induced by the projection functor $\mathrm{pr}$. It is easy to see that if $s\in\mathcal{H}_1$, the tangent map $d_p$ is an isomorphism of tangent spaces at $s$ and $p(s)$. By discussion above, $p$ is an etale map on $\mathcal{H}_1$. As $\mathcal{H}$ is an irreducible smooth projective variety by Theorem~\ref{theorem_projectivity_hilbertscheme}, so $p$ must factor through one of irreducible components, say $\mathcal{Z}$ of $\mathcal{M}$ and embbeds the open denset subset $\mathcal{U}:=\mathcal{H}_1\subset\mathcal{H}$ to $\mathcal{Z}$. As $\mathcal{Z}$ is irreducible, $p(\mathcal{U})$ is also an open dense subset of $\mathcal{Z}$. On the other hand, $p$ is a proper morphism since $\mathcal{H}$ is projective variety. Then $p(\mathcal{H})$ is closed in $\mathcal{Z}$. Moreover, the Zarisiki tangent spaces at each point of $\mathcal{H}$ and at its image are of same dimension three by Proposition~\ref{Prop_extgroup} and Corollary~\ref{corollary_computing_ext_projection_object}. Then we have $$\mathcal{Z}=\overline{p(\mathcal{U})}\subset\overline{p(\mathcal{H})}=p(\mathcal{H})\subset\mathcal{Z}.$$ Thus we prove the first statement. It is clear that $p:\mathcal{H}\rightarrow\mathcal{Z}$ is birational proper morphism. By the discussion above, $p$ contracts each fiber $F_{[L]}$ of the ruled surface $\mathcal{H}_2$ to a smooth point on $\Sigma(X)$. The result follows.
\end{proof}

\subsection{Projectivity of the moduli space}
\begin{theorem}
\label{Theorem_projectivity_modulispace}
\leavevmode
\begin{enumerate}
    \item The fiber $F_{[L]}$ in ruled surface $\mathcal{H}_2$ over each point $[L]\in\Sigma(X)$ is a $K_{\mathcal{H}}$-negative extremal curves. 
    \item The morphism $p$ is a divisorial contraction and it is a blow up of a smooth curve. In particular, $\mathcal{Z}$ is smooth projective. 
\end{enumerate}
\end{theorem}

\begin{proof}
\leavevmode\begin{enumerate}
    \item The map $p:\mathcal{H}\rightarrow\mathcal{Z}$ is a birational proper surjective morphism by Proposition~\ref{prop_irreduciblecomponent_modulispace}. As $\mathcal{H},\mathcal{Z}$ are both smooth and $\mathcal{H}$ is projective, then by \cite[Section 2.3]{KM08Birational} we have 
$$K_{\mathcal{H}}=p^*K_{\mathcal{Z}}+aE$$ where $E=\mathcal{H}_2$ is the exceptional divisor and $a$ is discrepancy, a positive number since $\mathcal{Z}$ is smooth. Then $$K_{\mathcal{H}}\cdot F_{[L]}=p^*K_{\mathcal{Z}}\cdot F_{[L]}+aE\cdot F_{[L]}.$$ 
The first term is $0$ since $p$ contracts $F_{[L]}$ to a smooth point, the second term is $-a$ since $E\cdot F_{[L]}=-1$. Thus $K_{\mathcal{H}}\cdot F_{[L]}=-a<0$. 

\item By (1), the extremal ray $R=\mathbb{R}_{\geq 0}[F_{[L]}]$ in $\mathcal{H}$ is $K_{\mathcal{H}}$-negative.  Then by Cone theorem \cite[Theorem 3.7]{KM08Birational}, there is a unique map $\mathrm{cont}_R: \mathcal{H}\rightarrow Z$ with $Z$ being projective and $\mathrm{cont}_R$ maps each fiber  $F_{[L]}$ to a point at a smooth irreducible variety $\mathcal{Z}$. The uniqueness of $\mathrm{cont}_R$ means that $\mathcal{Z}\cong Z$, therefore $\mathcal{Z}$ is projective. Then by \cite[Theorem 3.3 (1),Corollary 3.4.1]{Mori82Threefolds} and \cite[Theorem 1.1.2.1]{Kollar91Extrem}, we conclude that $\mathrm{cont}_R=p$ is a divisorial contraction,  it is blowing up a smooth curve on $\mathcal{Z}$. This means that $p:\mathcal{H}\rightarrow\mathcal{Z}$ induced by projection functor $\mathrm{pr}$ is a divisorical contraction.
\end{enumerate}
\end{proof}

\section{Non general special Gushel-Mukai threefold}
\label{section_nongeneral_GM}
\subsection{Singular Hilbert scheme of twisted cubics}
Let $X$ be a special Gushel-Mukai threefold and with branch locus $\mathcal{B}$ containing lines or conics.  

\begin{lemma}
\label{lemma_finitemany_curves}
Let $S$ be a smooth K3 surface, then the Hilbert scheme $F_d(S)$ of degree $d$ smooth rational curves consists of finitely many points for each $d$. 
\end{lemma}

\begin{proof}
Let $E\subset S$ be a smooth rational curve, then $E^2=-2$, the normal bundle of $E\subset S$ $\mathcal{N}_{E|S}\cong\mathcal{O}_E(-2)$, thus $E$ is rigid. Then the set of rational curves is discrete. Since $F_d(S)$ is proper for each $d$, Then $F_d(S)$ only has finitely many points. 
\end{proof}

\begin{lemma}
\label{corollary_ext_idealsheaf2}
Let $C$ be a twisted cubic such that $\pi(C)$ is a conic and the normal bundle $\mathcal{N}_{L|X}\cong\mathcal{O}_L(-2)\oplus\mathcal{O}_L(1)$ for the line component $L$ of $C$. Then $\mathrm{Ext}^2(I_C,I_C)=k$. 
\end{lemma}

\begin{proof}
The argument is very similar to Proposition~\ref{Prop_extgroup}. We compute $\mathrm{Ext}^2(I_C,I_C)$ by spectral sequence with respect to the triangle $$G[-1]\rightarrow\mathcal{E}\rightarrow I_C$$ Then the first page of the spectral sequence is given by 
$$E^{p,q}_1=\begin{cases}
\mathrm{Ext}^3(\mathcal{E},G[-1])=\mathrm{Ext}^2(\mathcal{E},G), p=-1\\
\mathrm{Ext}^2(G,G)\oplus\mathrm{Ext}^2(\mathcal{E},\mathcal{E}), p=0\\
\mathrm{Ext}^2(G,\mathcal{E}), p=1\\
0, p\leq -2, p\geq 2. \end{cases}$$
By the same computation in Proposition ~\ref{Prop_extgroup} $\mathrm{Ext}^2(G,\mathcal{E})=0$ and $\mathrm{Ext}^2(\mathcal{E},G)=0$. But $\mathrm{Ext}^2(G,G)\cong k^2$ by Lemma~\ref{lemma_computing_extG}. 
Then we have a sequence of morphisms: 
$$k=\mathrm{Ext}^1(\mathcal{E},G)=E^{-1,2}_1
\xrightarrow{f} k=\mathrm{Ext}^2(G,G)\oplus\mathrm{Ext}^2(\mathcal{E},\mathcal{E})=E^{0,2}_1\rightarrow E^{1,2}_2=\mathrm{Ext}^3(G,\mathcal{E})=0$$
The map $f$ is differential $d^{-1,2}_1$. Thus it is enough to show $f$ is a zero map. From the triangle $\mathcal{E}\rightarrow I_C\rightarrow G$, we get a long exact sequence
$\ldots\rightarrow\mathrm{Ext}^1(\mathcal{E},G)\xrightarrow{f}\mathrm{Ext}^2(G,G)\rightarrow\mathrm{Ext}^2(I_C,G)\rightarrow 0$. So it is enough to show that $\mathrm{Ext}^2(I_C,G)=k$. Apply $\mathrm{Hom}(I_C,-)$ to the triangle defining $G$ in Proposition~\ref{Prop_computingcone}, we have $\mathrm{Ext}^2(I_C,G)\cong\mathrm{Ext}^2(I_C,\mathcal{O}_L(-2))\cong\mathrm{Ext}^1(\mathcal{O}_L(-1),I_C)$. We may assume that $C$ is union of a line $L$ and a smooth conic $M$. Apply $\mathrm{Hom}_{\mathcal{O}_X}(\mathcal{O}_L(-1),-)$ to the decomposition exact sequence 
$$0\rightarrow I_C\rightarrow I_M\rightarrow\mathcal{O}_L(-1)\rightarrow 0,$$ we get an exact sequence $0\rightarrow k\rightarrow\mathrm{Ext}^1(\mathcal{O}_L(-1),I_C)\rightarrow\mathrm{Ext}^1(\mathcal{O}_L(-1),I_M)$. So it is enough to show $\mathrm{Ext}^1(\mathcal{O}_L(-1),I_M)=0$. Apply $\mathrm{Hom}_{\mathcal{O}_X}(\mathcal{O}_L(-1),-)$ to the standard exact sequence of $M$, we get $\mathrm{Ext}^1(\mathcal{O}_{L}(-1),I_M)\cong \mathrm{Hom}(\mathcal{O}_L(-1),\mathcal{O}_M)=0$. This means that the differential $f=0$. Then $E^{0,2}_2=k$. Further note that $E^{-2,3}_1=E^{2,1}_1=0$, then $E^{-2,3}_2=E^{2,1}_2=0$. Then the sequence of morphisms $$E^{-2,3}_2\rightarrow E^{0,2}_2\rightarrow E^{2,1}_2$$ gives $E^{0,2}_3=k$. But the spectral sequence degenerates at the $E_3$ page, hence $E^{0,2}_{\infty}=k$ and $\mathrm{Ext}^2(I_C,I_C)=k$.

\end{proof}

\begin{theorem}
\label{theorem_singular_hilbertscheme}
Let $X$ be a special Gushel-Mukai threefold with branch locus $\mathcal{B}\subset Y$ containing lines or conics, then the Hilbert scheme of $\mathcal{H}$ of twisted cubics is a singular irreducible threefold. Its singular locus consists of finitely many isolated points given by $\pi$-preimage of twisted cubics contained in $\mathcal{B}$ and finitely many rational curves in $\mathcal{H}_2$ such that each of them consists of reducible twisted cubics with a common line component $L$ whose normal bundle $\mathcal{N}_{L|X}\cong\mathcal{O}_L(-2)\oplus\mathcal{O}_L(1)$. 
\end{theorem}

\begin{proof}
If $C$ is a twisted cubic whose ideal sheaf $I_C\in\mathcal{K}u(X)$. Then $\mathrm{ext}^2(I_C,I_C)\neq 0$ if and only if $\tau(C)=C$. This implies that $\pi(C)$ is a twisted cubic $c\subset\mathcal{B}$. If $\mathcal{B}$ contains lines or conics, then there are at most finitely many twisted cubics contained in $\mathcal{B}$ by Lemma~\ref{lemma_finitemany_curves}.
Then there are only finitely many singular points in $\mathcal{H}_1$.\begin{enumerate}
    \item Assume that $\mathcal{B}$ contains one line $l$. By similar arguments in Lemma~\ref{lemma_connctedness_Hilbertscheme}, the smooth locus $\mathcal{D}$ of $\mathcal{H}$ is connected, hence irreducible. Now we show that the singular locus, which are singular fibers $F_{[L]}\cong\mathbb{P}^1$ with $\mathcal{N}_{L|X}=\mathcal{O}_L(-2)\oplus\mathcal{O}_L(1)$(by Proposition~\ref{proposition_tangentlocus}) are in the same connected component of $\mathcal{D}$. Note that each singular fiber is in $\mathcal{H}_2$, a ruled surface over $\Sigma(X)$(See Lemma~\ref{Lemma_H2_isa_ruledsurface}).
    All fibers but finitely many in $\mathcal{H}_2$ are smooth. Since $\Sigma(X)$ is a connected curve by  Proposition~\ref{proposition_tangentlocus}, the singular fibers $F_{[L]}$ and other smooth fibers are in the same connected component. Next we show that each singular fiber $F_{[L]}$ is in the closure of $\mathcal{D}$. Indeed, all the points but finitely many in $\Sigma(X)$ are smooth. Thus the ruled surface $\mathcal{H}_2$ over $\Sigma(X)$ is reducible. In each irreducible component, the singular fibers are in the closure of the other smooth fibers. Moreover, all the smooth fibers are in smooth locus $\mathcal{D}$, so $\mathcal{H}$ is irreducible.  
    \item If $\mathcal{B}$ contains a line $l$ and and a smooth conic $m$ and they intersect. Then $\mathrm{Sing}(\mathcal{H})=\{F_{[L]}|[L]\in\Sigma(X),\mathcal{N}_{L|X}=\mathcal{O}_L(-2)\oplus\mathcal{O}_L(1)\}\cup \{[I_C]\},$ where $C$ is $\pi$-preimage of $l\cup m\subset\mathcal{B}$. 
    \item Since $\mathcal{B}$ contain finitely many lines and conics by  Lemma~\ref{lemma_finitemany_curves}, each line will produce an additional singular fiber fixed by $\tau$ and there are finitely many of them. Further note that there are finitely many singular points on $\Sigma(X)$ by Proposition~\ref{prop_finitemany_singularfibers} and hence finitely many singular fibers in ruled surface $\mathcal{H}_2\subset\mathcal{H}$.  Moreover, there are at most finitely many reducible twisted cubics in $\mathcal{B}$, thus only finitely many singular points outside the ruled surface. 
\end{enumerate}
The Projectivity of $\mathcal{H}$ follows from the same argument in proving Theorem~\ref{theorem_projectivity_hilbertscheme}. Note that $\mathrm{ext}^1(I_C,I_C)=3$  for all smooth points $[I_C]$. We conclude that $\mathcal{H}$ is an irreducible projective threefold with singular locus described above. 
\end{proof}



\begin{predl}
\label{prop_finitemany_singularfibers}
Let $X$ be a smooth special Gushel-Mukai threefold with branch locus $\mathcal{B}$ on $Y_5$ containing lines or conics. Then $\Sigma(X)$ has at most finitely many singular points and the ruled surface $\mathcal{H}_2$ has at most finitely many singular fibers 
\end{predl}

\begin{proof}
By Proposition~\ref{proposition_lines_Y}, the singular points in $\Sigma(X)$ comes from either $\pi$-preimage of a line $l\in\mathcal{B}$ or $\pi$-preimage of a $(-1,1)$-line which is tangent to $\mathcal{B}$. By Lemma~\ref{lemma_finitemany_curves}, there are only finitely lines $L\subset\Sigma(X)$ which is fixed by $\tau$. On the other hand, if $l'$ is $(-1,1)$-line tangent to $\mathcal{B}$, then  $L+\tau(L)=\pi^{-1}(l')$ and $L, \tau(L)$ are sitting inside of double lines, thus both are singular points in $\Sigma(X)$. We claim that there are finitely many $(-1,1)$-lines which are tangent to $\mathcal{B}$. Indeed, the set of $(-1,1)$-lines on $Y$ is a smooth plane conic by \cite{Iliev94line} and the set of lines which are tangent to $\mathcal{B}$ is a degree $10$ plane curve. Thus the set of $(-1,1)$-lines which are tangent to $\mathcal{B}$ is the intersection of the two curves, which has at most 20 points. Then there are at most 40 singular points in $\Sigma(X)$ coming from the splitting of $(-1,1)$-lines. Then by Corollary~\ref{corollary_ext_idealsheaf2}, we have finitely many singular fibers in ruled surface. 
\end{proof}


\subsection{Singular Bridgeland moduli spaces in the Kuznetsov components}

\begin{lemma}\label{lemma_stability_of_projection}Let $\mathrm{pr}(I_C)\in\mathcal{K}u(X)$ be the projection of ideal sheaf of twisted cubic $C\subset X$. 
\leavevmode\begin{enumerate}
    \item If $\tau(C)=C$, then $\mathrm{ext}^1(\mathrm{pr}(I_C),\mathrm{pr}(I_C))=4$ and $\mathrm{ext}^2(\mathrm{pr}(I_C),\mathrm{pr}(I_C))=1$. 
    \item If $\tau(C)\neq C$, then $\mathrm{ext}^1(\mathrm{pr}(I_C),\mathrm{pr}(I_C))=3$ and $\mathrm{ext}^2(\mathrm{pr}(I_C),\mathrm{pr}(I_C))=0$
 \end{enumerate}
 In particular, $\mathrm{pr}(I_C)$ is $\sigma$-stable for every $\tau$-invariant stability condition $\sigma$.
\end{lemma}

\begin{proof}\leavevmode\begin{enumerate}
    \item If $I_C\in\mathcal{K}u(X)$, then it is clear that $\mathrm{Ext}^2(\mathrm{pr}(I_C),\mathrm{pr}(I_C))\cong\mathrm{Ext}^2(I_C,I_C)=\mathrm{Hom}(I_C,\tau I_C)=\mathrm{Hom}(I_C,I_{\tau C})$. \begin{itemize}
        \item If $\tau(C)=C$, then $\mathrm{Ext}^2(I_C,I_C)=k$
        \item If $\tau(C)\neq C$, then $\mathrm{Ext}^2(I_C,I_C)=0$. 
    \end{itemize} 
    \item If $I_C\not\in\mathcal{K}u(X)$. \begin{itemize}
        \item If $\tau(C)=C$, then $\tau(\mathrm{pr}(I_C))\cong\mathrm{pr}(I_C)$, then $\mathrm{Ext}^2(\mathrm{pr}(I_C),\mathrm{pr}(I_C))\cong\mathrm{Hom}(\mathrm{pr}(I_C),\mathrm{pr}(I_C))=k$ by Corollary~\ref{corollary_computing_ext_projection_object}. Similarly, $\mathrm{Ext}^3(\mathrm{pr}(I_C),\mathrm{pr}(I_C))=0$ by Corollary~\ref{corollary_computing_ext_projection_object}
    \item If $\tau(C)\neq C$, then $\mathrm{Ext}^2(\mathrm{pr}(I_C),\mathrm{pr}(I_C))\cong\mathrm{Hom}(\mathrm{pr}(I_C),\mathrm{pr}(I_{C'}))$ with $C'\neq C$. Then we compute it by spectral sequence in Lemma~\ref{lemma_SS} with respect to the triangles $\mathcal{E}[-1]\rightarrow G\rightarrow\mathrm{pr}(I_C)$ and $\mathcal{E}[-1]\rightarrow G'\rightarrow\mathrm{pr}(I_{C'})$
    The first page of the spectral sequence is given by $$E^{p,q}_1=\begin{cases}
\mathrm{Ext}^1(G',\mathcal{E}[-1])=\mathrm{Hom}(G',\mathcal{E}), p=-1\\
\mathrm{Hom}(G,G')\oplus\mathrm{Hom}(\mathcal{E},\mathcal{E}), p=0\\
\mathrm{Hom}(\mathcal{E},G), p=1\\
0, p\leq -2, p\geq 2. \end{cases}$$
The same computation as in Corollary~\ref{corollary_computing_ext_projection_object} shows that $\mathrm{Hom}(G', \mathcal{E})=0$ and $\mathrm{Hom}(\mathcal{E},G)=0$. Note that $\mathrm{Hom}(G,G')\cong\mathrm{Hom}(I_L,I_{L'})=0$ since $I_L\not\cong I_{L'}$. 
Then we have a sequence of morphism: $$0=\mathrm{Hom}(G,\mathcal{E}[-1])=E^{-1,0}_1\rightarrow k=\mathrm{Hom}(\mathcal{E},\mathcal{E})=E^{0,0}_1\rightarrow k=\mathrm{Hom}(\mathcal{E}[-1],G)=E^{1,0}_1$$
Then by the same argument in Corollary~\ref{corollary_computing_ext_projection_object}, $E^{0,0}_2=0$ and $\mathrm{Hom}(\mathrm{pr}(I_C),\mathrm{pr}(I_{C'})=0$. Similar computation gives $\mathrm{Ext}^3(\mathrm{pr}(I_C),\mathrm{pr}(I_C))=0$ and $\mathrm{Hom}(\mathrm{pr}(I_C),\mathrm{pr}(I_C))=k$. 
    \end{itemize}
   
\end{enumerate}
Then by Lemma~\ref{lemma_ext3_inheart}, Corollary~\ref{lemma_ext3_stable},  Proposition~\ref{prop_stability_ext14} and Lemma~\ref{lemma_inheart_ext14} $\mathrm{pr}(I_C)$ is $\sigma$-stable. 

\end{proof}

\begin{predl}
\label{prop_singular_modulispace}
The projection functor $\mathrm{pr}$ produces an irreducible component $\mathcal{X}$ of dimension three in moduli space $\mathcal{M}_{\sigma}(\mathcal{K}u(X),[I_C])$ of stable objects in $\mathcal{K}u(X)$. The singular locus $\mathrm{Sing}(\mathcal{X})=\mathcal{S}_1\cup \mathcal{S}_2$, where $\mathcal{S}_1$ is a set of finitely many points on $\Sigma(X)$ and $\mathcal{S}_2$ is a set of finitely many points outside $\Sigma(X)$. 
\end{predl}

\begin{proof}
By similary argument in  Lemma~\ref{lemma_compatibility_stableobject} and Proposition~\ref{prop_irreduciblecomponent_modulispace}, the projection functor induces a proper dominant morphism: $p:\mathcal{H}\rightarrow\mathcal{M}$ and produces an irreducible component $\mathcal{X}=p(\mathcal{H})$ of moduli space $\mathcal{M}_{\sigma}(\mathcal{K}u(X),[I_C])$. Since there are only finitely many $\tau$-fixed points on $\Sigma(X)$ and finitly many twisted cubics in the branch locus $\mathcal{B}$. The result follows from Lemma~\ref{lemma_stability_of_projection}
\end{proof}

\section{Bridgeland moduli space on ordinary Gushel-Mukai threefold}
\label{section_OGM}
In this section, we construct an irreducible component of Bridgeland moduli space of $\sigma$-stable object with a $-1$-classe in the Kuznetsov component of smooth ordinary Gushel-Mukai threefold $X'$. 

\begin{predl}\cite[Proposition 2.3]{KPS2018}
\label{prop_GM_category}
The semi-orthogonal decomposition of $D^b(X')$ is given by $$D^b(X')=\langle\mathcal{A}_{X'},\mathcal{O}_{X'},\mathcal{E}^{\vee}_{X'}\rangle$$  \end{predl}

\begin{remark}\label{remark_twoversion_Kuznetsovcomponent}
The subcategory $\mathcal{A}_{X'}\subset D^b(X')$ is also called Kuznetsov component and it is triangulated equivalent to $\mathcal{K}u(X')$. In this section, it is easier to work with $\mathcal{A}_{X'}$. The projection functor is given by $\mathrm{pr}:=\bL_{\mathcal{O}_{X'}}\bL_{\mathcal{E}_{X'}^{\vee}}$. 
\end{remark}

\begin{lemma}
\label{lemma_equivalence_twoKuz}
The two Kuznetsov components are equivalent, i.e. there is an equivalence of categories $\Xi : \mathcal{A}_{X'} \xrightarrow{\sim} \mathcal{K}u(X)$ given by $E \mapsto \bL_{\mathcal{O}_{X'}}(E\otimes\mathcal{O}_X(H))$, in the reverse direction. $F \mapsto (\mathrm{R}_{\mathcal{O}_{X'}} F) \otimes\mathcal{O}_X(-H)$ for the forwards direction.
\end{lemma}

\begin{proof}
We manipulate the semiorthogonal decomposition as follows:
\begin{align*}
     D^b(X) &= \langle\mathcal{K}u(X'),\mathcal{E},\mathcal{O}_{X'}\rangle \\
     &\simeq \langle\mathcal{K}u(X')\otimes\mathcal{O}_{X'}(H),\mathcal{E}^{\vee},\mathcal{O}_{X'}(H)\rangle \\
     &\simeq \langle\mathcal{O}_{X'},\mathcal{K}u(X')\otimes\mathcal{O}_{X'}(H),\mathcal{E}^{\vee}\rangle \\
     &\simeq \langle \bL_{\mathcal{O}_{X'}}(\mathcal{K}u(X')\otimes\mathcal{O}_{X'}(H)),\mathcal{O}_{X'},\mathcal{E}^{\vee}\rangle .
\end{align*}
The desired result follows. 
\end{proof}

\begin{lemma}
\label{lemma_Grothendieckgroup_kuz}
The numerical Grothendieck group $\mathcal{N}(\mathcal{A}_{X'})$ of $\mathcal{A}_{X'}$ is a rank 2 lattice spanned by
\[x:=[I_D]=1-2L, y:=H-4L-\frac{5}{6}P\] where $I_D$ is an ideal sheaf of a conic $D$. The The Euler form is given by
\begin{equation}
\left[               
\begin{array}{cc}   
-1 & -2 \\  
-2 & -5 \\
\end{array}
\right] 
\end{equation}
In particular, the $(-1)$-classes in $\mathcal{N}(A_{X'})$ is $x=1-2L$ and $2x-y=2-H+\frac{5}{6}P$ up to sign.
\end{lemma}

\begin{proof}
It is a simple computation.
\end{proof}

\begin{lemma}\label{unique-extension}
There is a short exact sequence $$0\rightarrow \mathcal{E}\rightarrow\mathcal{Q}^{\vee}\rightarrow I_D\rightarrow 0,$$ where $D\subset X'$ is a conic. 
\end{lemma}

\begin{proof}
Note that $\mathrm{Hom}(\mathcal{E},\mathcal{Q}^{\vee})=k^2$, this is a simple computation by Borel-Weil-Bott theorem. Take a nontrivial map $p\in\mathrm{Hom}(\mathcal{E},\mathcal{Q}^{\vee})$, this map is injective. Indeed, since $\mathcal{Q}^{\vee}$ and $\mathcal{E}$ are both $\mu$-stable and $\mu(\mathcal{Q}^{\vee})=-\frac{1}{3}$ and $\mu(\mathcal{E})=-\frac{1}{2}$, so that if $p$ is non injective, then either its kernel destablizes $\mathcal{E}$ or its image destablizes $\mathcal{Q}^{\vee}$. Note that the image of $\mathcal{Q}^{\vee}\xrightarrow{s}\mathcal{O}_X$ is the ideal sheaf $I_D$ of conic $D\subset X$ or two points. Indeed, the image of $s$ is the ideal sheaf of zero locus of section $s$ of $\mathcal{Q}$. By very similar argument in \cite[Lemma 2.18]{sanna2014rational}, the zero locus of $s$ is either a $\mathbb{P}^2\cap Q$ or $\mathbb{P}^1\cap Q$, where $Q$ is the quadric defining the ordinary GM threefold. But $X$ does not contain any quadric or plane, so the zero locus of $s$ is either a conic or two points. Then since the character of $\mathrm{Coker}(\mathcal{E}\xrightarrow{p}\mathcal{Q}^{\vee})$ is exactly the character of ideal sheaf of some conic, we have the short exact sequence:
$$0\rightarrow\mathcal{E}\rightarrow\mathcal{Q}^{\vee}\rightarrow I_D\rightarrow 0.$$
\end{proof}

\begin{predl}\label{prop_conics_contraction_P1OGM}
Let $D\subset X$ be a conic on ordinary Gushel-Mukai threefold $X$, then $I_D\not\in\mathcal{A}_X$ iff there is a resolution of $I_D$ of the form $$0\rightarrow\cE\rightarrow\cQ^{\vee}\rightarrow I_D\rightarrow 0.$$
In particular, such family of conics is parametrised by $\mathbb{P}^1$. 
\end{predl}

\begin{proof}\leavevmode\begin{enumerate}\item First we show that if there is a short exact sequence $$0\rightarrow\cE\rightarrow\cQ^{\vee}\rightarrow I_D\rightarrow 0,$$ then $I_C\not\in\mathcal{A}_X$. Indeed, apply $\mathrm{Hom}(\cE^{\vee},-)$ to this exact sequence, we get $$\mathrm{Hom}(\cE^{\vee},\cE)\rightarrow\mathrm{Hom}(\cE^{\vee},\cQ^{\vee})\rightarrow\mathrm{Hom}(\cE^{\vee},I_D)\rightarrow\mathrm{Ext}^1(\cE^{\vee},\cE)\rightarrow\mathrm{Ext}^1(\cE^{\vee},\cQ^{\vee})$$
    $$\rightarrow\mathrm{Ext}^1(\cE^{\vee},I_D)\rightarrow\mathrm{Ext}^2(\cE^{\vee},\cE)\rightarrow\mathrm{Ext}^2(\cE^{\vee},\cQ^{\vee})\rightarrow\mathrm{Ext}^2(\cE^{\vee},I_D)\rightarrow$$
$$\mathrm{Ext}^3(\cE^{\vee},\cE)\rightarrow\mathrm{Ext}^3(\cE^{\vee},\cQ^{\vee})\rightarrow\mathrm{Ext}^3(\cE^{\vee},I_D)\rightarrow 0$$
Note that $\mathrm{Hom}(\cE^{\vee},\cQ^{\vee})=0$ by slope stability of $\cE^{\vee}$ and $\cQ^{\vee}$. Further note that $\mathrm{Ext}^1(\cE^{\vee},\cE)\cong\mathrm{Ext}^2(\cE,\cE)=0$, so that $\mathrm{Hom}(\cE^{\vee},I_D)=0$. It also follows that $\mathrm{Ext}^1(\cE^{\vee},I_D)\cong\mathrm{Ext}^1(\cE^{\vee},\cQ^{\vee})$ since $\mathrm{Ext}^2(\cE^{\vee},\cE)\cong\mathrm{Ext}^1(\cE,\cE)=0$. Also note that $\mathrm{Ext}^3(\cE^{\vee},I_D)=0$ since $\mathrm{Ext}^3(\cE^{\vee},\cQ^{\vee})\cong\mathrm{Hom}(\cQ^{\vee},\cE)=0$ again by stability. Apply $\mathrm{Hom}(-,\cE)$ to tautological exact sequence $$0\rightarrow\cE\rightarrow\mathcal{O}_X^{\oplus 5}\rightarrow\cQ\rightarrow 0,$$ we get $\mathrm{Ext}^1(\cE^{\vee},\cQ^{\vee})\cong\mathrm{Ext}^1(\cQ,\cE)\cong\mathrm{Hom}(\cE,\cE)=k$. But $\chi(\cE^{\vee},I_D)=0$, thus $\mathrm{Ext}^2(\cE,I_D)=k$. Then $I_D\not\in\cA_X$. 
\item Now we show that if $I_D\not\in\mathcal{A}_X$, then there is the short exact sequence above. Note that $\mathrm{Hom}(\cE^{\vee},I_D)=\mathrm{Ext}^3(\cE^{\vee},I_D)=0$. Indeed, apply $\mathrm{Hom}(\cE^{\vee},-)$ to the standard exact sequence 
$$0\rightarrow I_D\rightarrow\mathcal{O}_X\rightarrow\mathcal{O}_D\rightarrow 0, $$ we get $\mathrm{Hom}(\cE^{\vee},I_D)\hookrightarrow\mathrm{Hom}(\mathcal{O}_X,\cE)=0$ and $\mathrm{Ext}^3(\cE^{\vee},I_D)=0$ since $\mathrm{Ext}^2(\mathcal{O}_X,\cE|_D)=0$ and $\mathrm{Ext}^3(\mathcal{O}_X,\cE)=0$. But note that $\chi(\cE^{\vee},I_D)=0$, this implies that $I_D\not\in\mathcal{A}_X$ if and only if $\mathrm{Ext}^1(\cE^{\vee},I_D)\neq 0$ since $\mathrm{Hom}^{\bullet}(\mathcal{O}_X,I_D)=0$. If $\mathrm{Ext}^1(\cE^{\vee},I_D)\neq 0$, then $\mathrm{Hom}(\cQ^{\vee},I_D)\neq 0$. Thus we only need to show that if $\mathrm{Hom}(\cQ^{\vee},I_D)=0$ then $\mathrm{Ext}^1(\cE^{\vee},I_D)=0$. Apply $\mathrm{Hom}(-,I_D)$ to the dual of tautological exact sequence $$0\rightarrow\cQ^{\vee}\rightarrow\mathcal{O}_X^5\rightarrow\cE^{\vee}\rightarrow 0,$$ we get $$0\rightarrow\mathrm{Hom}(\cQ^{\vee},I_D)\rightarrow\mathrm{Ext}^1(\cE^{\vee},I_D)\rightarrow\mathrm{Ext}^1(\mathcal{O}_X,I_D)^{\oplus 5}=0.$$
Then $\mathrm{Ext}^1(\mathcal{E}^{\vee},I_D)=0$. Thus if $I_D\not\in\mathcal{A}_X$ then $\mathrm{Hom}(\cQ^{\vee},I_D)\neq 0$. Let $\pi:\cQ^{\vee}\rightarrow I_D$ be a non zero map. Then we claim that $\pi$ is surjective. indeed, its image is the ideal sheaf $I_{D'}$ of the zero locus $D'$ of section $s$ of $\cQ$ containing the conic $D$. But it is known that the zero locus of $s$ is either two points or a conic, then $D'$ must be $D$ and image of $\pi$ is just $I_D$, hence $\pi$ is surjective. Then we have a short exact sequence $$0\rightarrow\mathrm{Ker}\pi\rightarrow\cQ^{\vee}\rightarrow I_D\rightarrow 0.$$
Then by Lemma~\ref{unique-extension}, $\mathrm{Ker}\pi\cong\cE$ and we have this short exact sequence:
$$0\rightarrow\cE\rightarrow\cQ^{\vee}\rightarrow I_D\rightarrow 0.$$
\end{enumerate}
In particular, such conics are parametrised by $\mathbb{P}\mathrm{Hom}(\cE,\cQ^{\vee})\cong\mathbb{P}^1$. Indeed $\mathrm{Hom}(\cE,\cQ^{\vee})\cong k^2$ by application of Borel-Weil-Bott theorem. The proof is complete. 
\end{proof}

\begin{predl}
\label{prop_conic_inKuz}
Let $X'$ be a smooth ordinary Gushel-Mukai threefold and $D\subset X'$ a conic. Let $\mathrm{pr}:D^b(X')\rightarrow\mathcal{A}_{X'}$ be the projection functor. If $I_D\in\cA_{X'}$, then $\mathrm{pr}(I_D)=I_D$ and if $I_D\not\in\cA_{X'}$, then we have the 
following exact triangle:
$$\mathcal{E}[1]\rightarrow\mathrm{pr}(I_D)\rightarrow \mathcal{Q}^{\vee}$$ where $\mathcal{Q}$ is the tautological quotient bundle. 
\end{predl}

\begin{proof}
If $I_D\not\in\cA_{X'}$, then by Proposition~\ref{prop_conics_contraction_P1OGM}, there is a resolution of $I_D$ of the form
$$0\rightarrow\cE\rightarrow\cQ^{\vee}\rightarrow I_D\rightarrow 0.$$
Apply $\mathrm{pr}=\bL_{\mathcal{O}}\bL_{\cE^{\vee}}$ to this short exact sequence, we get the triangle $$\mathrm{pr}(\cE)\rightarrow\mathrm{pr}(\cQ^{\vee})\rightarrow\mathrm{pr}(I_D).$$
Note that $\mathrm{pr}(\cQ^{\vee})=0$. Indeed, $\mathrm{RHom}^\bullet(\cE^{\vee},\cQ^{\vee})\cong k[-1]$. This is because  $\mathrm{Hom}(\cE^{\vee},\cQ^{\vee})=\mathrm{Ext}^3(\cE^{\vee},\cQ^{\vee})=0$ by stability of the vector bundles $\cE^{\vee}$ and $\cQ^{\vee}$ and comparison of their slopes. Also note that $\mathrm{Ext}^1(\cE^{\vee},\cQ^{\vee})\cong\mathrm{Ext}^1(\cQ,\cE)\cong\mathrm{Hom}(\cQ,\cQ)=k$ and $\mathrm{Ext}^2(\cQ,\cE)\cong\mathrm{Ext}^1(\cQ,\cQ)=0$ since $\cQ$ is an exceptional bundle. Then we have the triangle 
$$\cE^{\vee}[-1]\rightarrow\cQ^{\vee}\rightarrow\bL_{\cE^{\vee}}\cQ^{\vee}.$$ Applying $\bL_{\mathcal{O}_X}$ to this triangle we get 
$$\cQ^{\vee}\rightarrow\cQ^{\vee}\rightarrow\mathrm{pr}(\cQ^{\vee}),$$ so $\mathrm{pr}(\cQ^{\vee})=0$. Then $\mathrm{pr}(I_C)\cong\mathrm{pr}(\cE)[1]$. Now we compute the projection $\mathrm{pr}(\cE)=\bL_{\mathcal{O}_X}\bL_{\mathcal{E}^{\vee}}\cE$. We have the triangle
$$\mathrm{RHom}^\bullet(\cE^{\vee},\cE)\otimes\cE^{\vee}\rightarrow\cE\rightarrow\bL_{\cE^{\vee}}\cE.$$ Since $\mathrm{RHom}^\bullet(\cE^{\vee},\cE) \cong k[-3]$, we get $\cE^{\vee}[-3]\rightarrow\cE\rightarrow\bL_{\cE^{\vee}}\cE$. Now apply $\bL_{\mathcal{O}_X}$ to this triangle. We get
$\bL_{\mathcal{O}_X}\cE^{\vee}[-3]\rightarrow \cE\rightarrow\bL_{\mathcal{O}_X}\bL_{\mathcal{E}^{\vee}}\cE=\mathrm{pr}(\cE)$, which is equivalently $$\cQ^{\vee}[-2]\rightarrow\cE\rightarrow\mathrm{pr}(\cE).$$
Therefore we obtain the triangle
$$\cE[1]\rightarrow\mathrm{pr}(\cE)[1]\rightarrow\cQ^{\vee}$$
and the desired result follows. 
\end{proof}

\begin{lemma}\label{lemma_ideal_projection_smooth}
\label{lemma_ideal_projection_smooth} If $D\subset X'$ is a conic such that $I_D\not\in\cA_{X'}$, then \[ \mathrm{RHom}^\bullet(\mathrm{pr}(I_D), \mathrm{pr}(I_D)) = k \oplus k^2[-1] . \]
\end{lemma}

\begin{proof}
Note that 
We have an exact triangle: $Q^{\vee}[-1]\rightarrow\mathcal{E}[1]\rightarrow\mathrm{pr}(I_D)$. We apply the spectral sequence in Lemma~\ref{lemma_SS} to compute $\mathrm{Hom}^2(\mathrm{pr}(I_D),\mathrm{pr}(I_D))$
$$E_1^{p,q}=\begin{cases}
\mathrm{Hom}^3(\mathcal{E}[1],Q^{\vee}[-1])=\mathrm{Ext}^1(Q,\mathcal{E}^{\vee}), p=-1\\
\mathrm{Ext}^2(\mathcal{E},\mathcal{E})\oplus\mathrm{Ext}^2(Q^{\vee},Q^{\vee})=0, p=0\\
\mathrm{Ext}^1(Q^{\vee}[-1],\mathcal{E}[1]), p=1\\
0, p\geq 2, p\leq -2
\end{cases}$$

Note that $\mathrm{Ext}^3(Q^{\vee},\mathcal{E})\cong\mathrm{Hom}(\mathcal{E}^{\vee},Q^{\vee})=0$, since $\mathcal{E}^{\vee}$ and $Q^{\vee}$ are both $\mu$-stable and $\mu(\mathcal{E}^{\vee})=\frac{1}{2}>-\frac{1}{3}=\mu(Q^{\vee})$. $\mathrm{Ext}^3(Q,\mathcal{E}^{\vee})\cong\mathrm{Hom}(\mathcal{E}^{\vee},Q(-H))=0$ since $\mu(\mathcal{E}^{\vee})=\frac{1}{2}>-\frac{2}{3}=\mu(Q(-H))$. On the other hand $\mathrm{Ext}^1(Q,\mathcal{E}^{\vee})\cong H^1(X,Q^{\vee}\otimes\mathcal{E}^{\vee})=0$, indeed one computes $H^1(X,Q^{\vee}\otimes\mathcal{E}^{\vee})$ by taking global section in the Koszul complex of $X\xhookrightarrow{}\mathrm{Gr}(2,5)$ and apply Borel-Weil-Bott over $\mathrm{Gr}(2,5)$. 
Then we have $\mathrm{Ext}^2(\mathrm{pr}(I_D),\mathrm{pr}(I_D))=0$. Similar computations give $\mathrm{Ext}^3(\mathrm{pr}(I_D),\mathrm{pr}(I_D))=0$ and $\mathrm{Hom}(\mathrm{pr}(I_D),\mathrm{pr}(I_D))=k$. As $\mathrm{pr}(I_D)$ is a $(-1)$-class, the desired result follows. 


\end{proof}

\begin{corollary}\label{prop_stability_ideal_conic}
Let $\sigma$ be a $\tau$-invariant stability condition on $\mathcal{A}_{X'}$. Let $D\subset X'$ be a conic. Then $\mathrm{pr}(I_D)$ is $\sigma$-stable.
\end{corollary}
\begin{proof}
\leavevmode\begin{enumerate}
    \item If $I_D\in\mathcal{A}_{X'}$, then $\mathrm{pr}(I_D)=I_D$, then $\mathrm{Hom}(I_D,I_D)=1$ and $\mathrm{Ext}^3(I_D,I_D)=0$ by slope-stability.\begin{itemize}
        \item If $X'$ is general, then by \cite{logachev2012fano} the Fano surface $\mathcal{C}(X')$ is smooth. Then $\mathrm{Ext}^2(I_D,I_D)=0$ and $\mathrm{Ext}^1(I_D,I_D)=k^2$. By almost the same argument in Lemma~\ref{lemma_ext3_inheart}, Corollary~\ref{lemma_ext3_stable}, $I_D$ is in the heart $\mathcal{A}$ of stability condition $\sigma$ up to shift and it is $\sigma$-stable.
        \item If $X'$ is not general. Then $\mathrm{Ext}^2(I_D,I_D)\cong\mathrm{Hom}(I_D,\tau I_D)$. Then by \cite[Proposition 9.3]{JLLZh2021}, $\tau I_D$ is either $I_{D'}$ for some conic $D'\subset X'$ such that $D\bigcup D'=Z(s)$, where $Z(s)$ is the zero locus of section $s\in H^0(\mathcal{E}^{\vee})$ or $\tau I_D\cong\pi$, which is represented by the distinguished triangle in Proposition~\ref{prop_conic_inKuz}. In previous case, $\mathrm{Hom}(I_D,\tau I_D)\cong\mathrm{Hom}(I_D,I_{D'})$. It is either $k$ or $0$, depends on whether $D=D'$, then $\mathrm{ext}^1(I_D,I_D)=2$ or $3$. In the latter case, it is clear that $I_D\not\cong \tau I_D$, then $\mathrm{ext}^2(I_D,I_D)=0$ and $\mathrm{ext}^1(I_D,I_D)=2$. In both cases, by similar argument in Lemma~\ref{lemma_ext3_inheart}, Corollary~\ref{lemma_ext3_stable}, $I_D$ is in the heart $\mathcal{A}$ of stability condition $\sigma$ up to shift and it is $\sigma$-stable.
    \end{itemize}
\item If $I_D\not\in\mathcal{K}u(X')$. Then by Lemma~\ref{lemma_ideal_projection_smooth}, $\mathrm{ext}^1(\mathrm{pr}(I_D),\mathrm{pr}(I_D)=2$. Then by similar reasoning as above, $\mathrm{pr}(I_D)$ is $\sigma$-stable. 
\end{enumerate}
\end{proof}

Next, we state and prove the main result in this section. 

\begin{theorem}
\label{theorem_irreduciblecomponent_conics}
Let $X'$ be a general smooth ordinary Gushel-Mukai threefold. The projection functor $\mathrm{pr}: D^b(X')\rightarrow\mathcal{A}_{X'}$ produces a smooth irreducible component $\mathcal{S}=p(\mathcal{C}(X'))$ of dimension two in moduli space $\mathcal{M}_{\sigma}(\mathcal{A}_{X'},x=-(1-2L))$ of $\sigma$-stable objects in $\mathcal{A}_{X'}$, where $p:\mathcal{C}(X')\rightarrow\mathcal{S}$ is
a blow up at a smooth point in $\mathcal{S}$. It is induced 
by projection functor $\mathrm{pr}$. In particular, $\mathcal{S}$ is isomorphic to the minimal model $\mathcal{C}_m(X')$ of Fano surface $\mathcal{C}(X')$ of conics on $X'$.  If $X'$ is non general, then $\mathcal{S}=p(\mathcal{C}(X'))$ has singularities. 
\end{theorem}

\begin{proof}
By Proposition~\ref{prop_conics_contraction_P1OGM}, it is known that family $\mathbb{L}$ of conics $D\subset X'$ whose ideal sheaf $I_D\not\in\cA_{X'}$ is parametrised by a $\mathbb{P}^1$. It is the fiber of the first quadric fibration over the scheme $\Sigma_1(X')$ in the language of \cite{debarre2015gushel}. In particular, the projection object $\mathrm{pr}(I_D)\in\cA_{X'}$ is given by the exact triangle in Proposition~\ref{prop_conic_inKuz}. The ideal sheaf $I_D\in\cA_{X'}$ for all
the conics $[D]$ in the complement of $\mathbb{L}$ in Fano surface $\mathcal{C}(X)$ of conics. Then $\mathrm{pr}(I_D[1])=I_D[1]\in\cA_{X'}$, it is $\sigma$-stable by Proposition \ref{prop_stability_ideal_conic} and $\mathrm{ext}^1(I_D,I_D)=2, \mathrm{ext}^2(I_D,I_D)=0$ for all $[D]\in\mathcal{C}(X')$. Let $\mathcal{C}$ be the universal family of conics on $X'\times\mathcal{C}(X')$. The functor $\mathrm{pr}$ induces a morphism $p:\mathcal{C}(X')\rightarrow\mathcal{M}_{\sigma}(\mathcal{A}_{X'},-x)$ factoring through one of the irreducible components $\mathcal{S}$ of $\mathcal{M}_{\sigma}(\mathcal{A}_{X'},-x)$. The complement of $\mathbb{L}$ in $\mathcal{C}(X')$ is a dense open subset $\mathcal{U}$ of $\mathcal{C}(X')$ since $\mathcal{C}(X')$ is irreducible. The morphism $p|_{\mathcal{U}}$ is injective and étale, so $p(\mathcal{U})\subset\mathcal{S}$ is also a dense open subset of $\mathcal{S}$. But $\mathcal{C}(X')$ is a projective surface and $p$ is a proper morphism, so $p(\mathcal{C}(X'))=\mathcal{S}$. By Lemma~\ref{lemma_ideal_projection_smooth}, $\mathcal{S}$ is smooth. Then $p$ is a birational dominant proper morphism from  $\mathcal{C}(X')$ to $\mathcal{S}$. In particular, $\mathbb{L}\cong\mathbb{P}^1$ is contracted by $p$ to a smooth point by Lemma~\ref{lemma_ideal_projection_smooth}. Thus $\mathcal{S}$ is smooth surface obtained by blowing down a smooth rational curve on a smooth irreducible projective surface. This implies that $\mathcal{S}$ is also a smooth projective surface. On the other hand, it is known that there is a unique rational curve $L_{\sigma}\cong\mathbb{L}\subset\mathcal{C}(X')$ and it is the unique exceptional curve by \cite[Section 5.1]{debarre2012period}. Thus $\mathcal{S}$ is the mininal model $\mathcal{C}_m(X')$ of Fano surface of conics on $X'$. If $X'$ is non general, then by the argument in 
Corollary~\ref{prop_stability_ideal_conic} and \cite{debarre2012period}, the surface $\mathcal{C}(X')$ is irreducible and may have singularities. Then $\mathcal{S}=p(\mathcal{C}(X'))$ is an irreducible surface with singularities. 
\end{proof}

\begin{remark}
\label{minimalmodel_conics}
In \cite[Theorem 14.5]{JLLZh2021}, it is proved that the Bridgeland moduli space $\mathcal{M}_{\sigma}(\cA_X,-x)\cong p(\mathcal{C}(X))$. 
\end{remark}


    

\section{Kuznetsov components of quartic double solids and special Gushel-Mukai threefolds}
\label{section7}
In this section, we state and prove the main result in the article for smooth special Gushel-Mukai threefolds. We will write  $\mathcal{M}_{\sigma}(\mathcal{K}u(X),t)$ as $\mathcal{M}_{\sigma}(t)$ when there is no ambiguity. 
\begin{theorem}
\label{theorem_main}
    Let $Y$ be a quartic double solid and $X$ be a special Gushel-Mukai threefold. Then $\mathcal{K}u(Y)\not\simeq\mathcal{K}u(X)$.
 \end{theorem}

We outline several steps toward proving this theorem. We prove it by contradiction, that is we assume that there is an exact equivalence $\Phi:\mathcal{K}u(Y)\simeq\mathcal{K}u(X)$. Then \begin{itemize}
    \item Show that $\Phi$ induces a bijective map $h$ between the closed points of $\mathcal{M}_{\sigma}(\mathcal{K}u(Y),w)$ and $\mathcal{M}_{\sigma'}(\mathcal{K}u(X),\phi(w))$, where $\sigma'$ and $\sigma$ are $\tau$-invariant stability conditions in $\mathcal{K}u(X)$ and $\mathcal{K}u(Y)$ respectively. 
    
    \item Show that the rotation functor $\mathrm{R}:=\bL_{\mathcal{O}_Y}(-\otimes\mathcal{O}_Y(H))$ will identify moduli spaces $\mathcal{M}_{\sigma}(\mathcal{K}u(Y),w)$ and $\mathcal{M}_{\sigma}(\mathcal{K}u(Y),w-2v)$. 
    
    \item Show that there is a morphism $\gamma$ (induced by $h$) identifying irreducible component $\mathcal{Y}$ (or $\mathcal{C}$) of Bridgeland moduli space $\mathcal{M}_{\sigma}(\mathcal{K}u(Y),w)$ with that of $\mathcal{M}_{\sigma'}(\mathcal{K}u(X),\phi(w))$ and then observe that this is impossiple. This means that the hypothetical equivalence does not exist. 
\end{itemize}



Recall that rotation functor $\Psi:E\mapsto \bL_{\mathcal{O}_Y}(E\otimes\mathcal{O}_Y(H))$ is 
an auto-equivalence of $\mathcal{K}u(Y)$ of quartic double solid. It is easy to check that the induced linear automorphism on numerical Grothendieck group $\mathcal{N}(\mathcal{K}u(Y))$ maps $w\mapsto 2v-w$ and $v\mapsto v-w$(See \cite[Lemma 3.15]{altavilla2019moduli}. 

\begin{lemma}
\label{lemma_rotation_functor_preserves_stability}
Let $Y$ be a quartic double solid and $\sigma$ the stability condition constructed in \cite[Section 3.2]{pertusi2020some}. Then there are isomorphisms of the moduli spaces: $\mathcal{M}_{\sigma}(\mathcal{K}u(Y),w)\cong\mathcal{M}_{\sigma}(\mathcal{K}u(Y), 2v-w)$.  $\mathcal{M}_{\sigma}(\mathcal{K}u(Y),v)\cong\mathcal{M}_{\sigma}(\mathcal{K}u(Y), v-w)$. Where $\sigma$ is any $\tau$-invariant stability condition on $\mathcal{K}u(Y)$. 
\end{lemma}

\begin{proof}
It follows from the fact that the rotation functor $\Psi:\mathcal{K}u(Y)\simeq\mathcal{K}u(Y)$ preserves the stability condition by \cite[Proposition 5.4]{pertusi2020some}. 
\end{proof}

\begin{theorem}
\label{Theorem_bijective_closedpoints}
Let $E\in\mathcal{K}u(Y)$ be a $\sigma$-stable object of $-2$-class $w$ and $\Phi:\mathcal{K}u(Y)\rightarrow\mathcal{K}u(X)$ the equivalence of Kuznetsov components. Then the closed points of Bridgeland moduli spaces $\mathcal{M}_{\sigma}(w)(k)$ and $\mathcal{M}_{\sigma'}(\phi(w))(k)$ are bijective, where $\sigma$ and $\sigma'$ are $\tau$-invariant stability conditions in $\mathcal{K}u(Y)$ and $\mathcal{K}u(X)$ respectively. 
\end{theorem}

\begin{proof}
Let $E$ be a $\sigma$-stable object of class $[E]=w$ in $\mathcal{K}u(Y)$. As $\sigma$ is $\tau$-invariant, then $\Phi(\sigma)$ is also an $\tau$-invariant stability condition on $\mathcal{K}u(X)$. Since $E$ is $\sigma$-stable and $\chi(E,E)=-2$,  $\mathrm{ext}^1(E,E)=3$ or $4$. Thus $F=\Phi(E)$ is also a $(-2)$-class on $\mathcal{K}u(X)$ with $\mathrm{ext}^1(F,F)=3$ or $4$. Then by Lemma~\ref{lemma_ext3_inheart},Corollary~\ref{lemma_ext3_stable}, Proposition~\ref{prop_stability_ext14} and Lemma~\ref{lemma_inheart_ext14}. $F$ is $\sigma'$-stable for any $\tau$-invariant stability condition on $\mathcal{K}u(X)$. Thus we have shown that the map $E\mapsto\Phi(E)$ sends a point in $\mathcal{M}_{\sigma}(w)$ to a point in $\mathcal{M}_{\sigma'}(\phi(w))$ and the map
$$h: \mathcal{M}_{\sigma}(w) \rightarrow \mathcal{M}_{\sigma'}(\phi(w)) $$ is constructed. 

Conversely, Let $F$ be a $\sigma'$-stable object with class $\phi(w)$ in $\mathcal{K}u(X)$.($\chi(F,F)=-2$). Then up to sign, $\Phi^{-1}(F)$ is an object in $\mathcal{K}u(Y)$ with class $w$ or $2v-w$. By the same argument above, $E:=\Phi^{-1}(F)$ is $\sigma$-stable for every $\tau$-invariant stability condition on $\mathcal{K}u(Y)$. Then by Lemma~\ref{lemma_rotation_functor_preserves_stability}, we may assume that $E:=\Phi^{-1}(F)$ is an object with class $w$. Thus we have shown that the map $F\mapsto\Phi^{-1}(F)=E$ sends a point in 
$\mathcal{M}_{\sigma'}(\phi(w))$ to a point $\mathcal{M}_{\sigma}(w)$, which provides an inverse map 
to $h$. Then we have a bijection on closed points on the two Bridgeland moduli spaces. Note that by Corollary~\ref{lemma_ext3_inheart} and ~\ref{lemma_inheart_ext14}, for any object $E\in\mathcal{M}_{\sigma}(\mathcal{K}u(Y),w), \Phi(E)[n_E]\in\mathcal{A}'$ for some integer $n_E$, where $\mathcal{A}'$ is the heart of stability condition $\sigma'$ on $\mathcal{K}u(X)$. By similar argument in \cite[Section 5.1]{bernardara2012categorical} and \cite[Proposition 3.16]{altavilla2019moduli}, we show that the shift $n_E$ can be chosen uniformly. Let $E,F$ be different $\sigma$-stable objects of class $w$ and assume $\Phi(E)[n_E]$ has phase $\frac{1}{2}$ and $\Phi(F)[n_F]$ has phase $\phi\in\mathbb{R}$. Note that $\mathrm{Hom}(\phi(E),\Phi(F)[1])\neq 0, \mathrm{Hom}(\Phi(F),\Phi(E)[1])\neq 0$ by Euler characteristic calulations. This implies that $-\frac{1}{2}<\phi<\frac{3}{2}$, implies $\phi=\frac{1}{2}$. But this means that $n_E=n_F$. \end{proof}

 \subsection{Construct Morphisms identifying two moduli spaces} 
\label{subsection_family}
 Let $\mathcal{Y}$ and $\mathcal{C}$ be the two irreducible components of $\mathcal{M}_{\sigma}(w)$ \cite[Theorem 1.3]{altavilla2019moduli} and let $\mathcal{F}_{\mathcal{Y}}$ and $\mathcal{F}_{\mathcal{C}}$ be the universal families on $\mathcal{Y}\times X$ and $\mathcal{C}\times X$ respectively. To construct a morphism $\gamma: \mathcal{M}_{\sigma}(\mathcal{K}u(Y),w)\rightarrow\mathcal{M}_{\sigma'}(\mathcal{K}u(X),\phi(w)=s)$, one needs to construct morphisms $\gamma_1:\mathcal{Y}\rightarrow\mathcal{M}_{\sigma'}(s)$ and $\gamma_2:\mathcal{C}\rightarrow\mathcal{M}_{\sigma'}(s)$ respectively. It is sufficient to construct families of objects in $D^b(X)$ parametrised by $\mathcal{Y}$ and $\mathcal{C}$, respectively. If $\Phi:\mathcal{K}u(Y)\simeq\mathcal{K}u(X)$ is a Fourier-Mukai type equivalence, i.e. the composition $$D^b(Y)\xrightarrow{\mathrm{pr}}\mathcal{K}u(Y)\xrightarrow{\Phi}\mathcal{K}u(X)\xrightarrow{e} D^b(X)$$ is a Fourier-Mukai transform(still denoted by $\Phi$). Then $\Phi\cong\mathrm{\Phi}_{\mathcal{J}}$ for some integral kernel on $D^b(X\times Y)$. Following \cite[Section 3.32]{altavilla2019moduli}, we define
$$\mathrm{\Phi}_{\mathcal{J}}\times\mathrm{id}_{\mathcal{Y}}:=\mathrm{\Phi}_{\mathcal{J}\boxtimes\mathcal{O}_{\Delta_{\mathcal{Y}}}}:D^b(Y\times\mathcal{Y})\rightarrow D^b(X\times\mathcal{Y})$$ and $\mathrm{\Phi}_{\mathcal{J}\boxtimes\mathcal{O}_{\Delta_{\mathcal{Y}}}}(\mathcal{\mathcal{F}_{\mathcal{Y}}})$ is a family of objects of $D^b(X)$ parametrised by $\mathcal{Y}$. Similarly, define $$\mathrm{\Phi}_{\mathcal{J}}\times\mathrm{id}_{\mathcal{C}}:=\mathrm{\Phi}_{\mathcal{J}\boxtimes\mathcal{O}_{\Delta_{\mathcal{C}}}}:D^b(Y\times\mathcal{C})\rightarrow D^b(X\times\mathcal{C})$$ and $\mathrm{\Phi}_{\mathcal{J}\boxtimes\mathcal{O}_{\Delta_{\mathcal{C}}}}(\mathcal{\mathcal{F}_{\mathcal{C}}})$ is a family of objects of $D^b(X)$ parametrised by $\mathcal{C}$. The constructions give morphisms: $\gamma_1:\mathcal{Y}\rightarrow\mathcal{M}_{\sigma'}(s)$ and $\gamma_2:\mathcal{C}\rightarrow\mathcal{M}_{\sigma'}(s)$ and by the same argument in  Lemma~\ref{lemma_compatibility_stableobject}, one sees that the image $\gamma_1(p)\in\mathcal{M}_{\sigma'}(s)$ of point  $p\in\mathcal{Y}$ is given by $\sigma'$-stable object $\Phi(E_p)$ and $\gamma_2(p)\in\mathcal{M}_{\sigma'}(s)$ of point $p\in\mathcal{C}$ is given by $\sigma'$-stable object $\Phi(E'_p)$, where $E_p=i_p^*\mathcal{F}_{\mathcal{Y}}, p\in\mathcal{Y}$ and $E_p'=i_p^*\mathcal{F}_{\mathcal{C}}, p\in\mathcal{C}$. This defines a morphism $\gamma:\mathcal{M}_{\sigma}(w)\rightarrow\mathcal{M}_{\sigma'}(s)$ via  $$\gamma(p)=\begin{cases}\gamma_1(p), p\in\mathcal{Y}\\
\gamma_2(p), p\in\mathcal{C}\end{cases}$$. 

If the equivalence $\Phi$ is not of Fourier-Mukai type, then we can construct families of objects $\mathcal{Q}_1$ and $\mathcal{Q}_2$ correspondent to $\mathcal{F}_{\mathcal{Y}}$ and $\mathcal{F}_{\mathcal{C}}$ by hand,  following the same pattern as in \cite[Section 3.32]{altavilla2019moduli} and \cite[Section 5.2]{BMMS} since $\mathcal{Y}$ and $\mathcal{C}$ are both irreducible projective varieties. We omit the details here. 

\subsection{Upshot}
\label{subsection_upshot}
Either by Fourier-Mukai functor or convolution techniques described in \cite[Section 3.32]{altavilla2019moduli} and \cite[Section 5.2]{BMMS}, the morphisms $\gamma_1,\gamma_2$ are projective morphisms from $\mathcal{Y}$ or $\mathcal{C}$ to $\mathcal{M}_{\sigma'}(s)$. Together with $\gamma:\mathcal{M}_{\sigma}(w)\rightarrow\mathcal{M}_{\sigma'}(s)$, they are induced by the conjectural equivalence of categories $\Phi:\mathcal{K}u(Y)\simeq\mathcal{K}u(X)$. Then by Theorem~\ref{Theorem_bijective_closedpoints}, $\gamma$ is a bijective map between the closed points of $\mathcal{M}_{\sigma}(w)$ and $\mathcal{M}_{\sigma'}(s)$. Thus there must be one of morphisms $\gamma_1$ and $\gamma_2$ factoring through $\mathcal{Z}\subset\mathcal{M}_{\sigma'}(s)$(or $\mathcal{X}\subset\mathcal{M}_{\sigma'}(s)$ if $X$ is not general). Note that for each point $p\in\mathcal{M}_{\sigma}(w)$, the differential of $\gamma:T_p\mathcal{M}_{\sigma}(w)\rightarrow T_{\gamma(p)}\mathcal{M}_{\sigma'}(s)$ is an isomorphism since $d_{\gamma_p}$ is given by $$\mathrm{d}\Phi:\mathrm{Ext}^1(E_p,E_p)\rightarrow\mathrm{Ext}^1(\Phi(E_p),\Phi(E_p))$$ and $\Phi$ is equivalence, hence fully faithful. This means that $\gamma$ is etale and $\gamma_1$ and $\gamma_2$ are both projective dominant. \begin{itemize}
    \item If $\gamma_1:\mathcal{Y}\rightarrow\mathcal{M}_{\sigma'}(s)$ factors through $\mathcal{Z}$(or $\mathcal{X}$), then $\mathrm{im}(\gamma_1)=\mathcal{Z}$(or $\mathrm{im}(\gamma_1)=\mathcal{X}$), thus at each point $p\in\mathcal{Y}$, tangent spaces at $p$ and its image $\gamma_1(p)$ will be isomorphic, but this is impossible since dimension of tangent spaces of points in $\mathcal{Y}$ are 4 at a two dimensional locus \cite[Lemma 3.4]{altavilla2019moduli}. It is a K3 surface $\mathcal{R}$ \cite[Theorem 1.3]{altavilla2019moduli}, while dimension of tangent spaces is 3 everywhere at $\mathcal{Z}$  (dimension of tangent spaces is 4 only at finitely many points of $\mathcal{X}$ if $X$ is non general) by Theorem~\ref{theorem_main1}.
    \item If $\gamma_2:\mathcal{C}\rightarrow\mathcal{M}_{\sigma'}(s)$ factors through $\mathcal{Z}$ (or $\mathcal{X}$), it is also impossible by the same reason as above.
\end{itemize}

\subsection{Proof of Theorem~\ref{theorem_main}}
\label{subsection_proof_main_result}
We start with a lemma:
\begin{lemma}
\label{lemma_isometry}
Assume that there is an equivalence $\Phi:\mathcal{K}u(Y)\rightarrow\mathcal{K}u(X)$ and let $\mathcal{N}(\mathcal{K}u(Y))=\langle v=1-L, w=H-L-\frac{2}{3}P\rangle$. Then up to sign, $$\phi\langle w, v\rangle=\begin{cases}
\langle  3s-2t, s-t\rangle\\
\langle  3s-2t, 2s-t\rangle\\
\langle  s,     s-t\rangle\\
\langle  s,     2s-t\rangle
\end{cases},$$
where $\phi:\mathcal{N}(\mathcal{K}u(Y))\rightarrow \mathcal{N}(\mathcal{K}u(X))$ is the isometry induced by $\Phi$. 
\end{lemma}

\begin{proof}
Assume that there is an equivalence $\Phi: \mathcal{K}u(Y)\cong\mathcal{K}u(X)$. Since $\Phi$ preserves the hom-space, it induces the isometry between $\mathcal{N}(\mathcal{K}u(Y)$ and $\mathcal{N}(\mathcal{K}u(X))$ with respect to the Euler form, denoted by $\phi$: $$\chi_{\mathcal{K}u(X)}(\phi(u),\phi(v))=\chi_{\mathcal{K}u(Y)}(u,v)$$ for all $u,v\in \mathcal{N}(\mathcal{K}u(Y))$. Let $K_0(\mathcal{K}u(Y_2))=\langle v=1-L,w=H-L-\frac{2}{3}P\rangle$. The Euler form with respect to the basis vectors is 
$$\chi_{\mathcal{K}u(Y_2)}=\begin{bmatrix}
   -1      & -1  \\
   -1      & -2\\
    \end{bmatrix} $$
    
Then $\langle\phi(w),\phi(w)\rangle_X=\langle w, w\rangle_Y=-2$, Since $\mathcal{N}(\mathcal{K}u(X))$ is generated by $s=1-3L+\frac{1}{2}P=[I_C], t=H-6L+\frac{1}{6}P$. We may assume that $\phi(w)=as+bt$, then $\langle\phi(w),\phi(w)\rangle_X=\langle as+bt, as+bt\rangle=-2$. It is known that the Euler form with repsect to the basis vectors $s$ and $t$ is 
$$\chi_{\mathcal{K}u(X)}=\begin{bmatrix}
   -2      & -3  \\
   -3      & -5\\
    \end{bmatrix} $$ 
Then we have $2a^2+6ab+5b^2=2$. Then the desired result follows. 
\end{proof}

Now, we start to prove Theorem~\ref{theorem_main}.

\begin{proof}[Proof of Theorem~\ref{theorem_main}]
Assume that there is an equivalence $\Phi: \mathcal{K}u(Y)\simeq\mathcal{K}u(X)$, then by Lemma~\ref{lemma_isometry} we have 
$$\phi\langle w, v\rangle=\begin{cases}
\langle 3s-2t, s-t\rangle\\
\langle 3s-2t, 2s-t\rangle\\
\langle s, s-t\rangle\\
\langle s, 2s-t\rangle
\end{cases}$$

If $\phi(w)=s([I_C]=s)$, then by Section~\ref{subsection_upshot}, it is impossibe. Then we must have $\phi(w)=3s-2t$. Then we have two cases: \begin{itemize}
    \item If $\phi(v)=2s-t$, then $\phi\langle w,v\rangle=\langle 3s-2t, 2s-t\rangle$. This means that $\phi(2v-w)=s$. By Lemma~\ref{lemma_rotation_functor_preserves_stability}, $\mathcal{M}_{\sigma}(w)\cong\mathcal{M}_{\sigma}(2v-w)$. It reduces to the previous case, which is again impossible. 
     \item if $\phi(v)=s-t$. Then $\phi(2v-w)=-s$. But $\mathcal{M}_{\sigma'}(s)\cong\mathcal{M}_{\sigma'}(-s)$ and $\mathcal{M}_{\sigma}(w)\cong\mathcal{M}_{\sigma}(2v-w)$. Then it again reduces to the first case, another contradiction arises. 
     \end{itemize} 
Therefore, $\phi$ does not exist, this means that the hypothetical equivalence $\Phi$ does not exist. Thus $\mathcal{K}u(Y)$ and $\mathcal{K}u(X)$ are not equivalent. This completes the proof. 
\end{proof}

\section{Kuznetsov components of quartic double solids and ordinary Gushel-Mukai threefolds}
\label{section_maintheorem_for_OGM}
In this section, we prove the main result for smooth ordinary Gushel-Mukai threefolds. 
\begin{theorem}
\label{Theorem_main_OGM}
 Let $Y$ be a quartic double solid and $X'$ be an  ordinary Gushel-Mukai threefold. Then $\mathcal{K}u(Y)\not\simeq\mathcal{K}u(X')$
\end{theorem}
\begin{theorem}
\label{theorem_bijective_closedpoints_conics}
Let $E\in\mathcal{K}u(Y)$ be a $\sigma$-stable object of $-1$-class $v$ and $\widetilde{\Phi}:\mathcal{K}u(Y_2)\simeq\mathcal{A}_{X'}$ be an equivalence of the Kuznetsov components. Then the closed points of Bridgeland moduli spaces $\mathcal{M}_{\sigma}(v)(k)$ and $\mathcal{M}_{\sigma}(\sigma')(k)$ are bijective, where $\sigma$ and $\sigma'$ are $\tau$-invariant stability conditions in $\mathcal{K}u(Y)$ and $\mathcal{A}_{X'}$ respectively. 
\end{theorem}

\begin{proof}
The argument is almost the same as Theorem~\ref{Theorem_bijective_closedpoints}. The only difference is that here we use the fact the object $E\in\mathcal{K}u(Y)$ or $\mathcal{A}_{X'}$ with $\mathrm{ext}^1(E,E)=2$ or $3$ is $\sigma$-stable(or $\sigma'$-stable) for every $\tau$(or $\tau'$)-invariant stabilty conditions on $\mathcal{K}u(Y)$(or $\mathcal{A}_{X'}$), which can be shown very easily by similar arguments in Lemma~\ref{lemma_ext3_stable}. Also in this case, we make use of the fact that $\mathcal{M}_{\sigma}(\mathcal{K}u(Y),v)\cong\mathcal{M}_{\sigma}(\mathcal{K}u(Y), v-w)$. The details are left to interested reader. 
\end{proof}

We collect some classical facts on Fano surface of lines on quartic double solids and Fano surface of conics on ordinary Gushel-Mukai threefold in the following proposition.
\begin{predl}
\label{prop_lines_Y_2}
Let $Y$ be a smooth quartic double solid with branch locus $W\subset\mathbb{P}^3$. 
\begin{enumerate}
    \item The Hilbert scheme $\Sigma(Y)$ of lines is a smooth irreducible projective surface if $W$ does not contain lines while it is singular if $W$ contains lines and the singular locus of $\Sigma(Y)$ consists of finitely many points correspond to finitely many lines in $W$. \cite[Theorem 2.1]{Tihomirov1981Fanosurface} 
    \item The Bridgeland moduli space of $\sigma$-stable objects $\mathcal{M}_{\sigma}(\mathcal{K}u(Y),v=1-L)$ in $\mathcal{K}u(Y)$ is isomorphic to $\Sigma(Y)$. \cite[Theorem 1.1]{pertusi2020some}
    \item The Hilbert scheme $\Sigma(Y)$ of lines is not isomorphic to the minimal model of Fano surface $\mathcal{C}_m(X')$ of conics on $X'$. \cite[Section 3]{logachev2012fano}. 
\end{enumerate}
\end{predl}

\begin{lemma}
\label{lemma_isometry_-1_class}
Assume that there is an equivalence $\widetilde{\Phi}:\mathcal{K}u(Y_2)\simeq\mathcal{A}_{X'}$ and let $\mathcal{N}(\mathcal{K}u(Y_2))=\langle v=1-L, w=H-L-\frac{2}{3}P\rangle$.  $$\widetilde{\phi}\langle v, w\rangle=\begin{cases}
\langle  x, y-x\rangle\\
\langle  x, 3x-y\rangle\\
\langle  y-2x, y-x\rangle\\
\langle  y-2x, y-3x\rangle
\end{cases}$$ up to sign
where $\widetilde{\phi}:\mathcal{N}(\mathcal{K}u(Y))\rightarrow \mathcal{N}(\mathcal{A}_{X'})$ is the isometry induced by $\widetilde{\Phi}$. 
\end{lemma}

\begin{proof}
Simple calculations, we leave them to reader.
\end{proof}

\subsection{Construct Morphism identifying two moduli spaces}
\label{subsection_construct_morphism_conicslines}
Assume that there is an equivalence $\widetilde{\Phi}:\mathcal{K}u(Y_2)\simeq\mathcal{A}_{X'}$, denote by $S=p(\mathcal{C}(X'))$ an irreducible component of $M_{\sigma'}(\mathcal{A}_{X'}, x=1-2L)$ constructed in Theorem~\ref{theorem_irreduciblecomponent_conics}. Let $\mathcal{I}_{\Sigma}$ be the universal ideal on $\Sigma(Y)\times Y$. Note that $Y$ is a smooth projective variety and $\Sigma(Y)$ is irreducible and projective. Then we can construct a morphism $\gamma: \Sigma(Y)\cong M_{\sigma}(\mathcal{K}u(Y_2),v)\rightarrow M_{\sigma'}(\mathcal{A}_{X'}, \widetilde{\phi}(v))$, either by Fourier-Mukai functor or convolution techniques described in \cite[Section 3.32]{altavilla2019moduli} and \cite[Section 5.2]{BMMS}. By the same argument in Lemma~\ref{lemma_compatibility_stableobject}, the image $\gamma(q)\in M_{\widetilde{\Phi}(\sigma)}(\mathcal{A}_{X'},\widetilde{\phi}(v))$ of point $q\in\Sigma(Y_2)$ is given by $\widetilde{\sigma'}$-stable object $\widetilde{\Phi}(I_L)$(Proposition~\ref{prop_stability_ideal_conic}). \begin{enumerate}
    \item If $\widetilde{\phi}(v)=x$. Then $\gamma$ factor through one of irreducible components of $M_{\sigma'}(\mathcal{A}_{X'},x)$. But since $\mathcal{M}_{\sigma}(\mathcal{K}u(Y),v)$ is irreducible and $\widetilde{\Phi}$ is an equivalence, then $\gamma$ must factor through $S$ by Theorem~\ref{theorem_bijective_closedpoints_conics}. Also note that $\gamma$ is projective dominant morphism. Then we have $\mathcal{S}=p(\mathcal{C}(X'))=\gamma(\Sigma(Y))$.
    \item If $\widetilde{\phi}(v)=y-2x$. By the same argument above, we can construct a projective dominant morphism $\gamma$ from  $\Sigma(Y)\cong M_{\sigma}(\mathcal{K}u(Y_2),v-w)\cong M_{\sigma}(\mathcal{K}u(Y_2),v)$ to  $M_{\widetilde{\Phi}(\sigma)}(\mathcal{A}_{X'},\widetilde{\phi}(v-w)=\pm x)\cong p(\mathcal{C}(X'))$. 
\end{enumerate}

\begin{proof}[Proof of Theorem~\ref{Theorem_main_OGM}]
The argument is very similar to the proof of Theorem~\ref{theorem_main} and discussions in Section~\ref{subsection_upshot}. It almost suffices to show that the Fano surface $\Sigma(Y)$ of lines on $Y$ can not be ismorphic to  $\mathcal{C}_m(X')$ of conics on a smooth ordinary Gushel-Mukai threefold $X'$. If $\tilde{\phi}(v)=x$, then 
\begin{itemize}
    \item If $X'$ is general and $Y$ is general, then the morphism $\gamma$ constructed in Section~\ref{subsection_construct_morphism_conicslines} is a bijection of $\Sigma(Y)$ and $\mathcal{S}=p(\mathcal{C}(X'))=\mathcal{C}_m(X')$ and bijection of their tangent spaces at each point. Then $\Sigma(Y)\cong\mathcal{C}_m(X')$ since both of them are smooth. But this is impossible by Proposition~\ref{prop_lines_Y_2}. 
    \item If $X'$ is general and $Y$ is not general, i.e, $\mathcal{M}_{\sigma}(\mathcal{K}u(Y),v)\cong\Sigma(Y)$ has finitely many singular points and dimension of tangent spaces at those points are 3. $\gamma$ would identify tangent spaces of $\mathcal{C}_m(X')$ and $\Sigma(Y)$ at each point, but this is again impossible since $\mathcal{C}_m(X')$ is smooth.
    \item If $X'$ is not general and $Y$ is general, then $p(\mathcal{C}(X'))$ has singular points, where the dimension of tangent space at $[C]$ would jump. But $\Sigma(Y)$ is smooth, where the dimension of tangent space is a constant. It is again impossible. 
    \item If none of $X'$ and $Y$ is general. As the morphism $\gamma$ is a bijectivity on closed points of $\Sigma(Y)$ and $p(\mathcal{C}(X'))$, the surface $\Sigma(Y)$ is connected(\cite[Remark 2.2.9]{KPS2018}) and the surface $p(\mathcal{C}(X'))$ is irreducible and normal(\cite[Theorem 5.1]{debarre2020double} and \cite{debarre2012period}). This means that $\Sigma(Y)\cong p(\mathcal{C}(X'))$. But this is again impossible by Proposition~\ref{prop_lines_Y_2} (3). 
    \end{itemize}
   
Thus, $\tilde{\phi}(v)=y-2x$. Then there are also two cases: $\tilde{\phi}(w)=y-x$, then $\tilde{\phi}(v-w)=y-2x-y+x=-x$. If $\tilde{\phi}(w)=y-3x$, then $\tilde{\phi}(v-w)=y-2x-y+3x=x$. But $\mathcal{M}_{\sigma}(\mathcal{K}u(Y),v)\cong\mathcal{M}_{\sigma}(\mathcal{K}u(Y), v-w)$ by Lemma~\ref{lemma_rotation_functor_preserves_stability} and $M_{\sigma'}(\mathcal{A}_{X'},\widetilde{\phi}(v-w)=\pm x)\cong p(\mathcal{C}(X'))$. Then it reduces to previous case. Thus $\mathcal{K}u(Y)\not\simeq\mathcal{A}_{X'}$. But  $\mathcal{A}_{X'}\simeq\mathcal{K}u(X')$ by Remark~\ref{remark_twoversion_Kuznetsovcomponent}. The desired result follows. The proof for Theorem~\ref{Theorem_main_OGM} is complete.
\end{proof}

\bibliographystyle{alpha}
{\small{\bibliography{Fano threefold conjecture.bib}}}

\end{document}